\documentclass[a4paper,DIV=classic, 10 pt]{myart}

\usepackage[normalem]{ulem}

\newcommand{\norm}[1]{\left\Vert#1\right\Vert}
\newcommand{\abs}[1]{\left\vert#1\right\vert}
\newcommand{\Set}[1]{\ensuremath{ \left\{ #1 \right\} }}
\newcommand{\set}[1]{\ensuremath{ \{ #1 \} }}

\renewcommand{\mid}{\,|\,}
\newcommand{\Mid}{\:\big | \:}
\newcommand{\mmid}{\;\Big | \;}

\DeclareMathOperator*{\esssup}{ess\,sup}
\DeclareMathOperator*{\essinf}{ess\,inf}

\usepackage{verbatim}

\begin{document}

\title{Portfolio Optimization under Nonlinear Utility}
\tnotetext[t]{We thank Victor Nzengang for  fruitful discussions}

\keyAMSClassification{60H20; 93E20; 91B16; 91G10 }

\author[1,s]{Gregor Heyne}
\author[b,2]{Michael Kupper}
\author[b,3,u]{Ludovic Tangpi}

\address[b]{Universit\"at Konstanz, Universit\"atsstra\ss e 10, 78464 Konstanz, Germany}

\eMail[1]{gh.heyne@gmail.com}
\eMail[2]{kupper@uni-konstanz.de}
\eMail[3]{ludovic.tangpi@uni-konstanz.de}


\myThanks[s]{MATHEON project E.2}
\myThanks[u]{Partially supported by a Ph.D. scholarship of the Berlin Mathematical School}


\date{\today}

\abstract{
	This paper studies the utility maximization problem of an agent with non-trivial endowment, and whose preferences are modeled by the maximal subsolution of a BSDE.
	We prove existence of an optimal trading strategy and relate our existence result to the existence of a maximal subsolution to a controlled decoupled FBSDE.
	Using BSDE duality, we show that the utility maximization problem can be seen as a robust control problem admitting a saddle point if the generator of the BSDE additionally satisfies a specific growth condition.
	We show by convex duality that any saddle point of the robust control problem agrees with a primal and a dual optimizer of the utility maximization problem, and can be characterized in terms of a BSDE solution.
	}
\keyWords{Subsolutions of BSDEs; submartingale; Convex duality; Utility maximization.}

\maketitle

\section{Introduction}
The theory of expected utility is of fundamental importance in finance and economy.
Introduced by \citet{Bernouilli}, the expected utility represents the level of satisfaction of a financial agent acting in a risky environment.
In their seminal \emph{Theory of Games and Economic Behavior}, \citet{vonN-Mor} have provided an axiomatic foundation for decision making under risk 
based on \emph{rational} principles; and by the work of \citet{Savage01}, under these axioms preferences can be modeled as expected utility. 
However, the axioms of von Neumann and Morgenstern have been much criticized by empirical studies such as the 
well known Allais paradox and Ellsberg paradox.
On the other hand, expected utility does not capture uncertainty in the underlying probabilistic model. 
Many alternative approaches have been suggested to model decision beyond expected utility.
A few examples include the concepts of capacity and weighted expected utility and, more recently, the recursive utility and the $g$-expectation.
Following this trend, we consider in the present work the portfolio optimization of an agent whose utility is modeled by the maximal subsolution of a nonlinear backward stochastic differential equation (BSDE). Our principal aim is to give sufficient, and necessary conditions of existence of an optimal portfolio in this framework.

Amongst the numerous attempts that have been made in the literature to study portfolio optimization under nonlinear utility, the work of \citet{Kar-Peng-Que01} on the optimization of stochastic differential utility is especially related to ours.
This class of utility functions were introduced by \citet{Duf-Eps} and can be seen as solutions of nonlinear BSDEs.
In a non-Markovian model, \citet{Kar-Peng-Que01} prove existence of an optimal trading strategy and an optimal consumption policy and characterize the optimal wealth process and the utility as solutions of a forward-backward system.
They assume that the generator of the BSDE satisfies a linear growth condition and is continuously differentiable in all variables, so that the utility itself is differentiable and satisfies a comparison principle.
Their results are based on BSDE theory:
Notably, the existence result follows from a penalization method which consists in approaching the problem by a sequence of penalized problems that can be solved, and then obtain the solution by compactness arguments.


The first contribution of the present paper is to give conditions that guarantee the existence of an optimal trading strategy for an agent whose utility is given as the maximal subsolution of a BSDE. 
We consider a non-Markovian incomplete market model where the agent also has a random terminal endowment, and the utility is modeled by a BSDE whose generator is convex, positive, lower semicontinuous and satisfies a normalization condition.
The technique of the proof, inspired from \citet{DHK1101}, rests on localization arguments and compactness principles.
We do not impose any artificial integrability with respect to the historical probability measure on the wealth process.
Hence, the central idea here is to introduce an auxiliary function under which the image of the terminal conditions will be uniformly integrable in the set of subsolutions. 
To this end, we require the drift to satisfy a suitable integrability condition. 
This uniform integrability allows for the construction of a localizing sequence of stopping times that makes the value processes of the admissible subsolutions local submartingales. 
Thus, compactness results for sequences of martingales, see \citet{Del-Sch96}, and sequences of increasing finite variation processes can be used locally in time, and the candidate solutions obtained by almost sure convergence of the sequence of stopping times to the time horizon.
The verification follows from Fatou's lemma and join convexity of the generator.

Analogous to the case of recursive utility studied by \citet{Kar-Peng-Que01}, there is an intrinsic link between the optimal wealth process and its utility:
They can be seen as a maximal subsolution of a forward-backward system.

We also address the question of characterization of an optimal trading strategy.
In the optimal stochastic control literature, such a characterization is usually a consequence of the stochastic maximum principle.
One introduces a perturbation of the optimal diffusion and, by It\^o's formula, obtains at the limit a variational equation which enables to characterize the optimal control, see for instance \citet{PengFBSDE} and \citet{Hor-etal}.
This characterization follows from the fact that the expectation operator is linear, a property that our operator does not enjoy.
The idea to get around this difficulty is to use the duality of BSDEs studied by \citet{tarpodual}, and transform the original control problem into a robust control problem with non-zero penalty term.
Provided that the robust control problem admits a saddle point, the problem can be linearized and the maximum principle applies.
The proof of the existence of a saddle point follows from the existence of an optimal trading strategy and a weak compactness argument introduced by \citet{Delbaen11} which is achieved under a growth condition on the generator of the BSDE.

The theory of BSDE duality fits quite well to our setting.
It shows for instance that our maximization problem is nothing but the maximization of recursive utilities under model uncertainty.
And because our generator depends on the value process, the uncertainty here also encompasses the uncertainty about the time value of money, see \citet{Kar-Rav} and \citet{tarpodual}.
It also enables us to write and solve the dual problem and characterize its solution in terms of solutions of a BSDE, and shows that the dual optimizer is, in fact, the optimal probabilistic model.

Before presenting the structure of our work, let us give further references of related works. 
Using a convex duality approach, the expected utility maximization problem was studied by \citet{Kra-Sch}.
They give precise conditions on the utility function for a solution to exist.
\citet{Cvi-Sch-Wang} have extended their results to the non-zero random endowment case.
A fully probabilistic method to study the problem has been investigated by \citet{Hu-Imk-Mul}. 
For exponential utility, they characterize the value function and the optimal strategy of the problem with random endowment as the solution of a quadratic BSDE.
Beyond the exponential utility case, \citet{Hor-etal} show that the problem can be solved via forward backward systems.
Robust expected utilities have been considered by \citet{Bor-Mat-Sch} and \citet{Fai-Mat-Mnif}.
The latter authors consider a problem with non-zero penalty term and prove existence of and optimal model.
\citet{Oks-Sul13} show that the robust control problem can be treated as a stochastic differential game, a consideration that is also implicitly used in the present paper. 

The next section of the paper is dedicated to the setting of the probabilistic framework of our study and introduces the market model.
Section \ref{sec:existence} studies the primal problem: 
We prove existence of an optimal strategy and stability of the utility operator.
The third section deals with the dual problem. 
Notably, we prove existence of a dual optimizer and characterize the dual and primal optimizers by means of BSDE solutions.
In the last section, we draw the link between duality of BSDEs and the general theory of convex duality.
We gather in an appendix some proofs that are classical in the theory of convex BSDEs but still need to be adapted to our setting for completeness.
\section{Setup and Market Model}
\label{sec:setting}
Let $T\in (0,\infty)$ be a fixed time horizon, and let $(\Omega, \mathcal{F}, (\mathcal{F}_t)_{t\in[0,T]}, P)$ be a filtered probability space. 
The filtration $(\mathcal{F}_t)$ is generated by a $d$-dimensional Brownian motion $W$ and satisfies the usual assumptions of completeness and right-continuity, with $\mathcal{F}_T = \mathcal{F}$. 
Statements concerning random variables or stochastic processes are understood in the $P$-almost sure or the $P\otimes dt$-almost sure sense, respectively. 
Indistinguishable processes are identified. 
When we make a statement without any precision regarding the probability measure, then we are referring to the probability measure $P$. 
Thus, by ``$M$ is a martingale'' we mean ``$M$ is a $P$-martingale''.

We write $L^0$ for the space of $\mathcal{F}$-measurable random variables endowed with the topology of convergence in probability with respect to the measure $P$.
By $\mathcal{S}:=\mathcal{S}(\mathbb{R})$ we denote the set of adapted processes with values in $\mathbb{R}$ which are c\`adl\`ag. 
For $p\in [1, \infty]$, the space $L^p(\Omega, \mathcal{F},P)$ is denoted by $L^p$ and for a different measure $Q$ we write $L^p(Q)$ for $L^p(\Omega, \mathcal{F}, Q)$. 
The space $L^p_+$ is the space of positive random variables belonging to $L^p$.  
We further denote by $\mathcal{L}^p:= \mathcal{L}^p(P)$ the set of predictable processes $Z$ with values in $\mathbb{R}^{1\times d}$, endowed with the norm $\norm{ Z }_{\mathcal{L}^p}:= E_P[(\int_0^T\norm{ Z_s }^2\,ds)^{p/2}]^{1/p}$. 
From \cite{Pro}, for every $Z \in \mathcal{L}^p$ the process $(\int_0^tZ_s\,dW_s)_{t\in[0,T]}$ is well defined and by means of Burkholder-Davis-Gundy's inequality, it is a continuous martingale. 
By $\mathcal{L}$ we denote the set of predictable processes valued in $\mathbb{R}^{1\times d}$ such that there exits a localizing sequence of stopping times $(\tau^n)$ with $Z1_{[0,\tau^n]} \in \mathcal{L}^1$, for all $n\in \mathbb{N}$. 
For $Z \in \mathcal{L}$, the stochastic process $(\int_0^tZ_u\,dW_u)_{t\in[0,T]}$ is a well defined continuous local martingale. Furthermore, for adequate integrands $a$ and $Z$ we write $\int a\,ds$ and $\int Z\,dW$ for $(\int_0^ta_s\,ds)_{t\in[0,T]}$ and $(\int_0^tZ_u\,dW_u)_{t\in[0,T]}$, respectively. 
The running maximum of a process $X$ is denoted by $X^*_t = \sup_{s\in[0,t]}\vert X_s\vert$. 
Given a sequence $(x_n)$ in some convex set, a sequence $(\tilde{x}_n)$ is said to be in the asymptotic convex hull of $(x_n)$ if $\tilde{x}_n\in conv\{x^n, x^{n+1},\dots\}$ for all $n$.

In the financial market, there are available for trading $n$ stocks, $n\leq d$, with price dynamics
\begin{equation*}
	dS^i_t = S_t^i(\mu_t^i\,dt + \sigma^i_t\,dW_t), \quad i = 1,\dots,n,
\end{equation*}
such that $\mu^i$ and $\sigma^i$ are predictable processes valued in $\mathbb{R}$ and $\mathbb{R}^d$, respectively. 
Let us denote by $\sigma$ the $n\times d$ matrix with row vectors $\sigma^i$, the matrix\footnote{$\sigma'$ is the transpose of $\sigma$.} $\sigma\sigma'$ is assumed to be of full rank, so that the market price of risk $\theta$ takes the form $\theta_t = \sigma_t'(\sigma_t\sigma_t')^{-1}\mu_t$, $t\in [0,T]$. 
For the rest of the paper, we make the following standing assumption concerning $\theta$:
\begin{itemize}
	\item There exist constants $p>1$ and $C_\theta> 0$ such that for all stopping times $0\le \tau \leq T$, one has 
	\begin{equation}\label{assump}
		E\left[\left( \mathcal{E}\left(\int\theta \,dW\right)_{\tau} /\mathcal{E}\left(\int\theta \,dW\right)_T \right)^{\frac{1}{p-1}} \mid \mathcal{F}_{\tau} \right] \leq C_\theta, \tag{A}
	\end{equation}
\end{itemize}
where $\mathcal{E}(\int \theta\,dW)$ denotes the stochastic exponential of $\int \theta \,dW$.
This is the so-called Muckenhoupt $A_p$ condition. 
Under this assumption, by \cite[Theorem 2.4]{Kaz}, $\int\theta\,dW$ is a BMO martingale, and therefore $\frac{dQ}{dP} = \mathcal{E}(-\int\theta \,dW)_T$ defines a probability measure $Q$ equivalent to $P$. 
This type of drift conditions are well-known, especially in the context of expected utility maximization, see for instance \citet{Delbaen2}.
Let $x> 0$ be a fixed initial capital.
A trading strategy is a predictable $d$-dimensional process $\pi$ such that $\pi\sigma \in {\cal L}(Q)$ and $X^\pi\ge 0$, 
where the wealth process $X^{\pi}$ is given by
\begin{equation}
\label{eq:wealth}
	X^{\pi}_t = x + \int_0^t\pi_s\sigma_s(\theta_s\,ds + dW_s), \quad t\in [0,T].
\end{equation}
We denote by $\Pi$ the set of trading strategies. 
For every $\pi \in \Pi$, $X^{\pi}$ is a positive $Q$-local martingale and thus a $Q$-supermartingale.
In particular, the market is free of arbitrage opportunities.
The principal objective of this paper is to study the utility maximization from the terminal wealth of an agent who has a non-trivial endowment $\xi$ and whose utility is modeled by a BSDE.

The generator we consider for the BSDEs is a jointly measurable function $g :\Omega\times [0,T] \times \mathbb{R}_+ \times \mathbb{R}^{1\times d} \to \mathbb{R} \cup \{ +\infty\}$, where $\Omega \times [0,T]$ is endowed with the predictable $\sigma$-algebra. 
Furthermore, a generator $g$ is said to be
\begin{enumerate}[label=(\textsc{Conv}),leftmargin=40pt]
	\item convex, if $(y,z)\mapsto g(y,z)$ is convex,\label{jconv}
\end{enumerate}
\begin{enumerate}[label=(\textsc{Lsc}),leftmargin=40pt]
	\item lower semicontinuous, if $(y,z)\mapsto g(y,z)$ is lower semicontinuous,\label{lsc}
\end{enumerate}
\begin{enumerate}[label=(\textsc{Nor}),leftmargin=40pt]
	\item normalized, if $g(y,0)=0$ for all $y \in \mathbb{R}_+$\label{nor},
\end{enumerate}
\begin{enumerate}[label=(\textsc{Pos}),leftmargin=40pt]
	\item positive, if $g\geq 0$.\label{pos}
\end{enumerate}
Given a random variable $H \in L^0$, a subsolution of the BSDE with generator $g$ and terminal condition $H$ is a pair $(Y, Z)$ of processes satisfying
\begin{equation}
\label{eq:subsol}
Y_s + \int_s^tg_u(Y_u, Z_u)\,du - \int_s^tZ_u\,dW_u \leq Y_t;\quad Y_T \leq H,
\end{equation}
for all $0\leq s\leq t\leq T$. 
Let $u:\mathbb{R}_+ \to \mathbb{R}$ be a continuous concave, strictly increasing function such that there exists $C>0$, $\abs{u(x)}^{p^2} \le C(1+ \abs{x})$ for every $x>0$, with $p$ introduced in the condition (A) and such that $L\mapsto u^{-1}(E[u(L)])$ is concave on $\set{L \in L^0_+: E[u(L)]< +\infty}$. 
Examples of such a function include $u(x) = x^r$ with $ rp^2<1$, and $u(x) = -\exp(-rx)$ with $r>0$, see \cite[Section 3]{Dra-Kup}.

A value process $Y \in {\cal S}_+$ is said to be admissible if the process
 $u(Y)$ is a submartingale.
We consider the operator
\begin{equation*}
	\mathcal{E}^g_0(H) := \sup\Set{ Y_0: (Y, Z) \in \mathcal{A}^u(H, g) }
\end{equation*}
with
\begin{equation*}
	\mathcal{A}^u(H, g) := \Set{ (Y,Z)\in \mathcal{S}\times \mathcal{L}:Y \text{ admissible and \eqref{eq:subsol} holds} },
\end{equation*}
the set of admissible subsolutions with respect to $u$. 
The reader will notice that the operator $\mathcal{E}^g_0(\cdot)$ depends	 on $u$. 
Similar to \cite{DHK1101} the operator $\mathcal{E}^g_0(\cdot)$ is a nonlinear utility function. 
In particular, it is monotone, concave but not necessarily cash-additive.
We study the investment problem
\begin{equation}
\label{eq:optprob}
	V(x) := \sup_{\pi \in \Pi}\mathcal{E}^g_0(\xi + X^{\pi}_T).
\end{equation}
More precisely, we would like to give conditions of existence of a pair $(\bar{Y}, \bar{Z})$ along with a trading strategy $\bar{\pi} \in \Pi$ such that $(\bar{Y}, \bar{Z}) \in \mathcal{A}^u(\xi+X^{\bar{\pi}}_T,g)$ and  for any other trading strategy $\pi \in \Pi$ one has
\begin{equation*}
	\bar{Y}_0 = V(x) = \mathcal{E}^g_0(\xi + X^{\bar{\pi}}_T) \geq \mathcal{E}^g_0(\xi + X^{\pi}_T).
\end{equation*}
Henceforth, the function $V$ will be referred to as the value function of the optimization problem \eqref{eq:optprob}, and the triple $(\bar{X}, \bar{Y}, \bar{Z}) $ with $\bar{X} = X^{\bar{\pi}}$, a maximal subsolution.

\begin{example}
\label{example1}
\begin{enumerate}[label= \textbf{\arabic*.},fullwidth]
\item Certainty equivalent:
	Let $X$ be an $\mathcal{F}_T$-measurable random variable such that $u(X)$ is integrable.
	The certainty equivalent $C_t(X)$ of $X$ is defined as $C_t(X) := u^{-1}\left( E[u(X)\mid \mathcal{F}_t] \right)$, $t\in [0,T]$.
	Consider the utility maximization problem
	\begin{equation}
	\label{eq:prob-certainty-equiv}
		V(x) = \sup_{\pi \in \Pi} C_0(X^\pi_T + \xi).
	\end{equation} 
	The martingale representation theorem yields a process $N \in {\cal L}^1$ such that 
	\begin{equation*}
		E[u(X)\mid \mathcal{F}_t] = E[u(X)] + \int_0^tN_u\,dW_u,\quad \text{for all $t\in [0,T]$}.
	\end{equation*}
	Applying It\^o's formula to $Y_t = u^{-1}\left( E[u(X)\mid \mathcal{F}_t] \right)$, we have 
	\begin{equation*}
		dY_t = \frac{1}{u^\prime(Y_t)}N_t\,dW_t - \frac{1}{2}\frac{u^{\prime \prime}(Y_t)}{(u^\prime(Y_t))^3}\abs{N_t}^2\,dt.
	\end{equation*}
	Hence, putting $Z_t = \frac{1}{u'(Y_t)}N_t$, the pair $(Y,Z)$ solves the BSDE
	\begin{equation}
	\label{eq:u_concave}
		Y_t = X + \frac{1}{2}\int_t^T\frac{u''(Y_u)}{u'(Y_u)}\abs{Z_u}^2\,du - \int_t^TZ_u\,dW_u.
	\end{equation}
	For $u(x) = x^r$, $r \in (0,1)$, the generator of the BSDE \eqref{eq:u_concave} is given by $g(y,z) = \frac{1}{2}(r -1)\abs{z}^2/y$ and satisfies the conditions \ref{jconv}, \ref{lsc}, \ref{nor} and \ref{pos} on $(0,+ \infty)\times\mathbb{R}^d$.
	By definition, we have $\mathcal{E}^g_0(X) \ge C_0(X) $.
	In addition, the admissibility condition implies $u(\mathcal{E}^g_0(X)) \le E\left[ u(\mathcal{E}^g_T(X)) \right] $.
	Therefore, $\mathcal{E}^g_0(X) \le C_0(X) $.
	Thus, the utility maximization problem \eqref{eq:prob-certainty-equiv} can be rewritten as $V(x) = \sup_{\pi \in \Pi} \mathcal{E}^g_0(X^\pi_T + \xi) $.
	\item $g$-expectation: Let $u$ be a utility function and $g$ a function defined on $\mathbb{R}\times \mathbb{R}^d$ and satisfying \ref{lsc}, \ref{nor} and \ref{pos} such that for every $\pi \in \Pi$ the BSDE with terminal condition $u(X^\pi_T+\xi)$ and generator $g$ has a unique solution $(Y^\pi, Z^\pi) \in {\cal S}\times {\cal L}^2$.
	Denote by ${\cal E}_g[u(X^\pi_T + \xi)\mid {\cal F}_t] := Y^\pi_t$ the $g$-expectation of $u(X^\pi_T + \xi)$.
	The operator ${\cal E}_g[\cdot]$ is a nonlinear expectation which coincides with the classical expectation $E_P[\cdot]$ when $g = 0$.
	Consider the utility maximization problem
	\begin{equation*}
		V(x) := \sup_{\pi \in \Pi}u^{-1}\left( \mathcal{E}_g[u(\xi + X^{\pi}_T)\mid {\cal F}_0] \right).
	\end{equation*}
	We further assume $u$ to be twice continuously differentiable and that $u'$ is bounded away from zero.
	For every $\pi \in \Pi$, we have 
	\begin{equation}
	\label{eq:g_expect}
		Y^\pi_t =  u(X^\pi_T + \xi)+\int_t^T g(Y^\pi_u, Z^\pi_u)\,du -\int_t^T Z^\pi_u\,dW_u.
	\end{equation}
	Applying It\^o's formula to $\hat{Y}^\pi_t:= u^{-1}(Y^\pi_t)$, we obtain
	\begin{equation}
	\label{eq:gtilde_expect}
		d\hat{Y}^\pi_t = -\left\{ \frac{1}{u'(\hat{Y}^\pi_t)} g(u(\hat{Y}^\pi_t), \hat{Z}^\pi_tu'(\hat{Y}^\pi_t)) - \frac{1}{2}\frac{u''(\hat{Y}^\pi_t)}{u'(\hat{Y}^\pi_t)}\abs{\hat{Z}^\pi_t}^2  \right\}dt + \hat{Z}^\pi_t\,dW_t,
	\end{equation}
	with $\hat{Z}^\pi_t = Z^\pi_t/u'(Y^\pi_t)$ and $\hat{Y}^\pi_T = X^\pi_T + \xi$.
	For $u(x) = -\exp(-rx)$, $r > 0$ and $g(y,z)=\abs{z}$, the generator of the above BSDE takes the form $\hat{g}(y,z) = \abs{z} + \frac{1}{2}(1-r)r^2\abs{z}^2$ 
	and it satisfies the properties \ref{jconv}, \ref{lsc}, \ref{nor} and \ref{pos}.
	Since $g$ is positive, $Y^\pi$ is a submartingale and we have $\mathcal{E}^{\hat{g}}_0(X^\pi_T + \xi) \ge \hat{Y}^\pi_0 = u^{-1}(Y^\pi_0) = u^{-1}({\cal E}_g(u(X^\pi_T + \xi)\mid {\cal F}_0)) $.
	In addition, the admissibility condition implies $u(\mathcal{E}^{\hat{g}}_0(X^\pi_T + \xi)) \le E[ u(\mathcal{E}^{\hat{g}}_T(X^\pi_T+ \xi)) ] \le E[u(X^\pi_T+ \xi)] $ by monotonicity of $u$.
	Since $g$ is positive, taking expectation of both sides of \eqref{eq:g_expect} yields ${\cal E}_g(u(X^\pi_T + \xi)\mid {\cal F}_0) \ge E[u(X^\pi_T + \xi)] $.
	Therefore, $\mathcal{E}^{\hat{g}}_0(X^\pi_T + \xi) \le u^{-1}({\cal E}_g(u(X^\pi_T + \xi)\mid {\cal F}_0))  $.
	Thus, the utility maximization problem \eqref{eq:prob-certainty-equiv} can be rewritten as $V(x) = \sup_{\pi \in \Pi} \mathcal{E}^{\hat{g}}_0(X^\pi_T + \xi) $.
\end{enumerate}

\end{example}
\section{Maximal Subsolutions}
\label{sec:existence}

 \subsection{Existence Results}
In this section we give sufficient conditions of existence of an optimal trading strategy to Problem \eqref{eq:optprob}. 
In order to simplify the presentation, let us introduce the set
\begin{equation*}
	\mathcal{A}(x) := \Set{ (X, Y, Z): X \text{ satisfies \eqref{eq:wealth} for some $\pi \in \Pi$ and } (Y, Z) \in \mathcal{A}^u(\xi + X_T,g) }.
\end{equation*}
The function $V(x)$ can be written as $$V(x) = \sup\{Y_0: (X, Y, Z) \in {\cal A}(x)\}.$$
If $g$ satisfies (\textsc{Nor}) and $\xi \geq 0$, the set $\mathcal{A}(x)$ is nonempty, and contains an element with positive value process. 
The triplet $(X^0, Y^0, Z^0)$, with $Z^0 = 0$, $Y^0=X^0 = x$ and  with associated trading strategy $\pi = 0$ is an element of $\mathcal{A}(x)$. Indeed, the pair $(Y^0, Z^0)$ satisfies \eqref{eq:subsol}, and we have $Y^0_T = x \leq x +\xi=  \xi + X^0_T$. 
Moreover, for all $(X, Y, Z) \in \mathcal{A}(x)$ the c\`adl\`ag process $Y$ can jump only up, since by taking the limit as $s$ tends to $t-$ in Equation \eqref{eq:subsol} we have $Y_t\geq Y_{t-}$, for all $t\in [0,T]$. 
Before stating our existence result, let us prove the following lemmas.
\begin{lemma}
\label{estimates_u}
	Assume $\xi \in L^1_+(\Omega, {\cal F}_T, Q)$. 
	Then there exists a constant $C \geq 0$ such that for all $(X, Y, Z)\in \mathcal{A}(x)$ with $Y\ge 0$, we have
	\begin{equation*}
		E[\abs{u(\xi + X_T)}^{p}] \leq C\quad \text{and}  \quad		u(Y_t) \leq E[u(\xi + X_T)\mid \mathcal{F}_t]\quad t\in [0,T].
	\end{equation*}
\end{lemma}
\begin{proof}
	Let $(X, Y, Z)$ be in $\mathcal{A}(x)$, and $q$ the H\"older conjugate of $p$. 
	We first prove the $L^{p}$ boundedness of $u(\xi + X_T)$.
	Using H\"older's inequality, we estimate as follows:
	\begin{align}
		\nonumber E[\abs{u(\xi + X_T)}^{p}] &= E_Q\left[\frac{1}{\mathcal{E}(\int\theta\,dW)_T}\abs{u(\xi + X_T)}^{p}\right]\\
		\nonumber                  &\leq E_Q\left[\left( \frac{1}{\mathcal{E}(\int\theta\,dW)_T} \right)^{q}\right]^{\frac{1}{q}}E_Q[\abs{u(\xi + X_T)}^{p^2}]^{\frac{1}{p}}.
	\end{align}
	Since there exists a positive constant $C$ such that 
	\begin{equation*}
		\abs{u(\xi + X_T)}^{p^2} \leq C(1 + \xi + X_T),
	\end{equation*}
	we have
	\begin{equation*}
		E[\abs{u(\xi + X_T)}^{p}] \leq C^{1/p}E\left[\mathcal{E}\left( \int\theta\,dW \right)_T\left( \frac{1}{\mathcal{E}(\int\theta\,dW)_T} \right)^{q} \right]^{\frac{1}{q}}E_Q[1+ \xi + X_T]^{\frac{1}{p}}.
	\end{equation*}
	Thus, since $q - 1 = \frac{1}{p - 1}$, it follows from the Muckenhoupt $A_{p}$ condition and the $Q$-supermartingale property of $X$, that
	\begin{equation*}
		E\left[ \abs{u(\xi + X_T)}^{p} \right] \leq C^{1/p}C_\theta^{1/q}(1+ E_Q[\xi] + x)^{\frac{1}{p}},
	\end{equation*}
	hence the first estimate.

	For the second estimate, first notice that $u(\xi + X_T)$ is integrable, and since $u$ is increasing and $(Y, Z)$ satisfies Equation \eqref{eq:subsol}, we have $u(Y_T) \leq u(\xi + X_T)$. 
	Since the value process $Y$ is admissible, we have $u(Y_t) \leq E[u(Y_T)\mid \mathcal{F}_t] \leq E[u(\xi + X_T) \mid \mathcal{F}_t]$ for all $t \in [0,T]$. 
\end{proof}
The previous lemma gives two \emph{a priori} estimates for subsolutions of Equation \eqref{eq:subsol}. In particular, it shows that the family of random variables $u(\xi + X_T)$, when $(X, Y, Z)$ runs through $\mathcal{A}(x)$, is uniformly integrable.
\begin{remark}
\begin{itemize}[fullwidth]
\item[a)]	Due to the admissibility condition and the previous lemma, it holds $V(x) \in \mathbb{R}$ for every $x>0$. 
	In fact, for any $(X, Y, Z) \in \mathcal{A}(x)$, since $(x,x,0) \in {\cal A}$, we can assume $Y_0\ge x$.
	By admissibility,
	\begin{equation*}
		u(Y_0) \leq E[u(Y_T)] \leq E[u(\xi + X_T)].
	\end{equation*}
	Lemma \ref{estimates_u} and Jensen's inequality give 
	\begin{equation*}
		u(Y_0)^p \leq E[\abs{u(\xi + X_T)}^p] \leq C.
	\end{equation*}
\item[b)]	If a subsolution $(X, Y, Z) \in {\cal A}(x)$ is such that $\log(Y)$ is a submartingale, then since $Y_0\ge x$, we have $E[\log(Y_t)] \ge \log(x) > 0$ for all $t \in [0,T]$.
	Hence, $Y_t=0$ with probability zero. 
	Therefore, the function $u = \log$ can be used to defined admissibility of subsolutions. 
\end{itemize}
\end{remark}
 The next lemma describes the set of subsolutions.
\begin{lemma}
\label{convex_adm}
	If $g$ satisfies \ref{jconv}, then the set $\mathcal{A}(x)$ is convex.
\end{lemma}
\begin{proof}
See Appendix \ref{appendix}.
\end{proof}
The following existence theorem is the first main result of this paper.
\begin{theorem}
\label{existence}
	Assume that the generator $g$ satisfies \ref{jconv}, \ref{lsc}, \ref{nor} and \ref{pos}; and that the random endowment $\xi$ belongs to $L^{\infty}_+$. 
	Then there exists a trading strategy $\bar{\pi} \in \Pi$ with associated wealth process $\bar{X}$ and a pair $(\bar{Y}, \bar{Z}) \in \mathcal{A}^u(\xi + \bar{X}_T,g)$ such  that $\bar{Y}_0 = V(x)$.
\end{theorem}
\begin{proof}
	Let $((X^n, Y^n, Z^n))$ be a sequence in $\mathcal{A}(x)$ such that $Y^n_0 \uparrow V(x)$. 
	The proof goes in several steps. We start by making some transformations on the maximizing sequence $((X^n, Y^n, Z^n))$.
	\begin{enumerate}[label=\textit{Step \arabic*},fullwidth]
		\item \emph{ Preliminary transformations.} \label{step:preliminary}
		The sequence $((X^n, Y^n, Z^n))$ can be considered to be such that for all $n \in \mathbb{N}$, $Y^n_0\geq x$ and $Y^n\geq X^n$. 
		In fact, since the set $\mathcal{A}(x)$ contains the triple $(x,x,0)$, by definition of $V(x)$ it holds $V(x) \geq x$. 
		Hence, we can assume without loss of generality that $Y^n_0 \geq x$, for all $n$. For each $n\in \mathbb{N}$, define the stopping time $\delta^n$ by
		\begin{equation*}
			\delta^n := \inf\{t \geq 0: Y^n_t \leq X^n_t\}\wedge T,
		\end{equation*}
		and put
		\begin{equation*}
			\hat{Y}^n := Y^n1_{[0,\delta^n)}+ Y^n_{\delta^n}1_{[\delta^n,T]}; \quad \hat{Z}^n:= Z^n1_{[0,\delta^n]}
		\end{equation*}
		and
		\begin{equation*}
			\hat{X}^n := X^n1_{[0,\delta^n]} + X^n_{\delta^n}1_{[\delta^n , T]}.
		\end{equation*}
		The triple $(\hat{X}^n, \hat{Y}^n, \hat{Z}^n)$ belongs to $\mathcal{A}(x)$. 
		In fact, for all $s,t\in [0,T]$ with $0\leq s\leq t\leq T$, on the set $\{s\leq \delta^n\leq t\}$ we have
		\begin{align*}
			\hat{Y}^n_s + \int_s^t &g_u(\hat{Y}^n_u, \hat{Z}^n_u)\,du - \int_s^t\hat{Z}^n_u\,dW_u\\
				&= Y^n_s + \int_s^{\delta^n}g_u(Y^n_u, Z^n_u)\,du - \int_s^{\delta^n}Z^n_u\,dW_u + \int_{\delta^n}^tg_u(Y^n_{\delta^n}, 0)\,du \\
				& \leq Y^n_{\delta^n} = \hat{Y}^n_{\delta^n}.              
		\end{align*}
		On the sets $ \{s\geq \delta^n\}$ and $\{t\leq \delta^n\}$ the proof is the same. 
		Now for the forward process, let $t\in [0,T]$. 
		On the set $\{\delta^n\leq t\}$, putting $\hat{\pi}^n := \pi^n1_{[0,\delta^n]}$, we have
		\begin{align*}
			\hat{X}^n_t &= X^n_{\delta^n} = x + \int_0^{\delta^n}X^n_u\pi^n_u\sigma_u\,dW^Q_u + \int_{\delta^n}^t0\,dW^Q_u
 			= x + \int_0^{t}\hat{X}^n_u\hat{\pi}^n_u\sigma_u\,dW^Q_u.
		\end{align*}
 		On $\{t\leq \delta^n\}$ there is nothing to prove. 
 		In order to show that the terminal condition is satisfied, notice that on the set $\{\delta^n < T\}$ it holds $Y^n_{\delta^n} = X^n_{\delta^n}$. 
 		This is because $Y^n_0 \geq x$, $X^n$ is continuous and $Y^n$ only jumps upward. 
 		Thus, 
		\begin{equation*}
			\hat{Y}^n_T = Y^n_{\delta^n} = X^n_{\delta^n} \leq X^n_{\delta^n} + \xi = \hat{X}^n_{T} + \xi
		\end{equation*} 
		and on the set $\{\delta^n = T\}$ it holds 
		\begin{equation*}
			\hat{Y}^n_T = Y^n_T \leq \xi +X^n_T = \xi + \hat{X}^n_T.
		\end{equation*}
		In addition, for all $n\in \mathbb{N}$, $\hat{\pi}^n$ is a trading strategy and
		$u(\hat{Y}^n)$ is a $P$-submartingale. 
		In fact, for all $0\leq s\leq t\leq T$, due to the admissibility of $Y^n$, we have 
		\begin{equation*}
			E[u(\hat{Y}^n_t) - u(\hat{Y}^n_s)\mid \mathcal{F}_s] = E[u(Y^n_{(s\vee\delta^n)\wedge t}) - u(Y^n_s)\mid \mathcal{F}_s] \geq 0.
		\end{equation*}
		Hence $\hat{Y}^n$ is admissible. 
		Therefore, we have 
		\begin{equation*}
			((\hat{X}^n, \hat{Y}^n, \hat{Z}^n)) \subseteq \mathcal{A}(x)
		\end{equation*}
		with $\hat{Y}^n_0 \uparrow V(x)$ and for all $t\in [0,T]$, $\hat{X}^n_t \leq \hat{Y}^n_t$. 
		In the sequel of the proof we shall simply write $(X^n, Y^n,Z^n)$ for $(\hat{X}^n, \hat{Y}^n, \hat{Z}^n)$, for every $n \in \mathbb{N} $.
		\item \emph{An estimate for the value process.} \label{step:estimate}
		Now we provide a bound on the value process that will be a key ingredient for the localization in the subsequent step. 
		Since $(X^n_T)$ is a sequence of positive random variables, by \citep[Lemma A1.1]{Del-Sch94} there exists a sequence denoted $(\tilde{X}^n_T)$ in the asymptotic convex hull of $(X^n_T)$ and an $\mathcal{F}_T$-measurable random variable $X$ such that 
		\begin{equation*}
			\lim_{n\rightarrow \infty} \tilde{X}^n_T = X \quad \text{$Q$-a.s.}
		\end{equation*}
		Let $(\tilde{X}^n) $ be the sequence in the asymptotic convex hull associated to $(\tilde{X}^n_T) $.
		For each $n \in \mathbb{N}$ the process $\tilde{X}^n$ is positive and inherits the $Q$-supermartingale property of $X^n$, that is, $E_Q[\tilde{X}^n_T] \leq x$. 
		Hence, it follows from Fatou's lemma that 
		\begin{equation*}
			x \geq \liminf_{n \rightarrow \infty}E_Q[\tilde{X}^n_T]\geq E_Q[\liminf_{n\rightarrow \infty}\tilde{X}^n_T] = E_Q[X].
		\end{equation*}
		By continuity of the function $u$ and $Q$-almost sure convergence of $(\tilde{X}^n_T)$ it follows that $(u(\xi + \tilde{X}^n_T))$ converges to $u(\xi + X)$ $Q$-a.s., and therefore $P$-a.s. by equivalence of measures. 
		Moreover, due to Lemmas \ref{convex_adm} and \ref{estimates_u}, the family $(u(\xi + \tilde{X}^n_T))_{n}$ is uniformly integrable. 
		Therefore, we can conclude using the dominated convergence theorem that
		\begin{equation}
		\label{L1m}
			\lim_{n\rightarrow \infty} u(\xi +\tilde{X}^n_T) = u(\xi + X)\quad \text{in $L^1$}.
		\end{equation}
		For all $n\in \mathbb{N}$ and $t\in [0,T]$ define
		\begin{equation*}
			M^n_t := E[u(\xi + \tilde{X}^n_T)\mid \mathcal{F}_t] \quad \text{and}\quad M_t := E[u(\xi + X)\mid \mathcal{F}_t].
		\end{equation*}
		We denote by $((\tilde{X}^n, \tilde{Y}^n, \tilde{Z}^n))$ the sequence in the asymptotic convex hull of $((X^n, Y^n, Z^n))$ associated to $(\tilde{X}^n_T)$. 
		By Lemma \ref{convex_adm}, $((\tilde{X}^n, \tilde{Y}^n, \tilde{Z}^n)) \subseteq \mathcal{A}(x)$, and Lemma \ref{estimates_u} leads to 
		\begin{equation*}
			u(\tilde{Y}^n_t) \leq M_t^n \leq (M^n_.)^\ast_T \quad\text{for all $t\in [0,T]$},
		\end{equation*}
		which implies, since $u^{-1}$ is increasing, that $\tilde{Y}^n_t \leq u^{-1}((M^n_.)^\ast_T)$. 
		Thus, $(\tilde{Y}^n_.)^\ast_T \leq u^{-1}((M^n_.)^\ast_T)$; recall that $\tilde{Y}^n_t \geq \tilde{X}^n_t \geq 0$. 
		Using again the fact that $u^{-1}$ is increasing and the inequalities
		\begin{equation*}
			(M^n_.)^\ast_T \leq (M^n_. - M_. + M_.)^\ast_T \leq (M^n_. - M_.)^\ast_T + M^\ast_T
		\end{equation*}
		we finally have 
		\begin{equation*}
			(\tilde{Y}^n_.)^\ast_T \leq u^{-1}((M^n_. - M_.)^\ast_T + M^\ast_T).
		\end{equation*}
		\item \emph{Local bound for the control process.} 
		Here we obtain an estimate that will enable us to use a compactness argument for the space $\mathcal{L}^1$. 
		That estimate stems from the fact that $Y^n$ can be shown to be a local submartingale. 
		We start by introducing a localization of the value processes. 
		Since the sequence $(M^n_T)$ converges in $L^1$, for a given $k\in \mathbb{N}$ we may, and do, choose a subsequence $(M^{n,k})_n$ such that
		\begin{equation}
		\label{subseqM}
			E[\vert M^{n,k}_T-M_T\vert] \leq \frac{2^{-n}}{k} \quad  n\in \mathbb{N}.
		\end{equation}
		Let $((\tilde{X}^{n,k},\tilde{Y}^{n,k},\tilde{Z}^{n,k}))_n$ be the subsequence of $((\tilde{X}^n,\tilde{Y}^n,\tilde{Z}^n))_n$ associated to $(M^{n,k}_T)_n$. 
		Now, introduce the sequence of stopping times 
		\begin{equation*}
			\tau^k = \inf\Set{t\geq 0: (\tilde{Y}^{n,k}_.)^*_t \geq k, \,\,\text{for some $n\in \mathbb{N}$} } \wedge T.
		\end{equation*}
		Let us show that $(\tau^k)$ is in fact a localizing sequence.
		\begin{align}
			\nonumber P[\tau^k = T] &= P\left[ (\tilde{Y}^{n,k}_.)^*_T < k,\, \text{for all $n\in \mathbb{N}$} \right]\\
			\nonumber             &\geq P\left[ u^{-1}((M^{n,k}_. - M_.)^*_T + M^*_T) < k,\, \text{for all $n \in \mathbb{N}$} \right]\\
			\nonumber             &= 1 - P\left[(M^{n,k}_. - M_.)^*_T + M^*_T \geq u(k),\, \text{for some $n\in \mathbb{N}$} \right]\\
			\nonumber             & \geq 1 - P\left[ \Set{(M^{n,k}_. - M_.)^*_T \geq  1,\, \text{for some $n\in\mathbb{N}$}}\cup \Set{(M_.)^*_T > u(k)-1} \right]\\
			\nonumber             &= 1 - P\left[(M^{n,k}_. - M_.)^*_T \geq  1,\,\text{for some $n\in\mathbb{N}$}\right] - P\left[(M_.)^*_T > u(k)-1 \right]\\
			\nonumber             &\geq 1 - \sum_{n}P\left[ (M^{n,k}_. - M_.)^*_T \geq 1 \right] - P\left[ (M_.)^*_T > u(k)-1) \right]\\
			\label{Doob_Markov}             &\geq 1 - \sum_{n}E\left[ \vert M^{n,k}_T -M_T\vert \right]-\frac{E\left[(M_.)^*_T \right]}{u(k)-1}\\
			\nonumber             &\geq 1 - \frac{1}{k} - \frac{E[(M_.)^*_T ]}{u(k)-1}\longrightarrow 1\\
			\nonumber             & \quad \quad \quad \quad \quad \quad \quad \quad k \longrightarrow \infty,
		\end{align}
		where we used Markov's inequality to obtain \eqref{Doob_Markov}. 
		Therefore $(\tau^k)$ is a localizing sequence.

		Let $n, k \in \mathbb{N}$, for all $t\in [0,T]$, $\tilde{Y}^{n,k}_{t\wedge\tau^k}$ is integrable. It follows from Jensen's inequality, since $u^{-1}$ is convex, that for all $0\leq s\leq t\leq T$,
		\begin{align}
		\nonumber
			E\left[ \tilde{Y}^{n,k}_{t\wedge\tau^k}\mid \mathcal{F}_s \right]& = E\left[ u^{-1}\circ u(\tilde{Y}^{n,k}_{t\wedge\tau^k}) \mid\mathcal{F}_s \right]\\
		\label{y_stopped}
                          & \geq u^{-1}\left( E\left[ u(\tilde{Y}^{n,k}_{t\wedge\tau^k})\mid \mathcal{F}_s \right] \right).
		\end{align}
		On the set $\{\tau^k \leq s\}$ it holds $E[u(\tilde{Y}^{n,k}_{\tau^k}) \mid \mathcal{F}_s] = u(\tilde{Y}^{n,k}_{\tau^k})$, and recalling that $u(\tilde{Y}^{n,k})$ is a submartingale, on the set $\{ \tau_k > s\}$ it holds $E[u(\tilde{Y}^{n,k}_{t\wedge\tau^k}) \mid \mathcal{F}_s] \geq u(\tilde{Y}^{n,k}_{s\wedge\tau^k} )$ by the optional sampling theorem.
		As $u^{-1}$ is increasing, \eqref{y_stopped} leads to
		\begin{align*}
			E\left[ \tilde{Y}^{n,k}_{t\wedge\tau^k}\mid \mathcal{F}_s \right] & \geq u^{-1}\left( u\left(\tilde{Y}^{n,k}_{s\wedge\tau^k} \right) \right)
                            = \tilde{Y}^{n,k}_{s\wedge\tau^k}.
		\end{align*}
 		Hence for all $n\in\mathbb{N}$, $\tilde{Y}^{n,k,\tau^k}:=\tilde{Y}^{n,k}_{.\wedge\tau^k}$ is a submartingale and $E[\tilde{Y}^{n,k}_{\tau^k}\mid\mathcal{F}_.]$ is a martingale. 
 		By Doob-Meyer decomposition, see \citep[Theorem 3.3.13]{Pro}, the c\`adl\`ag submartingale $\tilde{Y}^{n,k,\tau^k}$ admits the unique decomposition 
		\begin{equation}
		\label{decomY}
			\tilde{Y}^{n,k}_{t\wedge\tau^k} = \tilde{Y}_0^{n,k} + \tilde{A}^{n,k}_{t\wedge\tau^k} + \tilde{N}_{t\wedge\tau^k}^{n,k}, \quad t \in [0,T],
		\end{equation}
 		where $\tilde{A}^{n,k,\tau^k}$ is an increasing predictable process starting at $0$ and $\tilde{N}^{n,k,\tau^k}$ is a local martingale. Moreover, by Equation \eqref{eq:subsol} and Lemma \ref{convex_adm} there exists an increasing c\`adl\`ag process $\tilde{K}^{n,k}$ with $\tilde{K}^{n,k}_0 =0 $ such that 
		\begin{equation*}
			\tilde{Y}^{n,k}_t = \tilde{Y}^{n,k}_0 + \int_0^tg_u(\tilde{Y}^{n,k}_u, \tilde{Z}^{n,k}_u)\,du + \tilde{K}^{n,k}_t - \int_0^t\tilde{Z}^{n,k}_u\,dW_u;
		\end{equation*} 
		where $\int g(\tilde{Y}^{n,k}, \tilde{Z}^{n,k})\,du + \tilde{K}^{n,k}$ is increasing, since $g$ fulfills \ref{pos}, and is predictable. 
		In addition $\int\tilde{Z}^{n,k}\,dW$ is a local martingale. 
		By uniqueness of Doob-Meyer decomposition the processes $-\int\tilde{Z}^{n,k}1_{[0,\tau^k]}\,dW$ and $\tilde{N}^{n,k,\tau^k}$ as well as $\int g(\tilde{Y}^{n,k}, \tilde{Z}^{n,k})1_{[0,\tau^k]}\,du + \tilde{K}^{n,k,\tau^k}$ and $\tilde{A}^{n,k,\tau^k}$ are indistinguishable.
		Then, from Equation \eqref{decomY} and $\tilde{Y}^{n,k}_t \geq 0$ we have for all $t\in [0,T]$
		\begin{align}
		\nonumber
			\int_0^{t\wedge\tau^k}\tilde{Z}^{n,k}_u\,dW_u & = \tilde{Y}_0^{n,k} - \tilde{Y}_{t\wedge\tau^k}^{n,k} + \tilde{A}^{n,k}_{t\wedge\tau^k}\\
		\label{z1}                            & \leq V(x) + \tilde{A}^{n,k}_{\tau^k},
		\end{align}
		where the last inequality comes from the fact that $(\tilde{Y}^{n,k}_0)_n$ increases to $V(x)$.
		On the other hand, since $(\tilde{Y}^{n,k},\tilde{Z}^{n,k})$ satisfies \eqref{eq:subsol} and $g$ satisfies \ref{pos},
		\begin{align}
		\nonumber
			\int_0^{t\wedge\tau^k}\tilde{Z}^{n,k}_u\,dW_u &\geq \tilde{Y}^{n,k}_0 - 
			\tilde{Y}^{n,k}_{t\wedge\tau^k} + \int_0^{t\wedge\tau^k}g(\tilde{Y}^{n,k}_u, 
		\nonumber
			\tilde{Z}^{n,k}_u)\,du\\
		\label{z2}
                            & \geq - \tilde{Y}^{n,k}_{t\wedge\tau^k}\\
		\nonumber
                            & \geq - E[\tilde{Y}^{n,k}_{\tau^k}\mid \mathcal{F}_{t\wedge\tau^k}],
		\end{align}
		where the last inequality comes from the fact that $\tilde{Y}^{n,k,\tau^k}$ is a submartingale.
 		Therefore, $\int\tilde{Z}^{n,k}1_{[0,\tau^k]}\,dW$ is a supermartingale, as a local martingale bounded from below by the martingale $-E[\tilde{Y}^{n,k}_{\tau^k}\mid \mathcal{F}_{.\wedge\tau^k}]$.
		Hence, the inequalities \eqref{z1} and \eqref{z2} above lead to
		\begin{equation*}
			\abs{ \int_0^{t\wedge\tau^k}\tilde{Z}^{n,k}_u\,dW_u} \leq V(x) +\abs{ \tilde{Y}^{n,k}_{t\wedge\tau^k}} + \tilde{A}^{n,k}_{\tau^k},
		\end{equation*}
		which implies
		\begin{equation*}
			\left(\int_0^.\tilde{Z}^{n,k}_u\,dW_u\right)^\ast_{T\wedge\tau^k}\leq  V(x) + k+ \tilde{A}_{\tau^k}^{n,k}.
		\end{equation*}
		The random variable $\tilde{A}^{n,k}_{\tau^k}$ is bounded in $L^1$, since we have $\tilde{A}^{n,k}_{\tau^k} = \tilde{Y}^{n,k}_0 - \tilde{Y}^{n,k}_{ \tau^k} + \int_0^{\tau^k}\tilde{Z}^{n,k}_u\,dW_u$ with $(\tilde{Y}^{n,k}_0)_n$ increasing; $\int\tilde{Z}^{n,k}1_{[0,\tau^k]}\,dW$ a $P$-supermartingale and $\tilde{Y}^{n,k}_{\tau^k}$ bounded. 
		Hence, by Burkholder-Davis-Gundy's inequality, $(\tilde{Z}^{n,k}1_{[0,\tau^k]})_n$ is bounded in $\mathcal{L}^1$. 
		\item \emph{Construction of the candidates $\bar{Z}$ and $\bar{Y}$.} \label{step:candidates_Y-Z}
		Now we are ready to construct the candidates maximizers for the control and the value processes. 
		These constructions are based on compactness principles for the spaces $\mathcal{L}^1$ and $L^1$. 
		Since $(\tilde{Z}^{n,k}1_{[0,\tau^k]})_n$ is $\mathcal{L}^1$ bounded, there exists, by means of \citep[Theorem A]{Del-Sch96}, a sequence again denoted $(\tilde{Z}^{n,k}1_{[0,\tau^k]})_n$ in the asymptotic convex hull of $(\tilde{Z}^{n,k}1_{[0,\tau^k]})_n$ which converges in $\mathcal{L}^1$ along a localizing sequence $(\sigma^{n,k})_n$, and therefore $P\otimes dt$-a.s., to a process $\bar{Z}^k$. 
		We obtain $\bar{Z}$ by implementing a diagonalization procedure such as in step 7 of the proof of \citep[Theorem 4.1]{DHK1101}: For another $k' > k$, we can find a subsequence $(\tilde{Z}^{n,k'})_n$ such that $(\tilde{Z}^{n,k'}1_{[0,\tau^{k'}]}1_{[0,\sigma^{n,k'}]})_n$ converges to a process $\bar{Z}^{k'}$ in $\mathcal{L}^1$ and $P\otimes dt$-a.s. 
		By the same method, we can define the process $\bar{Z}$ by
		\begin{equation*}
  			\bar{Z}_0 = 0; \quad \bar{Z} = \sum_{k = 1}^{\infty}\bar{Z}^k1_{(\tau^{k-1}, \tau^k]},
		\end{equation*}
		and put $\tilde{Z}^n = \tilde{Z}^{n,n}$ and $\sigma^{n,n} = \sigma^n$. 
		Hence $( \tilde{Z}^n1_{[0,\tau^n]}1_{[0,\sigma^n]} )$ converges to $\bar{Z}$ in $\mathcal{L}^1$ and $P\otimes dt$-a.s., but we also have $( \tilde{Z}^n1_{[0,\tau^k]}1_{[0,\sigma^k]} )_n$ converges to $\bar{Z}^k$ for all $k$. 
		Thus, by Burkholder-Davis-Gundy's inequality, 
		\begin{equation*}
			\int_0^{t\wedge \tau^k\wedge\sigma^k}\tilde{Z}^n_s\,dW_s \longrightarrow \int_0^{t\wedge \tau^k\wedge\sigma^k}\bar{Z}_s\,dW_s,\quad \text{for all $t$, $P$-a.s. and for each $k$}.
 		\end{equation*}
		Taking the limit as $k\rightarrow \infty$ we have, for all t,
		\begin{equation}
		\label{limitz}
			\int_0^t \tilde{Z}^n_u\,dW_u \longrightarrow \int_0^t\bar{Z}_u\,dW_u ,\quad \text{$P$-a.s.}
		\end{equation}

		Let $(\tilde{Y}^n)$ be a sequence in the asymptotic convex hull of $(Y^n)$ corresponding to $(\tilde{Z}^n)$. 
		For all $t\in [0,T]$ and $k \in \mathbb{N}$, we have $\tilde{Y}^n_{t\wedge\tau^k} = \tilde{Y}^n_0 + \tilde{A}^n_{t\wedge\tau^k} - \int_0^{t\wedge\tau^k}\tilde{Z}^n_u\,dW_u$. 
		The sequence $(\tilde{A}^n_{T\wedge\tau^k})$ is bounded in $L^1$ as a consequence of the $L^1$-boundedness of $(A^n_{T\wedge\tau^k})_n$. 
		Therefore, by Helly's theorem, we can find a subsequence in the asymptotic convex hull of $(\tilde{A}^{n,\tau^k})_n$ still denoted $(\tilde{A}^{n,\tau^k})_n$ such that, for $k$ fixed, $(\tilde{A}^n_{t\wedge\tau^k})_n	$ converges to $\tilde{A}_{t\wedge\tau^k}$ for all $t\in [0,T]$, $P$-a.s. and such that $\tilde{A}^{\tau^k}$ is an increasing positive integrable process with $\tilde{A}_0 = 0$. 
		In particular, $(\tilde{A}_T^n)$ converges to $\tilde{A}_T$ $P$-a.s.
		Letting $k$ go to infinity, $(\tilde{A}_{t\wedge\tau^k})_k$ converges to $\tilde{A}_t$, for all $t\in[0,T)$, $P$-a.s. 
		Therefore we put 
		\begin{equation}
		\label{limity}
			\tilde{Y}_t := \lim_{k\rightarrow \infty}\lim_{n\rightarrow \infty}\tilde{Y}^n_{t\wedge\tau^k} = V(x) + \tilde{A}_t - \int_0^t\hat{Z}_u\,dW_u; \quad t \in [0, T).
		\end{equation}
		and for all $t\in [0,T)$, define
		\begin{equation*}
			\bar{Y}_t := \lim_{s\downarrow t\, s \in \mathbb{Q}}\tilde{Y}_s = V(x) + \lim_{s\downarrow t \,s \in \mathbb{Q}}\tilde{A}_s - \int_0^t\hat{Z}_u\,dW_u
		\end{equation*}
		and  $\bar{Y}_T:= \tilde{Y}_T$.
		We claim that 
		\begin{equation}
		\label{right_limit}
			\bar{Y} = \tilde{Y} \quad P\otimes dt \text{-a.s}.
		\end{equation} 
		This is because the jumps of $\tilde{Y}$ and $\bar{Y}$ coincide with the jumps of $\tilde{A}$, and being increasing, the latter process has countably many jumps.
		\item \emph{ Construction of the candidate $\bar{X}$.}\label{step:condidate_X}  
		Recall that since $g$ satisfies \ref{jconv}, by Lemma \ref{convex_adm} for all $n\in \mathbb{N}$ the triple $(\tilde{X}^n, \tilde{Y}^n, \tilde{Z}^n)$, element of the asymptotic convex hull of $((X^n, Y^n, Z^n))_n$ is in $\mathcal{A}(x)$; and from \ref{step:preliminary} we have $0\leq \tilde{X}^n_t\leq \tilde{Y}^n_t$. 
		Moreover, for each $n\in \mathbb{N}$ the process $\tilde{X}^n$ admits the representation 
		\begin{equation*}
			\tilde{X}^n_t = x + \int_0^t\tilde{\nu}^n_u\,dW^Q_u,  \quad t \in [0,T]
		\end{equation*}
		for some predictable process $\tilde{\nu}^n \in \mathcal{L}^1(Q)$.
 		Hence for all $t\in [0,T]$, for all $n\in \mathbb{N}$, we have
		\begin{align}
		\label{boundpi}
			\nonumber \abs{ \int_0^t\tilde{\nu}^n_u\,dW^Q_u} = \abs{ \tilde{X}^n_t - x} 
										\leq \vert \tilde{Y}^n_t\vert + x,
		\end{align}
 		which implies, taking $(\tilde{\nu}^{n,k})_{n}$ to be the subsequence corresponding to $(M^{n,k})_{n}$, recall \eqref{subseqM},
		\begin{equation}
		\label{eq:nubound}
  			\left( \int_0^. \tilde{\nu}^{n,k}_u\,dW_u^Q \right)^\ast_{T\wedge\tau^k} \leq (\tilde{Y}^{n,k}_.)^\ast_{T\wedge\tau^k} + x \leq k + x.
		\end{equation}
 		Therefore, by Burkholder-Davis-Gundy's inequality $(\tilde{\nu}^{n,k}1_{[0,\tau^k]})_{n}$ is bounded in $\mathcal{L}^1(Q)$. 
 		With this local $\mathcal{L}^1(Q)$ bound at hand, we can use similar arguments as in \ref{step:candidates_Y-Z} to obtain a process $\bar{\nu}$ such that
		\begin{equation}
		\label{eq:nubound2}
			\int_0^{t\wedge \tau^k}\tilde{\nu}^n_u\,dW^Q_u \longrightarrow \int_0^{t \wedge \tau^k}\bar{\nu}_u\,dW_u^Q\quad \text{for all } t, \,Q\text{-a.s. and for each k}
		\end{equation}
		and
		\begin{equation*}
			\int_0^t\tilde{\nu}^n_u\,dW^Q_u \longrightarrow \int_0^t\bar{\nu}_u\,dW^Q_u\quad \text{for all $t\in [0,T]$ , $Q$-a.s.}
		\end{equation*}
		Put
		\begin{equation}
		\label{xhat}
			\bar{X}_t = x + \int_0^t\bar{\nu}_u\,dW^Q_u.
		\end{equation}

		\item \emph{Verification.} 
		It follows from the definition of $\bar{Y}$ that $\bar{Y}_0 \geq V(x)$; let us verify that $(\bar{X}, \bar{Y}, \bar{Z})$ actually belongs to $\mathcal{A}(x)$. 
		We start by showing that $\bar{X}$ is a wealth process.
		From $\tilde{X}^n \ge 0$ for all $n \in \mathbb{R}$, follows $\bar{X} \ge 0$.
		Since $\sigma\sigma^\prime$ is of full rank, we can find a predictable process $\bar{\pi}$ such that $\bar{\pi}\sigma = \bar{\nu}$.
		Hence, from \eqref{eq:nubound} and \eqref{eq:nubound2}, $\bar{\pi}\sigma1_{[0,\tau^k]} \in {\cal L}^1(Q)$ for all $k \in \mathbb{N}$ and therefore $\bar{\pi}\sigma \in {\cal L}(Q)$ and $d\bar{X}_t = \bar{\pi}_u\sigma_u(\theta_udu + dW_u)$. 
		Next let us show that $(\bar{Y}, \bar{Z}) \in \mathcal{A}^u(\xi + \bar{X}_T,g)$. 
		To that end, we use an argument from \citep{DHK1101} . 
		By \eqref{right_limit}, there exists a set $B \subseteq \Omega \times [0,T]$ with $P\otimes dt (B^c) = 0$ such that $\bar{Y}_t(\omega) = \tilde{Y}_t(\omega) $ for all $(\omega,t) \in B$. 
		Then, there exists a set $D \subseteq \set{\omega: (\omega,t) \in B, \text{ for some }t }$ with $P(D) = 1$ such that for all  $\omega \in D$ the set $I(\omega) := \set{t \in [0,T]: (\omega,t) \in B}$ is a set of Lebesgue measure $T$ and $\bar{Y}_t(\omega) = \tilde{Y}_t(\omega)$ for all $t \in I(\omega)$.
		Denote by $\lambda_i^n$, $n\leq i\leq \Lambda^n$, the convex weights of the convex combination $\tilde{Z}^n$. 
		Let $s,t \in I$, $s\leq t$, where $I$; $s$ and $t$ depend on $\omega \in D$. 
		Using subsequently Fatou's lemma and \ref{jconv} we are led to 
		\begin{align}
		\nonumber \bar{Y}_s &+ \int_s^tg_u(\bar{Y}_u, \bar{Z}_u)\,du - \int_s^t\bar{Z}_u\,dW_u \\
   		\nonumber       &\leq \lim_{k\rightarrow \infty}\liminf_{n \rightarrow \infty}\left( \tilde{Y}^n_{s\wedge\tau^k} + \int_{s\wedge\tau^k}^{t\wedge\tau^k}g_u(\tilde{Y}^n_u, \tilde{Z}^n_u1_{[0, \sigma^n]}(u))\,du - \int_{s\wedge\tau^k}^{t\wedge\tau^k}\tilde{Z}^n_u\,dW_u \right)\\
		\nonumber &\leq \lim_{k\rightarrow\infty}\liminf_{n\rightarrow \infty}\sum_{i = n}^{\Lambda^n}\lambda_i^n\left(Y^i_{s\wedge\tau^k} + \int_{s\wedge\tau^k}^{t\wedge\tau^k}g_u	(Y^i_u, Z^i_u)\,du - \int_{s\wedge\tau^k}^{t\wedge\tau^k}Z^i_u\,dW_u \right)\\
		\nonumber &\leq \lim_{k\rightarrow \infty}\liminf_{n\rightarrow \infty}\sum_{i = n}^{\Lambda^n}\lambda_i^nY^i_{t\wedge\tau^k} = \lim_{k \rightarrow \infty}\liminf_{n\rightarrow \infty}\tilde{Y}^n_{t\wedge\tau^k}
		\end{align}
		\begin{align}
			& = \lim_{k\rightarrow \infty}\tilde{Y}_{t\wedge\tau^k} = \tilde{Y}_t= \bar{Y}_t.
		\label{verif}
		\end{align}
		If $s$ or $t$ are not in $I$, then there exist two sequences $(s_n)$ and $(t_n)$ in $I$ such that $s_n\downarrow s$, $t_n\downarrow t$ and $s_n\leq t_n$. Equation \eqref{verif} holds for each $s_n$, $t_n$. 
		Namely,
		\begin{equation*}
			\bar{Y}_{s_n} + \int_{s_n}^{t_n}g(\bar{Y}_u, \bar{Z}_u)\,du - \int_{s_n}^{t_n}\bar{Z}_u\,dW_u \leq \bar{Y}_{t_n}
		\end{equation*}
		holds for all $n\in \mathbb{N}$. 
		Since $\bar{Y}$ is right continuous and the integrals are continuous, taking the limit as $n$ tends to infinity yields the desired result for $s$ and $t$. 
		Therefore, the pair $(\bar{Y}, \bar{Z})$ satisfies the inequality \eqref{eq:subsol} with terminal condition $H = \xi + \bar{X}_T$ since for all $n \in \mathbb{N}$, $\tilde{Y}^n_T \leq \xi + \tilde{X}^n_T$; and $(\tilde{Y}^n_T)$ and $(\tilde{X}^n_T)$ converges $P$-a.s. to $\bar{Y}_T$ and $\bar{X}_T$, respectively. 
		Now let us show that $\bar{Y}$ is admissible and is a c\`adl\`ag process. 
		Due to Lemmas \ref{estimates_u} and \ref{convex_adm} and positivity of $u$ we have for all $n\in \mathbb{N}$ and $t\in [0,T]$ 
		\begin{align*}
			u\left( \tilde{Y}^n_t \right)^{p} &\leq E\left[ u\left( \xi + \tilde{X}^n_T \right)\mid \mathcal{F}_t\right]^{p}\\
                 &\leq E\left[ u(\xi + \tilde{X}^n_T)^{p}\mid \mathcal{F}_t \right],
		\end{align*}
		where we used Jensen's inequality.
		Taking expectation on both sides leads to $E[u(\tilde{Y}^n_t)^{p}] \leq E[u(\xi + \tilde{X}^n_T)^{p}] \leq C$. 
		Hence, the family $(u(\tilde{Y}^n_t))_{n}$ is uniformly integrable, for all $t\in[0,T]$. 
		Since for all $n$ the process $\tilde{Y}^n$ is admissible, we have $u(\tilde{Y}^n_s) \leq E[u(\tilde{Y}^n_t)\mid\mathcal{F}_s]$, $0\leq s\leq t\leq T$. 
		Taking the limit as $n$ goes to infinity, we obtain by means of continuity of $u$ and dominated convergence theorem $u(\tilde{Y}_s) \leq E[u(\tilde{Y}_t)\mid \mathcal{F}_s]$, i.e. $u(\tilde{Y})$ is a submartingale. The continuity property of the function $u$ and definition of $\hat{Y}$ imply
		\begin{equation*}
			u(\bar{Y}_t) = \lim_{s\uparrow t, s\in \mathbb{Q}}u(\tilde{Y}_s),
		\end{equation*}
		therefore by \citep[Proposition 1.3.14]{Kar-Shr}, $u(\bar{Y})$ is a c\`adl\`ag submartingale, and $\bar{Y}$ is thus c\`adl\`ag as well.  Hence $(\bar{X},  \bar{Y}, \bar{Z}) \in \mathcal{A}(x)$ and consequently $V(x) = \bar{Y}_0$, which ends the proof.
	\end{enumerate}
\end{proof}
\begin{remarks}
\begin{itemize}[fullwidth]
	\item[a)] Unlike in \cite{DHK1101} and \cite{Hey-Kup-Mai} where minimal supersolutions of BSDEs are studied, we cannot guarantee that the stochastic integral of the process $\bar{Z}$ is a supermartingale even for a bounded terminal condition $\xi$. 
	This is due to the fact that the random variable $\bar{X}_T$ may not be integrable.
	\item[b)] In the above result, the assumption $\xi \in L^1_+(\Omega,\mathcal{F}_T, Q)$ can be replaced by $\xi \in L^2_+(\Omega,\mathcal{F}_T, P)$. 
	This would cost a stronger integrability condition on the process $\theta$. 
	Indeed, if the martingale $\mathcal{E}(-\int\theta\,dW)$ satisfies the reverse H\"older inequality $R_2$ that is, there is a positive constance $C$ such that for all stopping times $\tau \leq T$ it holds
	\begin{equation*}
		E\left[ \mathcal{E} \left(-\int\theta_u\,dW_u \right)_T^2 \Mid \mathcal{F}_{\tau}\right]^{\frac{1}{2}} \leq C \mathcal{E}\left( -\int\theta_u\,dW_u \right)_{\tau},
	\end{equation*}
	then by \cite[Proposition 3]{Dol-Mey} we have $E_Q[\xi] = E[\mathcal{E}(-\int_0^.\theta_u\,dW_u)_T\xi]\leq CE[\xi^2]$ and therefore the first estimate of Lemma \ref{estimates_u} remains valid. 
	\end{itemize}
\end{remarks}
We finish this section with a direct consequence of Theorem \ref{existence} and its proof. 
Namely, existence of a maximal subsolution of a decoupled controlled FBSDE:
\begin{corollary}
	Assume that the generator $g$ satisfies \ref{jconv}, \ref{lsc}, \ref{nor} and \ref{pos}; and $\xi \in L^{\infty}_+$. 
	Then the system
	\begin{equation}
	\label{system}
		\begin{cases}
			Y_s &\leq Y_t  - \int_s^tg(Y_u, Z_u)\,du + \int_s^tZ_u\,dW_u , \quad Y_T \leq \xi + X_T^\pi\\
			X_t^{\pi} &= x + \int_0^t \pi_u\sigma_u\left( \theta_u\,du + \,dW_u \right), \quad \pi \in \Pi
		\end{cases}
	\end{equation}
	admits a maximal subsolution. 
	That is, there exists a control $\pi^\ast \in \Pi$ and a triple $(X^{\pi^\ast}, Y^\ast, Z^\ast)$ satisfying \eqref{system} with $u\left( Y^\ast \right)$ being a submartingale such that for any control $\pi \in \Pi$ and any processes $(X^{\pi}, Y, Z)$ satisfying \eqref{system} with $u(Y)$ a submartingale, we have $Y^\ast_0 \geq Y_0$.
\end{corollary}
\begin{proof}
	This follows from Theorem \ref{existence}.
\end{proof}

\subsection{ Stability Results }

In this section we assess the stability of maximal subsolutions with respect to the terminal condition and the generator. We will show that maximal subsolutions have a monotone stability with respect to both data.
These stability results, already proved in \citep{DHK1101} for minimal supersolution, will enable us to obtain a robust representation of the operator $\mathcal{E}^g_0$.
\begin{proposition}
\label{thm:stability-xi}
	Assume that the generator $g$ satisfies \ref{jconv}, \ref{lsc}, \ref{nor} and \ref{pos}. 
	Let $(\xi^n) \subseteq L^{\infty}_+$.
	 If $(\xi^n)$ decreases pointwise to a random variable $\xi$, then $\mathcal{E}^g_0(\xi) = \lim_{n\rightarrow \infty}\mathcal{E}^g_0(\xi^n)$.
\end{proposition}
\begin{proof}
	See Appendix \ref{appendix}.
\end{proof}

\begin{proposition}
\label{thm:stability-g}
	Let $\xi \in L^{\infty}_+$ be a terminal condition, and $(g^n)$ be a sequence of generators decreasing pointwise to $g$. 
	Assume that each function satisfies \ref{jconv}, \ref{lsc}, \ref{nor} and \ref{pos}.
	Then $\mathcal{E}^{g}_0(\xi) = \lim_{n\rightarrow \infty}\mathcal{E}^{g^n}_0(\xi)$.
\end{proposition}
\begin{proof}
	See Appendix \ref{appendix}.
\end{proof}
\section{ Representation and Characterization }
In the previous section we obtained existence of optimal trading strategies of our control problem. 
This was a rather abstract result, and only gave us little information on how one could compute such an optimizer or how it depends on the other parameters. 
The point of this section is to find a characterization of the optimal controls of Problem \eqref{eq:optprob}. 
\subsection{Robust Representation}
We consider the set
\begin{equation*}
	\mathcal{D} := \Set{\beta: \beta \text{ predictable} \text{ and } \int_0^T\abs{\beta_u}\,du < \infty }.
\end{equation*}
For any $\beta \in \mathcal{D}$ and  $q \in \mathcal{L}$, we define, for $0\le s\le t \le T$
\begin{equation*}
	\frac{dQ^q}{dP} = \exp\left(\int_0^T q_u\,dW_u -\frac{1}{2}\int_0^T\norm{q_u}^2\,du \right)  \quad \text{and} \quad  D_{s,t}^{\beta} := e^{ -\int_s^t\beta_u\,du}, \quad t \in [0,T].
\end{equation*}
We also define the set
\begin{equation*}            
	\mathcal{Q} := \Set{q \in \mathcal{L}: \frac{dQ^q}{dP}\in L^1_{+}}.
\end{equation*}

For any admissible trading strategy $\pi \in \Pi$, the associated wealth process is given by $dX^\pi_t = \pi_t\sigma_t(\theta_t\,dt + dW_t)$, with $X^\pi_0 = x$ and $X^\pi \ge 0$.
Let $0 \le s\le t\le T$, and consider the functional
\begin{equation*}
	{\cal E}^g_{s,t}(H):= \esssup\Set{ Y_s: (Y, Z) \in {\cal A}^u(H,g)}, \quad H \in L^0({\cal F}_t).
\end{equation*}
Recall that ${\cal A}^u(H, g)$ is the set of subsolutions $(Y,Z) \in {\cal S}_+\times {\cal L}$ of the BSDE with terminal condition $H$ and generator $g$ such that $u(Y)$ is a submartingale.
In particular, ${\cal E}^g_0(H) = {\cal E}^g_{0,T}(H)$ for all $H \in L^0({\cal F}_T)$.
Let $\tau \le \gamma$ be two stopping times valued in $[0,T]$. 
For any $\pi \in \Pi$, define
\begin{equation*}
	\Theta_{\tau, \gamma}(\pi) := \Set{ \pi^\prime \in \Pi: \pi^\prime1_{[\tau,\gamma] } = \pi1_{[\tau, \gamma] } }
\end{equation*}
and
\begin{equation}
\label{eq:Ytau_optimal}
	Y_\tau(X^{\pi}_\tau ) := \esssup_{\pi^\prime \in \Theta_{0,\tau}(\pi) }\mathcal{E}^g_{\tau,T}\left( X^{\pi^\prime}_T + \xi \right),
\end{equation}
where $\xi \in L^\infty_+$ is the random endowment.
We define the convex conjugate $g^\ast$ of the generator $g$ by
\begin{equation*}
	g^\ast(\beta, q):= \sup_{y\in \mathbb{R}_+ , z \in \mathbb{R}^d}\Set{\beta y + qz - g(y,z)}, \quad \beta \in \mathbb{R}, q \in \mathbb{R}^d.
\end{equation*}
Consider the condition 
\begin{enumerate}[label=(\textsc{Adm}),leftmargin=40pt]
	\item $g(y,z) \ge -1/2\norm{z}^2u''(y)/u'(y) $ on $\mathbb{R}_+ \times \mathbb{R}^d$\label{adm}.
\end{enumerate}
The following theorem gives a robust representation of $\mathcal{E}^g_{0,\tau}$.
\begin{theorem}
\label{thm:to-robustprob}
	Assume that the generator $g$ satisfies \ref{jconv}, \ref{lsc}, \ref{nor}, \ref{pos} and \ref{adm}. 
	Then, for every $\pi \in \Pi$ and any stopping time $0\le \tau \le T$, the following robust representation holds:
	\begin{equation}
	\label{eq:robust_rep}
		\mathcal{E}^g_{0,\tau}(Y_\tau(X^\pi_\tau) ) = \inf_{(\beta,q) \in \mathcal{D}\times\mathcal{Q}}E_{Q^q}\left[ D^\beta_{0,\tau} Y_\tau(X^\pi_\tau) + \int_0^\tau D^\beta_{0,u}g^\ast_u(\beta_u,q_u)\,du \right], \quad \pi \in \Pi.
	\end{equation}
\end{theorem}
For the proof of the theorem we need the following lemma.
\begin{lemma}
\label{thm:admit_solu}
	Assume $H \in L^\infty$.
	Let $f$ be a function satisfying \ref{adm} 
	and such that the BSDE with terminal condition $H$ and generator $f$ has a solution $(Y, Z)\in {\cal S}\times {\cal L}^1$ satisfying $Y \ge c$ for some $c>0$.
 	Then, $u(Y)$ is a submartingale.
\end{lemma}
\begin{proof}
	By It\^o's formula it holds
	\begin{align}
	\label{eq:Ito_submg}
		u(Y_t) &= u(Y_0) + \int_0^t\left( u'(Y_u)f(Y_u, Z_u) + \frac{1}{2}u''(Y_u)Z_u^2 \right)\,du- \int_0^tu'(Y_u)Z_u\,dW_u,
	\end{align}
	for all $t\in [0,T]$.
	Therefore since $Y >0$, due to \ref{adm} we have $u'(Y_u)f(Y_u, Z_u) + \frac{1}{2}u''(Y_u)Z_u^2  \ge 0$ so that the second term of the right hand side in \eqref{eq:Ito_submg} defines an increasing process.
	Thus, as $H \in L^\infty$ and $Y\ge c$, $u(Y)$ is a submartingale.
	In other words, $(Y, Z)$ is an admissible subsolution of the BSDE with terminal condition $H$ and generator $f$.
\end{proof}
\begin{proof}[proof of Theorem \ref{thm:to-robustprob}]
	Let $\tau \le T$ be a stopping time.
	For every $\pi \in \Pi$ and $(\beta, q	) \in \mathcal{D}\times \mathcal{Q}$, if ${\cal A}^u(Y_\tau(X^\pi_\tau),g) \neq \emptyset $, let $(Y,Z) \in \mathcal{A}^u( Y_\tau(X^\pi_\tau), g)$. 
	There exists a c\`adl\`ag increasing process $K$ with $K_0 = 0$ such that on $\set{t \le \tau}$,
	\begin{equation*}
		Y_t = Y_s + \int_s^tg_u(Y_u, Z_u)\,du - \int_s^tZ_u\,dW_u + K_t - K_s, \quad 0\le s \le t.
	\end{equation*}
	Define the localizing sequence of stopping times $(\sigma_n) $ by
	\begin{equation*}
		\sigma_n := \inf\Set{ t \ge 0: \abs{\int_0^tD^\beta_{0,u}Z_u\,dW_u} \ge n } \wedge T.
	\end{equation*}
	 Applying It\^o's formula to $D_{0,t}^{\beta}Y_t$ and Girsanov's theorem such as in \citep{Kar-Rav}, we have
	\begin{equation*}
		Y_0 \leq E_{Q^q}\left[ D_{0,t\wedge\sigma_n}^{\beta}Y_{t\wedge\sigma_n}+ \int_0^{t\wedge\sigma_n} D_{0,u}^{\beta}g^\ast_u(\beta_u, q_u)\,du \right], \quad \text{for all }n \in \mathbb{N} \text{ and } t \in [0,T].
	\end{equation*}
	Since $g$ satisfies \ref{nor} the function $g^\ast$ is positive.
	Using the fact that $(\sigma_n) $ is a localizing sequence, there is $n$ large enough such that $\tau \le \sigma_n$; and since $Y_\tau \le Y_\tau(X^\pi_\tau) $ and $D^\beta$ is positive, we have
	\begin{equation*}
		Y_0 \leq E_{Q^q}\left[ D_{0,\tau}^{\beta}Y_\tau(X^\pi_\tau) + \int_0^\tau D_{0,u}^{\beta}g^\ast_u(\beta_u, q_u)\,du \right].
	\end{equation*}
	Therefore, 
	\begin{equation}
	\label{eq:weakdual}
		\mathcal{E}^g_{0,\tau}( Y_\tau(X^\pi_\tau) ) \leq \inf_{(\beta ,q) \in \mathcal{D} \times \mathcal{Q}}E_{Q^q}\left[ D_{0,\tau}^{\beta} Y_\tau(X^\pi_\tau) + \int_0^\tau D_{0,u}^{\beta}g^\ast_u(\beta_u, q_u)\,du \right].
	\end{equation}
	If ${\cal A}^u(Y_\tau(X^\pi_\tau),g) = \emptyset $, \eqref{eq:weakdual} is obvious.

	On the other hand, for each $k \in \mathbb{N}$ and $\pi \in \Pi$ we define $H^k(\pi) := Y_\tau(X^\pi_\tau) \wedge k $, which is a bounded $\mathcal{F}_\tau$-random variable.
	Defining for every $n \in \mathbb{N}$ the function $g^n$ on $\mathbb{R}_+\times\mathbb{R}^d$ by
	\begin{equation*}
		g^n(y,z) := \sup_{\abs{\beta} \leq n; \norm{q} \leq n}\Set{ \beta y + q z - g^\ast(\beta, q)} \vee -\frac{1}{2}u''(y)\norm{z}^2/u'(y), 
	\end{equation*}
	the sequence $(g^n)$ converges pointwise to $g$ as a consequence of the Fenchel-Moreau theorem. 
	In addition, for each $n \in \mathbb{N}$ the function $g^n$ satisfies the quadratic growth condition 
	\begin{equation*}
		g^n(y,z) \le C_n(1 + \abs{y} + \norm{z}^2 ), \quad y \in \mathbb{R},\, z \in \mathbb{R}^d,\,\, C_n \ge 0.
	\end{equation*}
 	Fixing $n \in \mathbb{N}$, for every $k \in \mathbb{N}$ there exists $(Y^{n,k}, Z^{n,k}) \in {\cal S}\times {\cal L}^1$ solution of the BSDE with terminal condition $H^k(\pi)$ and driver $g^n$, see for instance \citep{Dar-Par}. 
 	It follows from \citep{Kar-Rav} that there exist predictable processes $(\beta^n, q^n)$ satisfying $\abs{\beta^n}\le C_n$ and $\int q^n\,dW \in BMO$ such that on $\set{t \le \tau}$
	\begin{equation}
	\label{eq:g-expectation}
		Y^{n,k}_t = E_{ Q^{q^n} }\left[ D_{t,\tau}^{\beta^n} H^k(\pi) + \int_t^\tau D^{\beta^n}_{0,u}g^{n,\ast}_u(\beta^n_u, q^n_u)\,du \Mid {\cal F}_t \right], \quad P\text{-a.s.,}
	\end{equation}
	where $g^{n,\ast}$ is the convex conjugate of $g^n$.
	In particular, since $g$ satisfies \ref{nor}, we have $\beta y - g^\ast(\beta, q) \le 0$ for all $\beta, q$ so that $g^n$ also satisfies \ref{nor}.
	Thus, it holds $g^{n,\ast}\ge 0$, and from $(x,x,0) \in {\cal A}(x)$ it follows $H^k(\pi) \ge x$, which yields $Y^{n,k}_t \ge E_{Q^{q^n}}[D^{\beta^n}_{t, \tau}x]> 0$.
	Since $g^n(y,z) \ge -\frac{1}{2}u''(y)\norm{z}^2/u'(y) $, it follows from Lemma \ref{thm:admit_solu} that $u(Y^{n,k})$ is a submartingale. 
	That is, $(Y^{n,k}, Z^{n,k})$ is an admissible subsolution of the BSDE with generator $g^n$ and terminal condition $H^k(\pi)$.
	Therefore, $\mathcal{E}^{g^n}_0(H^k(\pi)) \geq Y^{n,k}_0$.
	Taking the limit as $k$ goes to infinity, it follows from the monotone stability of Proposition \ref{thm:stability-xi} and the monotone convergence theorem that
	\begin{equation*}
	\mathcal{E}^{g^n}_{0,\tau}( Y_\tau(X^\pi_\tau) ) \geq E_{ Q^{q^n} }\left[ D_{0,\tau}^{\beta^n} Y_\tau(X^\pi_\tau) + \int_0^\tau D^{\beta^n}_{0,u}g^{n,\ast}_u(\beta^n_u, q^n_u)\,du \right] \quad \text{for all } n \in \mathbb{N}.
	\end{equation*} 
	Since $(\beta^n, q^n) \in {\cal D}\times {\cal Q}$ for each $n$, we have
	\begin{equation*}
		\mathcal{E}^{g^n}_{0,\tau}( Y_\tau(X^\pi_\tau) ) \geq \inf_{(\beta, q) \in \mathcal{D}\times \mathcal{Q}}E_{Q^q}\left[ D_{0,\tau}^{\beta} Y_\tau(X^\pi_\tau) + \int_0^\tau D_{0,u}^{\beta}g^{n,\ast}_u(\beta_u, q_u)\,du \right].
	\end{equation*}
	Using $g^\ast \le g^{n,\ast}$ for all $n\in \mathbb{N}$ and then taking the limit as $n$ goes to infinity, the monotone stability of Proposition \ref{thm:stability-g} yields the second inequality, which concludes the proof.
\end{proof}
\begin{proposition}
\label{thm:local_robust}
	Under the assumptions of Theorem \ref{thm:to-robustprob}, 
	for any $[0,T]$-valued stopping time $\tau$, it holds
	\begin{equation}
	\label{eq:local_robustproblem}
		V(x) = \sup_{\pi \in \Pi} \inf_{(\beta, q) \in \mathcal{D}\times \mathcal{Q} }E_{Q^q}\left[ D_{0,\tau}^{\beta} Y_\tau(X^\pi_\tau)  + \int_0^\tau D_{0,u}^{\beta}g^\ast_u(\beta_u, q_u)\,du \right].
	\end{equation}
\end{proposition}
\begin{proof}
	We have
	\begin{equation}
	\label{eq:localproblem}
		V(x) = \sup_{\pi \in \Pi } \mathcal{E}_{0, \tau}^g( Y_\tau( X^\pi_\tau) ).
	\end{equation}		
	In fact, 
	\begin{align*}
		\sup_{\pi \in \Pi } \mathcal{E}_{0, \tau}^g( Y_\tau( X^\pi_\tau) ) & = \sup_{\pi \in \Pi }\mathcal{E}^g_{0,\tau} \left( \esssup_{\pi^\prime \in \Theta_{0, \tau}(\pi) }\mathcal{E}_{\tau, T}^g\left( X^{\pi^\prime}_T + \xi \right) \right)\\
				                                                           & = \sup_{\pi \in \Pi }\sup_{\pi^\prime \in \Theta_{0, \tau}(\pi) }\mathcal{E}_{0, T}^g\left( X_T^{\pi^\prime} + \xi \right)
		                                                                                          = V(x),
	\end{align*}
	where we used monotonicity and flow property of the operators $\mathcal{E}^g_{s,t}(\cdot)$, $0 \le s\le t\le T$, see \cite[Proposition 3.6]{DHK1101}. By Theorem \ref{thm:to-robustprob} the proof is done.
\end{proof}

\subsection{Existence of a Saddle Point}
Considering the dual representation of $\mathcal{E}^g_{0,\tau} $ derived in Theorem \ref{thm:to-robustprob}, a pair $(\beta,q) \in \mathcal{D}\times\mathcal{Q} $ is said to be a subgradient of $\mathcal{E}^g_{0, \tau}$ at $Y_\tau(X^\pi_\tau)$ if
\begin{equation*}
	\mathcal{E}^g_{0,\tau}(Y_\tau(X^\pi_\tau)) = E_{Q^q}\left[D^{\beta}_{0,\tau} Y_\tau(X^\pi_\tau) + \int_0^\tau D^{\beta}_{0,u}g^\ast_u(\beta_u,q_u)\,du \right].
\end{equation*}
In the case where the generator only depends on $z$, equivalence between existence of a subgradient of a monetary utility function and quadratic growth of the driver $g$ was proved by \citet{Delbaen11}. 
The following result uses their compactness argument.
We will also need the conditions
\begin{enumerate}[label=(\textsc{Qg}),leftmargin=40pt]
	\item quadratic growth: $g:\mathbb{R}\times\mathbb{R}^d\to \mathbb{R}\cup\set{+\infty}$ and $\forall \eta >0$ there exists $C>0$: $g(y,z)\le C(1+ \abs{y} + \norm{z}^2)$ for all $y \in \mathbb{R}$: $\abs{y} \ge \eta$ and $z \in \mathbb{R}^d$.\label{qg}
\end{enumerate}
\begin{theorem}
\label{thm:subgradient}
	Assume that $g$ satisfies \ref{adm}, \ref{jconv}, \ref{lsc}, \ref{qg}, \ref{nor} and \ref{pos}.
	Then, $\mathcal{E}^g_0 $ admits a local subgradient: For any $[0,T]$-valued stopping time $\tau$ and any $\pi \in \Pi$,
	${\cal E}_{0, \tau}^g$ admits a subgradient $(q^\tau, \beta^\tau) \in \mathcal{D} \times \mathcal{Q} $ at $Y_{\tau}(X^\pi_{\tau})$. 
\end{theorem}
\begin{proof}	Let $\pi \in \Pi$ be fixed for the rest of the proof.
Let $\eta >0 $ in \ref{qg}.
			Due to Theorem \ref{thm:to-robustprob}, we have
			\begin{equation*}
				\mathcal{E}^g_{0 , \tau}( Y_{\tau}( X^\pi_{\tau} ) )  = \inf_{ \frac{dQ^q}{dP}D^\beta_{0,\tau} \in \mathcal{K} }\Set{ E_{Q^q}\left[ D^\beta_{0,\tau} Y_{\tau}( X^\pi_{\tau} ) + \int_0^{\tau}D^\beta_{0,u}g^\ast_u(\beta_u,q_u)\,du \right] },
			\end{equation*}
		where 
		\begin{equation*}
			\mathcal{K} := \Set{ \frac{dQ^q}{dP}D^\beta_{0,\tau}: (\beta,q) \in \mathcal{D}\times \mathcal{Q}} \subseteq L^1.
		\end{equation*}
		For every $k \ge 0$ the set 
		\begin{equation}
		\label{eq:set-Gamma}
			\Gamma_{\tau} := \Set{\frac{dQ^q}{dP}D^\beta_{0,\tau}\in \mathcal{K}: E_{Q^q}\left[ \int_0^{\tau}D^\beta_{0,u}g^\ast_u(\beta_u,q_u)\,du \right]\le k}
		\end{equation}
		is convex, see \cite{tarpodual}. 
		Let us show that it is $\sigma(L^1, L^\infty)$-compact.
		Let $(\beta, q) \in \mathbb{R}\times\mathbb{R}^d $ be given. 
		By definition, we have
		\begin{align*}
			g^\ast(\beta, q) & =  \sup_{ y  \in \mathbb{R} , z \in \mathbb{R}^d}\Set{ \beta y + q z - g(y,z) }\\ 
							& \ge \sup_{ \abs{y}  \ge \eta, z \in \mathbb{R}^d}\Set{ \beta y + q z - g(y,z) }\\
                         	 & \ge \sup_{ \abs{y}  \ge \eta, z \in \mathbb{R}^d}\Set{ \beta y + q z -  C(1 +  y  + \norm{z}^2 )}\\
					     	 & \ge \sup_{ \abs{y}  \ge \eta}\Set{ \beta y - Cy } +  b\norm{q}^2  -C,
		\end{align*}
		with $b = \frac{1}{4C}$.
		If $\abs{\beta} > C $, then let $n \in \mathbb{N}$ be big enough such that $y  := n\beta$ satisfies $\abs{y} \ge \eta$. Then,
		\begin{equation*}
			\sup_{\abs{y} \ge \eta}\Set{ -\beta y - C \abs{y} } \ge n\abs{\beta}\left( \abs{\beta} -C \right),
		\end{equation*} 
		so that $g^\ast(\beta,q) = \infty$. 
		Therefore, we can restrict ourselves to $(\beta,q)\in {\cal D}\times {\cal Q}$ with $\abs{\beta} \le C$.
		Hence, we can find a positive constant $a$ such that
		\begin{equation}
		\label{eq:g_star_abc}
			g^\ast(\beta, q) \ge a\beta + b\norm{ q }^2 - C.
		\end{equation}
		Since $\beta$ is bounded, $D^\beta_{0,u} = e^{-\int_0^u\beta_r\,dr}$ is bounded as well.
		Thus multiplying both sides of \eqref{eq:g_star_abc} by $D^\beta_{0,u}$ and integrating with respect to $Q^q\otimes dt$ lead to
		\begin{equation*}
			E_{Q^q}\left[ \int_0^{\tau}D^\beta_{0,u}g^\ast(\beta_u,q_u)\,du \right] \ge A_1 + A_2E_{Q^q}\left[ \int_0^{\tau}\norm{ q_u }^2\,du \right],
		\end{equation*}
    	where $A_1$ and $A_2$ are positive constants which do not depend on $\beta$ and $q$.
    	Arguing similar to the proof of \citep[Theorem 2.2]{Delbaen11}, we can find a positive constant $c$ such that 
    	\begin{multline*}
    		\Set{ {\frac{dQ^q}{dP}D^\beta_{0,\tau}\in \mathcal{K}: E_{Q^q}\left[ \int_0^{\tau}D^\beta_{0,u}g^\ast_u(\beta_u,q_u)\,du \right]\le k} }\\ \subseteq \Set{ \frac{dQ^q}{dP}D^\beta_{0,\tau}\in {\cal K}: E\left[ \frac{dQ^q}{dP}\log\frac{dQ^q}{dP} \right] \le c }
    	\end{multline*} 
    	and therefore, we can conclude using the de la Vall\'ee Poussin theorem that the left hand side in the above inclusion is $L^1$- uniformly integrable.
	 	We take a maximizing sequence $(\frac{dQ^{q^n}}{dP}D^{\beta^n}_{0,\tau} )_n$ for the functional $\mathcal{E}^g_{0 ,\tau}(Y_{\tau}(X^\pi_{\tau})) $. 
	 	Since $ Y_{\tau}(X^\pi_{\tau}) $ is positive, it follows that the sequence $( E_{Q^{q^n}} [ \int_0^{\tau}D^{\beta^n}_{0,u}g^\ast_u(\beta^n_u,q^n_u)\,du ] )_n$ admits a subsequence which is bounded from above. 
	 	Therefore, the previous step shows that the sequence $( \frac{dQ^{q^n}}{dP}D^{\beta^n}_{0,\tau} )_n$ is uniformly integrable.
	 	In addition, applying a compactness argument of Komlos type, we can find a sequence denoted $( \tilde{M}^{n}_T ) $ in the asymptotic convex hull of $ (\frac{dQ^{q^n}}{dP}D^{\beta^n}_{0,\tau} )_n $ which converges $P$-a.s. to the limit $M_T \in L^0_+$. 
	 	The sequence $(\tilde{M}^{n}_T)$ is as well uniformly integrable and therefore converges to $M_T$ in $L^1$.
	 	By the arguments used in the proof of \citep[Theorem 3.10]{tarpodual}, it is possible to show that for all $n \in \mathbb{N} $ there exist $\tilde{q}^{n}$ and $\tilde{\beta}^{n}$ such that $\tilde{M}^{n}_T = \frac{ dQ^{ \tilde{q}^{n} } }{dP} D^{ \tilde{ \beta}^{n} }_{0,\tau} $ and, up to other convex combinations, the sequences  $(\tilde{q}^{n} ) $ and $(\tilde{\beta}^{n})$ converge $P \otimes dt $-a.s. to some $q^\tau$ and $\beta^\tau$, respectively and $M_T = \frac{ dQ^{ q^\tau}}{dP} D^{  \beta^\tau  }_{0,\tau} $.
	 	Since $\vert\tilde{\beta}^{n}\vert \le C$ for all $n$, it holds $\abs{\beta^\tau} \le C$.
	 	By Fatou's lemma and convexity, we have
	 	\begin{align*}
	 		 \mathcal{E}^g_{0 ,\tau}( Y_{\tau}(X^\pi_{\tau}) ) &= \liminf_{n \rightarrow \infty} E_{Q^{q^n}}\left[D^{\beta^n}_{0,\tau} Y_{\tau}(X^\pi_{\tau}) + \int_0^{\tau} D^{\beta^n}_{0,u}g^\ast(\beta^n_u, q^n_u)\,du \right]\\
	 									                              & \ge E\left[ \liminf_{n \rightarrow \infty} \frac{dQ^{\tilde{q}^{n}}}{dP}\left( D^{\tilde{\beta}^{n}}_{0,\tau} Y_{\tau}(X^\pi_{\tau}) + \int_0^{\tau}D^{\tilde{\beta}^{n}}_{0,u}g^\ast(\tilde{\beta}^{n}_u, \tilde{q}^{n}_u)\,du \right) \right].
	 	\end{align*}
		Lower-semicontinuity of $g^\ast$ yields 
		\begin{equation*}
			\mathcal{E}^g_{0,\tau}( Y_{\tau}(X^\pi_{\tau}) ) \ge E_{Q^{q^\tau}}\left[ D^{\beta^\tau}_{0,\tau} Y_{\tau}(X^\pi_{\tau}) +  \int_0^{\tau} D^{\beta^\tau}_{0,u}g^\ast(\beta^\tau_u, q^\tau_u)\,du \right].
	 	\end{equation*}	 
	 	Since $\abs{\beta^\tau} \le C$ and $M_T  \in L^1$, we have $\beta^\tau \in {\cal D}$ and $q^\tau \in {\cal Q}$.
\end{proof}
\begin{corollary}
\label{thm:saddlepoint}
	Under the assumptions of Theorem \ref{thm:subgradient}, for any optimal strategy $\pi^\ast \in \Pi$ and any $[0,T]$-valued stopping time $\tau$ one has
	\begin{equation}
	\label{eq:local_optimum}
		V(x) = \mathcal{E}^g_{0, \tau} \left( Y_{\tau}(X^{\pi^\ast}_{\tau}) \right).
	\end{equation}
	In addition, Problem \eqref{eq:optprob} admits a local saddle point in the sense that, there exists $(\beta^\tau,q^\tau) \in \mathcal{D} \times \mathcal{Q}$ satisfying
	\begin{align*}
		V(x) & = E_{Q^{q^\tau}}\left[ D^{\beta^\tau}_{0,\tau}Y_{\tau}(X^{\pi^\ast}_{\tau}) + \int_0^{\tau}D^{\beta^\tau}_{0,u}g^\ast_u(\beta^\tau_u, q^\tau_u)\, du  \right]\\
		     & = \inf_{(\beta,q) \in \mathcal{D} \times \mathcal{Q} }\sup_{\pi \in \Pi}E_{Q^{q}}\left[ D^{\beta}_{0,\tau}Y_{\tau}(X^{\pi}_{\tau}) + \int_0^{\tau}D^{\beta}_{0,u}g^\ast_u(\beta_u, q_u)\, du  \right].
	\end{align*}
\end{corollary}
\begin{proof}
	By definition of $Y_{\tau}(X^{\pi^\ast}_{\tau}) $, monotonicity and the flow property of $\mathcal{E}^g_{s,t}$; $0 \le s\le t\le T$, we have
	\begin{align*}
		\mathcal{E}^g_{0,\tau}(Y_{\tau}(X^{\pi^\ast}_{\tau}) ) & = \mathcal{E}^g_{0,\tau}\left( \esssup_{\pi \in \Theta_{0,\tau}(\pi^\ast) } \mathcal{E}_{\tau,T}^g \left( X^{\pi }_T + \xi \right) \right)\\
																	 & = \sup_{\pi \in \Theta_{0, \tau}(\pi^\ast)}\mathcal{E}^g_{0, T}\left( X^{\pi}_T + \xi \right) \ge V(x),
	\end{align*} 
	since $\pi^\ast \in \Theta_{0, \tau}(\pi^\ast)$.
	Thus, Equation \eqref{eq:local_optimum} is a consequence of Equation \eqref{eq:localproblem}. 

	It follows from Theorem \ref{thm:subgradient} and Equation \eqref{eq:local_optimum} that there exists $(\beta^\tau, q^\tau) \in \mathcal{D} \times \mathcal{Q}$ such that
	\begin{equation}
		V(x) = E_{Q^{q^\tau}}\left[ D_{0,T}^{\beta^\tau} Y_{\tau}(X^{\pi^\ast}_{\tau}) + \int_0^{\tau} D_{0,u}^{\beta^\tau}g^\ast_u(\beta^\tau_u,q^\tau_u)\,du \right]
	\end{equation}
	and	for every $\pi \in \Pi$ exists $(\beta(\pi), q(\pi)) \in \mathcal{D} \times \mathcal{Q}$ such that
	\begin{equation*}
		\mathcal{E}^g_{0, \tau}\left( Y_{\tau}(X^{\pi}_{\tau}) \right) = E_{Q^{q(\pi)}}\left[ D_{0,\tau}^{\beta(\pi)} Y_{\tau}(X^{\pi}_{\tau}) + \int_0^{\tau} D_{0,u}^{\beta(\pi)}g^\ast_u(\beta_u(\pi),q_u(\pi))\,du \right].
	\end{equation*}
	Thus, taking the supremum with respect to $\pi$ on both sides yields
	\begin{align*}
		\mathcal{E}^g_{0, \tau}\left(Y_{\tau}(X^{\pi^\ast}_{\tau}) \right) &= \sup_{\pi \in \Pi}E_{Q^{q(\pi)}}\left[ D_{0,\tau}^{\beta(\pi)} Y_{\tau}(X^{\pi}_{\tau}) + \int_0^{\tau}D_{0,u}^{\beta(\pi)}g^\ast_u(\beta_u(\pi),q_u(\pi))\,du \right]\\
                                            & \geq \inf_{(\beta, q)\in \mathcal{D}\times \mathcal{Q}} \sup_{\pi \in \Pi} E_{Q^{q}}\left[ D_{0,\tau}^{\beta} Y_{\tau}(X^{\pi}_{\tau}) + \int_0^{\tau}D_{0,u}^{\beta}g^\ast_u(\beta_u,q_u)\,du \right].
	\end{align*}
	Since we always have $\inf\sup \geq \sup\inf$, it follows that
	\begin{align*}
		\sup_{\pi \in \Pi}\inf_{(\beta, q) \in \mathcal{D}\times  \mathcal{Q}}&E_{Q^q}\left[D_{0,\tau}^{\beta} Y_{\tau}(X^{\pi}_{\tau}) + \int_0^{\tau}D_{0,u}^{\beta}g^\ast_u(\beta_u, q_u)\,du \right]\\
 		&= E_{Q^{q^\tau}}\left[ D_{0,\tau}^{\beta^\tau} Y_{\tau}(X^{\pi^\ast}_{\tau}) + \int_0^{\tau}D_{0,u}^{\beta^\tau}g^\ast_u(\beta^\tau_u,q^\tau_u)\,du \right]\\
 		& \quad \quad \quad  = \inf_{(\beta, q)\in \mathcal{D}\times \mathcal{Q}} \sup_{\pi \in \Pi} E_{Q^{q}}\left[ D_{0,\tau}^{\beta} Y_{\tau}(X^{\pi^\ast}_{\tau}) + \int_0^{\tau}D_{0,u}^{\beta}g^\ast_u(\beta_u,q_u)\,du \right].
	\end{align*}
	The proof is complete.
\end{proof}
 \begin{remark}
\label{rem:openinterval}
	If $g$ defined on the space $\mathbb{R}\times \mathbb{R}^d $ satisfies \ref{adm}, \ref{jconv}, \ref{nor}, \ref{pos} and \ref{qg}, then one can take $\tau = T$ in Equation \eqref{eq:Ytau_optimal}, that is $Y_\tau(X_\tau^\pi) = X^{\pi}_T + \xi $, and work on the whole time interval $[0,T] $ in the proof of Theorem \ref{thm:subgradient} and the subsequent corollary. 
	The main reason for working with stopping times is to allow for generators that satisfy the conditions \ref{jconv}, \ref{nor} and \ref{pos} only on a subset $I \times \mathbb{R}^d$, where $I \subseteq \mathbb{R}_+ $ is an open interval as in the following example. 
\end{remark}
\begin{example}[Certainty equivalent]
	Let us come back to the certainty equivalent example of Section \ref{sec:setting}.
	For $u(x) = \log(x) $, Equation \eqref{eq:u_concave} becomes
	\begin{equation*}
		Y_t = X - \frac{1}{2}\int_t^T\frac{\abs{Z_u}^2}{Y_u}\,du + \int_t^T Z_u\,dW_u, \quad t \in [0,T].		
	\end{equation*}
	The generator $g (y,z) = \frac{1}{2}\abs{z}^2/y $ satisfies \ref{lsc}, \ref{jconv}, \ref{nor} and \ref{pos} on $(0,\infty) \times \mathbb{R}^d$ and it can be extended on $\mathbb{R}_+\times \mathbb{R}^d$ to a generator satisfying the same conditions by putting
	\begin{equation*}
		g(y,z) = \begin{cases}
		           \frac{1}{2}\frac{\abs{z}^2}{y} &\text{ if } y>0\\
		           0                              &\text{ if } z = 0\\
		           + \infty                       &\text{ if } y = 0, z \neq 0.
		\end{cases}
	\end{equation*}
	Hence, Theorem \ref{existence} ensures the existence of an optimal trading strategy $\pi^\ast \in \Pi$. 
	However, if we consider the function on $\mathbb{R}_+ \times \mathbb{R}^d$, we can not guarantee, with our method, that the set $\Gamma_\tau$ defined in \eqref{eq:set-Gamma} is weakly compact and therefore that the problem admits a saddle point. 
	A way around is to introduce a stopping time $0< \tau \le T$ and work locally on $[0,\tau]$ as follows:
	Let $\pi^\ast \in \Pi$ be an optimal strategy and put $Y^{\pi}_t := u^{-1}(E[u(X^{\pi}_T+\xi)\mid {\cal F}_t])$. 
	Since $x >0$, there exists $m \in \mathbb{N}$ such that $x \ge \frac{1}{m}$.
	Define the stopping time $\tau $ by
	\begin{equation*}
		\tau : = \inf\Set{ t \ge 0: X^{\pi^\ast}_t \leq \frac{1}{m} }\wedge T.
	\end{equation*} 
	We can restrict the study to subsolutions $(Y,Z) \in \mathcal{A}^u(X^{\pi}_T + \xi) $ satisfying $Y \ge X^\pi$, for all $t \in [0,T]$.
	Hence, with BSDE duality we have $Y_{\tau \wedge t}^{\pi^\ast} \ge X^{\pi^\ast}_{\tau \wedge t} \ge \frac{1}{m} $. 
	Applying martingale representation theorem and It\^o's formula such as in Example \ref{example1}, we can find a process $Z^{\pi^\ast} \in {\cal L}^1$ such that
	\begin{equation*}
		Y^{\pi^\ast}_t = Y^{\pi^\ast}_\tau - \int_t^\tau g_u(Y^{\pi^\ast}_u, Z^{\pi^\ast}_u)\,du + \int_t^\tau Z^{\pi^\ast}_u\,dW_u \quad \text{on } \set{t\le \tau}.
	\end{equation*}
	Since the set $\set{Y_\tau: (X,Y, Z) \in {\cal A}(x)}$ is upward directed, using the arguments of Theorem \ref{existence} we can find a strategy $\bar{\pi} \in \Pi$ such that
	\begin{align*}
		Y_\tau(X^{\pi^\ast}_\tau) &:= \esssup_{\pi^\prime \in \Theta_{0,\tau}(\pi^\ast)}{\cal E}_{\tau,T}(X^{\pi^\prime}_T + \xi)
						    = Y^{\bar{\pi}}_\tau = {\cal E}^g_{\tau,T}(X^{\bar{\pi}}_T + \xi)
	\end{align*}
	with $\bar{\pi} \in \Theta_{0,\tau}(\pi^\ast)$, i.e. $\bar{\pi}1_{[0,\tau]} = \pi^\ast1_{[0, \tau]}$ and $\bar{\pi} \in \Pi$.
	Moreover, since $\Theta_{0,\tau}(\pi^\ast) = \Theta_{0,\tau}(\bar{\pi})$, we have $Y_\tau(X^{\pi^\ast}_\tau) = Y_\tau(X^{\bar{\pi}}_\tau) $.
	By $Y^{\bar{\pi}}_{t\wedge\tau} \ge X^{\bar{\pi}}_{t\wedge \tau} = X^{\pi^\ast}_{t\wedge\tau} \ge \frac{1}{m}>0 $, we also get
	\begin{align*}
		Y^{\bar{\pi}}_0 &= Y^{\bar{\pi}}_\tau - \int_0^\tau g_u(Y^{\bar{\pi}}_u, Z^{\bar{\pi}}_u)\,du + \int_0^\tau Z^{\bar{\pi}}_u\,dW_u\\
		 		&= Y_\tau(X_\tau^{\bar{\pi}}) - \int_0^\tau g_u(Y^{\bar{\pi}}_u, Z^{\bar{\pi}}_u)\,du + \int_0^\tau Z^{\bar{\pi}}_u\,dW_u.
	\end{align*}
	For almost every $(\omega,t)$ such that $t\le \tau(\omega)$ the function $g$ is differentiable at $(Y^{\bar{\pi}}_t(\omega), Z^{\bar{\pi}}_t(\omega))$ and it admits a unique subgradient $(\bar{\beta}_t(\omega), \bar{q}_t(\omega))$ given by
	\begin{equation*}
		\bar{q}_t = \frac{Z^{\bar{\pi}}_t}{Y^{\bar{\pi}}_t} \quad \text{and}\quad \bar{\beta}_t =-\frac{\abs{Z^{\bar{\pi}}_t}^2}{2(Y^{\bar{\pi}}_t)^2}\quad \text{on } \set{t\le \tau}.
	\end{equation*}
	Since $Y^{\bar{\pi}}_{t\wedge\tau} \ge 1/m$ and $Z^{\bar{\pi}} \in {\cal L}^1$, it follows that $(\bar{\beta}, \bar{q}) \in  {\cal D}\times {\cal Q} $ and we have $g_t(Y^{\bar{\pi}}_t, Z^{\bar{\pi}}_t ) = \bar{\beta}_tY^{\bar{\pi}}_t + \bar{q}_tZ^{\bar{\pi}}_t- g^\ast_t(\bar{\beta}_t, \bar{q}_t)$. 
	Thus, using the arguments leading to Equation \ref{eq:g-expectation}, one has
	\begin{equation}
	\label{eq:examp_grad1}
		Y^{\bar{\pi}}_0 = E_{Q^{\bar{q}}}\left[ D^{\bar{\beta}}_{0, \tau}Y_\tau(X^{\bar{\pi}}_\tau) + \int_0^\tau g^\ast_u(\bar{\beta}_u, \bar{q}_u)\,du \right].
	\end{equation}
	But since for every $(\beta, q)\in {\cal D}\times {\cal Q} $ it holds
	\begin{equation*}
		Y^{\bar{\pi}}_0 \le E_{Q^{q}}\left[ D^{\beta}_{0, \tau}Y_\tau(X^{\bar{\pi}}_\tau) + \int_0^\tau g^\ast_u(\beta_u, q_u)\,du \right],
	\end{equation*}
	it follows,
	\begin{align}
	\nonumber	Y^{\bar{\pi}}_0  &= \inf_{(\beta, q)\in {\cal D}\times {\cal Q}}E_{Q^{q}}\left[ D^{\beta}_{0, \tau}Y_\tau(X^{\bar{\pi}}_\tau) + \int_0^\tau g^\ast_u(\beta_u, q_u)\,du \right]\\
	\label{eq:examp_grad2}					 &= {\cal E}^g_{0,\tau}(Y_\tau(X^{\bar{\pi}}_\tau))
	\end{align}
	where the second equality above follows from the representation theorem \ref{thm:to-robustprob}.
	By the identity $Y_\tau(X^{\pi^\ast}_\tau) = Y_\tau(X^{\bar{\pi}}_\tau) $, one has ${\cal E}^g_{0,\tau}(Y_\tau(X^{\bar{\pi}}_\tau)) = {\cal E}^g_{0,\tau}(Y_\tau(X^{\pi^\ast}_\tau)) $, so that it follows from the equations \eqref{eq:examp_grad1} and \eqref{eq:examp_grad2} that ${\cal E}^g_{0,\tau}(Y_\tau(X^{\pi^\ast}_\tau))$ admits the subgradient $(\bar{\beta}, \bar{q})$.
	Therefore, the utility maximization problem $V(x) = \sup_{\pi \in \Pi} C_0(X^\pi_T + \xi) $ can be written as a robust control problem admitting a local saddle point in the sense of Corollary \ref{thm:saddlepoint}.   
	In fact,
	\begin{align*}
		& \inf_{(\beta, q) \in {\cal D}\times {\cal Q}}\sup_{\pi\in \Pi}E_{Q^q}\left[D^\beta_{0,\tau}Y_\tau(X^\pi_\tau) + \int_0^\tau g^\ast_u(\beta_u,q_u)\,du \right]\\
		 &\qquad \qquad \le \sup_{\pi \in \Pi}E_{Q^{\bar{q}}}\left[ D^{\bar{\beta}}_{0,\tau}Y_\tau(X^\pi_\tau) + \int_0^\tau g^\ast_u(\bar{\beta}_u,\bar{q}_u)\,du \right]\\
		 &\qquad \qquad \le {\cal E}^g_{0,\tau}(Y_\tau(X^{\pi^\ast}_\tau)) = E_{Q^{\bar{q}}}\left[ D^{\bar{\beta}}_{0, \tau}Y_\tau(X^{\pi^\ast}_\tau) + \int_0^\tau g^\ast_u(\bar{\beta}_u, \bar{q}_u)\,du \right]\\
		 &\qquad \qquad \le \inf_{(\beta, q) \in {\cal D}\times {\cal Q}}E_{Q^q}\left[D^\beta_{0,\tau}Y_\tau(X^{\pi^\ast}_\tau) + \int_0^\tau g^\ast_u(\beta_u,q_u)\,du \right]\\
		 &\qquad \qquad \le \sup_{\pi \in \Pi}\inf_{(\beta, q) \in {\cal D}\times {\cal Q}}E_{Q^q}\left[D^\beta_{0,\tau}Y_\tau(X^{\pi}_\tau) + \int_0^\tau g^\ast_u(\beta_u,q_u)\,du \right]\\
		 &\qquad \qquad \le \inf_{(\beta, q) \in {\cal D}\times {\cal Q}}\sup_{\pi\in \Pi}E_{Q^q}\left[D^\beta_{0,\tau}Y_\tau(X^\pi_\tau) + \int_0^\tau g^\ast_u(\beta_u,q_u)\,du \right].
	\end{align*}
To justify the second inequality above, notice that with the arguments leading to \eqref{eq:weakdual}, we have
\begin{equation*}
	{\cal E}^g_{\tau, T}(X^\pi_T + \xi) \le E_{Q^{\bar{q}}}\left[ D^{\bar{\beta}}_{\tau, T} (X^\pi_T + \xi) + \int_\tau^T g^\ast_u(\bar{\beta}_u, \bar{q}_u)\,du \mid {\cal F}_\tau \right].
\end{equation*}
Therefore,
\begin{align*}
	&\sup_{\pi \in \Pi}E_{Q^{\bar{q}}}\left[ D^{\bar{\beta}}_{0,\tau}Y_\tau(X^\pi_\tau) + \int_0^\tau g^\ast_u(\bar{\beta}_u,\bar{q}_u)\,du \right]\\
	& = \sup_{\pi \in \Pi}E_{Q^{\bar{q}}}\left[ D^{\bar{\beta}}_{0,\tau}\esssup_{\pi'\in \Theta_{0,\tau}(\pi)}{\cal E}^g_{\tau,T}(X^{\pi'}_T + \xi) + \int_0^\tau g^\ast_u(\bar{\beta}_u,\bar{q}_u)\,du \right]\\
	& =  \sup_{\pi \in \Pi}E_{Q^{\bar{q}}}\left[ D^{\bar{\beta}}_{0,\tau}{\cal E}^g_{\tau,T}(X^{\pi}_T + \xi) + \int_0^\tau g^\ast_u(\bar{\beta}_u,\bar{q}_u)\,du \right]\\
	& \le E_{Q^{\bar{q}}}\left[D^{\bar{\beta}}_{0,\tau}E_{Q^{\bar{q}}}\left[ D^{\bar{\beta}}_{\tau, T} (X^\pi_T + \xi) + \int_\tau^T g^\ast_u(\bar{\beta}_u, \bar{q}_u)\,du \mid {\cal F}_\tau \right] + \int_0^\tau g^\ast_u(\bar{\beta}_u,\bar{q}_u)\,du \right] \\
	&\le E_{Q^{\bar{q}}}\left[ D^{\bar{\beta}}_{0,T} (X^{\pi}_T + \xi) + \int_0^T g^\ast_u(\bar{\beta}_u,\bar{q}_u)\,du \right]\\
	&\le V(x) =  E_{Q^{\bar{q}}}\left[ D^{\bar{\beta}}_{0,\tau} Y_\tau(X_\tau^{\pi^\ast}) + \int_0^\tau g^\ast_u(\bar{\beta}_u,\bar{q}_u)\,du \right].
\end{align*}
\end{example}

\subsection{Characterization} 
\label{sec:charac}
We conclude this section by providing a characterization of an optimal trading strategy and a corresponding optimal model in the framework of the stochastic maximum principle.
It dates back to the work of Bismut in the 1970s.
The maximum principle has been widely used in the context of expected utility maximization to characterize optimal strategies, see for instance \citet{Hor-etal}.
Applying the perturbation techniques yielding the stochastic maximum principle as developed by \citet{PengFBSDE} to the control problem \eqref{eq:optprob} as it is does not give much information on the optimal solution because of the nonlinearity of the operator $\mathcal{E}^g_0 $.
This is where the dual representation for BSDEs becomes useful, in helping to linearize the problem by transforming it into a robust control problem under a linear operator.
In the following we denote by $\partial g^\ast/\partial a$ and $\partial g^\ast/\partial b$, when they exist, the derivative of the function $g^\ast:\mathbb{R}\times \mathbb{R}^d\to \mathbb{R}$ with respect to the first and the second variable, respectively.

Since for every $\pi \in \Pi$ the process $X^\pi$ is a positive $Q$-martingale, we can write $X^\pi$ as
\begin{equation}
\label{eq:divide}
	X^\pi_t = x + \int_0^t\sigma_u\tilde{\pi}_uX^\pi_u\,dW^Q_u
\end{equation}
for some predictable process $\tilde{\pi}$ satisfying $\set{\int_0^T\abs{\sigma_u\tilde{\pi}_u}^2\,du < \infty}=\set{X^\pi_T >0}$, see \cite[Chapter 1]{Kaz}.
The next theorem gives a characterization of the optimal model $(q^\ast, \beta^\ast)$ and of the process $\tilde{\pi}^\ast$ associated to the optimal strategy $\pi^\ast$.
\begin{theorem}
\label{thm:characterization}
	Assume that the driver $g$ is strictly convex, satisfies \ref{adm}, \ref{lsc}, \ref{nor}, \ref{pos},  and \ref{qg}.
	Further assume that $\xi \in L^\infty_+$.
	Then, for every saddle point $(\pi^\ast, (\beta^\ast, q^\ast))$ there exists a pair $(p,k)$ depending on $ \tilde{\pi}^\ast, \beta^\ast$ and $q^\ast$ such that $p_t\theta_t + p_tq^\ast_t + k_t = 0$ $P\otimes dt$~-a.s.	
	 and which solves the BSDE
	\begin{equation*}
		dp_t = -(\theta_tp_t +p_tq^\ast_t + k_t)\tilde{\pi}^\ast_t\sigma_t\,dt + k_t\,dW^{Q^{q^\ast}}_t, \quad p_T = D^{\beta^\ast}_{0,T}\quad  Q^{q^\ast}\text{a.s.}
	\end{equation*}	
	Furthermore, $g^\ast$ is differentiable at $(\beta^\ast,q^\ast)$ and satisfies
	\begin{equation}
	\label{eq:characdual}
		- \frac{\partial g^\ast_t}{\partial a}(\beta^\ast_t,q^\ast_t) + Y_t = 0 \quad \text{ and }\quad	-\frac{\partial g^\ast_t}{\partial b }(\beta^\ast_t,q^\ast_t) + Z_t = 0; \quad P\otimes dt \text{-a.s.},
	\end{equation}
	where $(Y,Z)$ solves the BSDE
	\begin{equation}
	\label{eq:BSDE}
		dY_t = g(Y_t,Z_t)\,dt - Z_t\,dW_t, \quad Y_T = X^{\pi^\ast}_T + \xi.
	\end{equation}
\end{theorem}
\begin{proof}
	By assumptions and Remark \ref{rem:openinterval} the control problem admits a saddle point $(\pi^\ast, (\beta^\ast, q^\ast))$, that is,
	\begin{align}
		\label{eq:Xast_int} V(x) = & E_{Q^{q^\ast}}\left[ D_{0,T}^{\beta^\ast} (X^{\pi^\ast}_T + \xi )  + \int_0^T D_{0,u}^{\beta^\ast}g^\ast_u(\beta^\ast_u, q^\ast_u)\,du \right]\\
		\nonumber	 = & \inf_{(\beta, q) \in \mathcal{D}\times \mathcal{Q}} \sup_{\pi \in \Pi} E_{Q^q}\left[ D_{0,T}^{\beta} (X^{\pi}_T + \xi ) + \int_0^T D_{0,u}^{\beta}g^\ast_u(\beta_u, q_u)\,du \right].
	\end{align}
	It follows from \eqref{eq:Xast_int} that $X^{\pi^\ast}_T$ is $Q^{q^\ast}$-integrable.
	Put
	\begin{equation*}
		Y_t := \essinf_{(\beta, q) \in \mathcal{D}\times \mathcal{Q}}  E_{Q^q}\left[ D_{t,T}^{\beta} (X^{\pi^\ast}_T + \xi ) + \int_t^T D^\beta_{t,u}g^\ast_u(\beta_u, q_u)\,du \mmid \mathcal{F}_t \right], \quad t \in [0,T].
	\end{equation*}
	By \citep[Corollary 4.3]{tarpodual}, for all $t \in [0,T]$, we have
	\begin{equation*}
		Y_t :=  E_{Q^{q^\ast}}\left[ D_{t,T}^{\beta^\ast} (X^{\pi^\ast}_T + \xi ) + \int_t^T D^{\beta^\ast}_{t,u}g^\ast_u(\beta^\ast_u, q^\ast_u)\,du \mmid \mathcal{F}_t \right]
	\end{equation*}
	so that applying martingale representation theorem and It\^o's formula, we can find a predictable process $Z$ such that $(Y,Z)$ solves the linear BSDE
			\begin{equation*}
		dY_t = \left( \beta^\ast_t Y_t + q^\ast_tZ_t -g^\ast_t(\beta^\ast_t, q^\ast_t) \right)\,dt - Z_t\,dW_t, \quad Y_T = X^{\pi^\ast}_T + \xi.
	\end{equation*}
	Moreover, by \citep[Theorem 4.6]{tarpodual}, for almost every $(\omega,t) $, the subgradients $\partial g(\omega,t,Y_t,Z_t) $ with respect to $(Y_t, Z_t)$ contain $(\beta^\ast_t,q^\ast_t)$.
	Hence, $(Y,Z)$ also solves the BSDE \eqref{eq:BSDE}.

	\begin{enumerate}[label =  ,fullwidth]
	\item \emph{Characterization of} $\tilde{\pi}^\ast$:
		For any $\pi \in \Pi$ define
	\begin{equation*}
		Y^\pi_t := E_{Q^{q^\ast}}\left[ D_{t,T}^{\beta^\ast} (X^{\pi}_T + \xi ) + \int_t^T D_{t,u}^{\beta^\ast}g^\ast_u(\beta^\ast_u, q^\ast_u)\,du \mmid \mathcal{F}_t \right], \quad t \in [0,T].
	\end{equation*}
	It follows from the saddle point property that
	\begin{equation*}
		V(x) = \sup_{\pi \in \Pi} Y^\pi_0 = Y^{\pi^\ast}_0.
	\end{equation*}
	Let $\pi \in \Pi$ be a bounded strategy such that for every $\varepsilon \in ( 0,1 )$, $\pi^\ast + \varepsilon \pi \in \Pi$ and let $\tilde{\pi}$ be the process associated to $\pi$, see \eqref{eq:divide}. 
	Then, by optimality of $\pi^\ast$, 
	\begin{equation*}
		0 = \lim_{\varepsilon \rightarrow 0}\frac{1}{\varepsilon}\left( Y_0^{\pi^\ast + \varepsilon\pi} - Y_0^{\pi^\ast} \right) = E_{Q^{q^\ast}}\left[ D^{\beta^\ast}_{0,T}\eta_T \right],
	\end{equation*}
	where $\eta_t := \lim_{\varepsilon \to 0}\frac{1}{\varepsilon}\left( X^{\pi^\ast + \varepsilon\pi}_t - X^{\pi^\ast}_t \right)$ solves the SDE
	\begin{align*}
		d\eta_t &= \theta_t\left( \tilde{\pi}^\ast_t\sigma_t\eta_t + X^{\pi^\ast}_t\sigma_t \tilde{\pi}_t \right)\,dt + \left( \tilde{\pi}^\ast_t\sigma_t\eta_t + X^{\pi^\ast}_t\sigma_t\tilde{\pi}_t \right)\,dW_t\\
		        &= \alpha_t(\theta_t + q^\ast_t)\,dt + \alpha_t\,dW_t^{Q^{q^\ast}}, \quad \eta_0 = 0\quad  Q^{q^\ast}\text{-a.s.}
	\end{align*}
	with $\alpha_t = ( \tilde{\pi}^\ast_t\sigma_t\eta_t + X^{\pi^\ast}_t\sigma_t \tilde{\pi}_t )$.
	In fact, this follows by applying the dominated convergence theorem \citep[Theorem IV.32]{Pro}, since 
	\begin{equation*}
		X^{\pi^\ast+\varepsilon\pi}_t = x{\cal E}\left( \tilde{\pi}^\ast+\varepsilon\tilde{\pi} dW^Q\right)_t \le x\exp\left( \int_0^t\tilde{\pi}_u^\ast\,dW^Q_u + \left\vert \int_0^t\tilde{\pi}_u\,dW_u^Q \right\vert + \frac{1}{2}\int_0^t(\tilde{\pi}_u^\ast + \tilde{\pi}_u)^2\,du \right),
	\end{equation*}
	where $dQ/dP = {\cal E}(-\int\theta\,dW)_T$.
	Let $(p,k)$ be the solution of the linear BSDE with bounded terminal condition
	\begin{equation*}
		dp_t = -(\theta_tp_t +  q^\ast_tp_t + k_t)\tilde{\pi}^\ast_t\sigma_t\,dt + k_t\,dW^{Q^{q^\ast}}_t, \quad p_T = D^{\beta^\ast}_{0,T} \quad  Q^{q^\ast}\text{-a.s.}
	\end{equation*}
	which is known as the adjoint equation.
	Observe that since $\beta^\ast \in {\cal D}$, $D^{\beta^\ast}_{0,T}$ is bounded.	
	Applying It\^o's formula to $\eta_tp_t$ yields
	\begin{equation}
	\label{eq:itoadj}
		\eta_tp_t = \int_0^tX^{\pi^\ast}_u\tilde{\pi}_u\sigma_u(p_u\theta_u +  p_uq^\ast_u + k_u)\,du + \int_0^t\left\{ \eta_uk_u + p_u(\tilde{\pi}^\ast_u\sigma_u\eta_u + X^{\pi^\ast}_u\sigma_u\tilde{\pi}_u) \right\}\,dW_u^{Q^{q^\ast}}.	
	\end{equation}
	Since we cannot ensure that the second term of the left hand side of Equation \eqref{eq:itoadj} is a true $Q^{q^\ast}$-martingale, we introduce the following localization:
	\begin{equation*}
		\tau^n := \inf\Set{ t \ge 0: \abs{\int_0^t\left\{ \eta_uk_u + p_u(\tilde{\pi}^\ast_u\sigma_u\eta_u + X^{\pi^\ast}_u\sigma_u\tilde{\pi}_u) \right\}\,dW_u^{Q^{q^\ast}}} > n }\wedge T.
	\end{equation*}
	Hence, taking expectation with respect to $Q^{q^\ast}$ on both sides of \eqref{eq:itoadj}, we have
	\begin{equation}
	\label{eq:stopped_max_prin}
		E_{Q^{q^\ast}}\left[ p_{\tau^n}\eta_{\tau^n} \right] = E_{Q^{q^\ast}}\left[ \int_0^{\tau^n}X^{\pi^\ast}_u\tilde{\pi}_u\sigma_u(p_u\theta_u + p_uq^\ast_u + k_u)\,du \right].
	\end{equation}
	By definition of ${\cal D}$, the family $(D_{0,\tau^n}^{\beta^\ast})_n$ is dominated by the bounded random variable $e^{\int_0^T(\beta^\ast_u)^-\,du}$. Moreover, for any $\delta > 0$ there exists $\varepsilon >0$ such that
	\begin{equation*}
		\eta_{\tau^n} \le \frac{1}{\varepsilon}(X^{\pi^\ast + \varepsilon\pi}_{\tau^n} - X^{\pi^\ast}_{\tau^n}) + \delta		              \le \frac{1}{\varepsilon}X^{\pi^\ast + \varepsilon\pi}_{\tau^n} + \delta.
	\end{equation*}
	Because we can restrict ourselves to subsolutions $(Y,Z) \in \mathcal{A}^u(X^\pi_T + \xi) $ satisfying $Y \ge X^\pi$, we can further estimate $\eta_{\tau^n}$ by
	\begin{equation*}
		\eta_{\tau^n} \le \frac{1}{\varepsilon}Y^{\pi^\ast + \varepsilon\pi}_{\tau^n} + \delta \le \frac{1}{\varepsilon}E_{Q^{q^\ast}}\left[D^{\beta^\ast}_{0,\tau^n}(X^{\pi^\ast + \varepsilon\pi}_T + \xi) + \int_{\tau^n}^Tg^\ast_u(\beta^\ast_u, q^\ast_u)\,du\mid {\cal F}_{\tau^n} \right] + \delta
	\end{equation*}
	where the second inequality follows from the same arguments which led to Equation \eqref{eq:weakdual} in the proof of Theorem \ref{thm:to-robustprob}.
	Hence,
	\begin{equation*}
		\eta_{\tau^n} \le \frac{1}{\varepsilon}E_{Q^{q^\ast}}\left[ e^{\int_0^T(\beta^\ast_u)^-\,du}(X^{\pi^\ast + \varepsilon\pi}_T + \xi) + \int_{0}^Tg^\ast_u(\beta^\ast_u, q^\ast_u)\,du\mid {\cal F}_{\tau^n} \right] + \delta.
	\end{equation*}
	Since the right hand side above is $Q^{q^\ast}$-uniformly integrable, taking the limit in \eqref{eq:stopped_max_prin} and using dominated convergence theorem and Fatou's lemma give
	\begin{equation*}
		E_{Q^{q^\ast}}\left[ \int_0^{T}X^{\pi^\ast}_u\tilde{\pi}_u\sigma_u(p_u\theta_u + p_uq^\ast_u + k_u)\,du \right] \le E_{Q^{q^\ast}}\left[D^{\beta^\ast}_{0,T}\eta_T \right] =0,
	\end{equation*}
	recall that both $p$ and $\eta$ are $Q^{q^\ast}$-a.s. continuous processes.
	Arguing as above with $-\pi$ instead of $\pi$, we have
	\begin{equation*}
		E_{Q^{q^\ast}}\left[ \int_0^{T}X^{\pi^\ast}_u\tilde{\pi}_u\sigma_u(p_u\theta_u + p_uq^\ast_u + k_u)\,du \right] = 0.
	\end{equation*}
	Thus, since $\pi$ was taken arbitrary, this leads to
	\begin{equation*}
		p_t\theta_t + p_tq^\ast_t + k_t = 0 \quad P\otimes dt\text{-a.s},
	\end{equation*}
	since $Q^{q^\ast} \sim P$.
	\item \emph{Characterization of} $\beta^\ast$ and $q^\ast$:
	 The function $g$ satisfies \ref{lsc} and $(\beta^\ast, q^\ast) \in \partial g(Y, Z) $ imply that $(Y,Z) \in \partial g^\ast(\beta^\ast, q^\ast) $, and since $g$ is strictly convex, it holds $\partial g^\ast(\beta^\ast, q^\ast) = \set{(Y,Z)}$ so that by \citep[Theorem 25.1]{roc70}, $g^\ast$ is differentiable at $(\beta^\ast, q^\ast)$.
	Hence, $\beta^\ast$ and $q^\ast$ are the points verifying
	\begin{equation*}
		- \frac{\partial g^\ast}{\partial a}(\beta^\ast_t,q^\ast_t) + Y_t = 0 \quad \text{ and }\quad	-\frac{\partial g^\ast}{\partial b }(\beta^\ast_t,q^\ast_t) + Z_t = 0 \quad P\otimes dt \text{-a.s.}
	\end{equation*}	
\end{enumerate}
\end{proof}

\section{Link to Conjugate Duality}
\label{sec:duality}

In this final section we show the inherent link between duality of BSDEs and the theory of conjugate duality in optimization as presented, for instance, in \citet{Eke-Tem}. 
We will exploit the general method of conjugate duality in convex optimization to study the problem at hands.
In Proposition \ref{thm:dualproblem} below we write the dual problem to \eqref{eq:optprob}.
The main result of this section, Theorem \ref{thm:minimax}, shows that even without the condition \ref{qg} which enabled us to have weak compactness, the robust control problem still satisfies a minimax property.
Consider the probability measure $Q = Q^\theta$ introduced in Section \ref{sec:setting}.
Recall that ${\cal H}^1(Q)$ is the set of $Q$-martingales $X$ such that $E_Q[\sup_{t \in [0,T]}\abs{X}_t]< \infty$.
We introduce the sets
\begin{equation*}
	\mathcal{C} := \Set{ X^\pi_T: \pi \in \Pi }\cap {\cal H}^1(Q), \quad \mathcal{M} := \Set{M \in \text{BMO}_{++}(Q): E_Q[MX^\pi_T] \le x \text{ for all } \pi \in \Pi}
\end{equation*}
and $\bar{\cal Q}:= \set{q \in {\cal L}: \frac{dQ^q}{dP} \in {\cal M}}$.
Let us define the perturbation function $F$ on $\mathcal{C}\times \mathcal{C}$ with values in $\mathbb{R}$ by
\begin{equation*}
	F(X^\pi_T, H) := \mathcal{E}^g_0\left( X^\pi_T + \xi + H \right).
\end{equation*}
For all $ H \in \mathcal{C}$ we put
\begin{equation*}
	u(H) := \sup_{\pi \in \Pi}F(X^\pi_T, H ).
\end{equation*}
The space $\text{BMO}(Q)$ can be identified with the dual of the space ${\cal H}^1(Q)$.
We extent the function $F$ to the Banach space ${\cal H}^1(Q) \times {\cal H}^1(Q)$ by setting $F(X^\pi_T, H) = -\infty$ whenever $H$ or $X^\pi_T$ does not belong to $\mathcal{C}$.
It holds $u(0) = V(x)$, the value function of the primal control problem.
Since $\mathcal{E}^g_0$ is concave increasing, the function $u$ is as well concave increasing, and from $u(0) = V(x) < \infty$ follows that $u(H)< \infty$ for all $H \in \mathcal{C}$. 
Define the concave conjugate $F^\ast$ of $F$ on $\text{BMO}(Q)\times \text{BMO}(Q)$ with values in $\bar{\mathbb{R}}$ by
\begin{equation*}
	F^*(M^\prime, M) := \inf_{ H, X^\pi_T \in {\cal H}^1(Q)  } \Set{ E_Q\left[ M^\prime X^\pi_T \right] + E_Q\left[ M H \right]  - F(X^\pi_T,H ) }.
\end{equation*}
The function $F^*$ is concave and upper semicontinuous. 
For each $M^\prime \in \text{BMO}(Q)$, put
\begin{equation}
\label{eq:dualprob1}
	v(M^\prime) := \inf_{M \in \text{BMO}(Q) }\set{-F^*( M^\prime,M ) }.
\end{equation}
For $M^\prime = 0$ Equation \eqref{eq:dualprob1} is the dual problem, and the relation $u(0)\leq v(0)$ follows as an immediate consequence of the definition of $F^*$. 
Since the functional $\mathcal{E}^g_0$ is increasing and $\mathcal{E}_0^g(0) > -\infty$ we have $u(0)\geq \mathcal{E}_0^g(0)> -\infty$. 
Hence $v(0)> - \infty$. 
\begin{lemma}
\label{thm:L0_usc}
	Assume  that the driver $g$ defined on $\mathbb{R}_+\times \mathbb{R}^d$ satisfies \ref{jconv}, \ref{lsc}, \ref{nor} and \ref{pos}.
	Then, the function $F$ is $\sigma({\cal H}^1(Q)\times {\cal H}^1(Q), \text{BMO}(Q)\times \text{BMO} (Q))$-upper semicontinuous.
\end{lemma}
\begin{proof}
	See Appendix \ref{appendix}.
\end{proof}
For any $M \in {\cal M}$, define by ${\cal E}^\ast_0$ the convex conjugate of ${\cal E}^g_0$ relative to the dual pair $({\cal H}^1(Q), \text{BMO}(Q))$. 
It follows from \citep[Remark 3.8]{tarpodual} that for each $M \in \mathcal{M}$, there exists $(\beta,q) \in \mathcal{D}\times \mathcal{Q}$, with $q$ unique such that $M = D^\beta_{0,T}dQ^q/dP $, and $D^\beta_{0,T} = E[M] $.
We put	
$$\mathcal{E}^\ast_0(\beta,q) := \inf_{\{M \in {\cal M}: E[M] = D^\beta_{0,T} \}} \mathcal{E}^\ast_0(M). $$

\begin{proposition}
\label{thm:dualproblem}
	Assume  that the driver $g$ defined on $\mathbb{R}_+\times \mathbb{R}^d$ satisfies \ref{jconv}, \ref{lsc}, \ref{nor} and \ref{pos}.
	Further assume that $\xi \in L^\infty_+$.
	Then the dual problem to \eqref{eq:optprob} is given by
	\begin{equation}
	\label{eq:dualproblem}
 		v(0) = \inf_{(\beta,q) \in \mathcal{D}\times \bar{\mathcal{Q}}}\Set{ \mathcal{E}^\ast_0(\beta,q) - E_{Q}\left[\frac{dQ^q}{dP}D^\beta_{0,T} \xi \right] } - x
	\end{equation}
	and the primal problem
	\begin{equation*}
		u(0) = \sup_{X^\pi_T \in {\cal C}} \inf_{(\beta, q) \in {\cal D}\times \bar{\cal Q}} E_{Q^q}\left[D^\beta_{0,T}\frac{dQ}{dP}(X^\pi_T + \xi)+ \int_0^TD^\beta_{0,u}g^\ast_u(\beta_u, q_u)\,du \right].
	\end{equation*}
\end{proposition}
\begin{proof}
	For every $M \in L^\infty $ one has
	\begin{align*}
		F^\ast(0, M ) & = \inf_{ H\in {\cal H}^1(Q), X^\pi_T \in {\cal C}}\Set{ E_Q\left[MH \right] - F(X^\pi_T, H)}\\
					  & = \inf_{ H^\prime\in {\cal H}^1(Q), X^\pi_T \in {\cal C}}\Set{ E_Q\left[ M(H^\prime - X^\pi_T - \xi) - F(X^\pi_T,H^\prime - X^\pi_T - \xi) \right] }.
	\end{align*}
	In fact, $\set{H^\prime - X^\pi_T  - \xi: X^\pi_T \in {\cal C}, \,\, H^\prime \in {\cal H}^1(Q)} \subseteq {\cal H}^1(Q)$ and, reciprocally, for any $H\in {\cal H}^1(Q)$ we can write $H= H^\prime - x - \xi = H^\prime - X^0_T - \xi$ for some $H^\prime \in {\cal H}^1(Q)$.
	Hence,
	\begin{equation*}
		F^\ast(0, M) = \inf_{ H^\prime \in {\cal H}^1(Q)}\Set{ E_Q\left[ MH^\prime \right] -\mathcal{E}^g_0( H^\prime) } - \sup_{X^\pi_T \in {\cal C}} E_Q\left[ M(X^\pi_T + \xi) \right]. 
	\end{equation*}
	It is clear that if there exists $X^\pi_T \in {\cal C}$ such that $E_Q[MX^\pi_T] > x$, then $F^\ast(0, M) = -\infty $. 
	Thus, the supremum in Equation \eqref{eq:dualprob1} can by restricted to $\mathcal{M} $, and $F^\ast(0,M)$ takes the form
	\begin{equation*}
		F^\ast(0,M) = \mathcal{E}^\ast_0( M)  - E_Q\left[ M\xi \right] - x.
	\end{equation*}
	Therefore, the dual problem \eqref{eq:dualprob1} to the control problem \eqref{eq:optprob} is given by
	\begin{align}
	\nonumber
		v(0) &= \inf_{M \in {\cal M}}\Set{ \mathcal{E}^\ast_0(M) - E_Q\left[ M\xi \right] } -x \\
	\nonumber	     
		     &= \inf_{(\beta,q) \in \mathcal{D}\times \bar{\mathcal{Q}}}\inf_{\{M: E[M] = D^\beta_{0,T} \}}\Set{ \mathcal{E}^\ast_0(M) - E_{Q}\left[\frac{dQ^q}{dP}D^\beta_{0,T} \xi \right] } - x\\
	\label{eq:dualprob}
		     &= \inf_{(\beta,q) \in \mathcal{D}\times \bar{\mathcal{Q}}}\Set{\mathcal{E}^\ast_0(\beta,q) - E_{Q}\left[\frac{dQ^q}{dP}D^\beta_{0,T} \xi \right] } - x.
	\end{align}

 	Now, let us introduce the following Lagrangian $L$, which is such that $-L$ is the $H$-conjugate of the function $F$, i.e.
	\begin{equation*}
		L(X^\pi_T, M) = \sup_{H \in \mathcal{C} } \Set{ F(X^\pi_T,H ) - E_{Q}\left[M H \right] }.
	\end{equation*}
	It is well known in convex duality theory, see for instance \cite{Eke-Tem}, that the following hold:
	\begin{equation*}
		F^*( M^\prime,M) = \inf_{X^\pi_T \in {\cal C} } \Set{ E_{Q}\left[ M^\prime X^\pi_T \right] - L(X^\pi_T, M )  }
	\end{equation*}
	and, since $F$ is $\sigma({\cal H}^1(Q)\times {\cal H}^1(Q), \text{BMO}(Q)\times \text{BMO}(Q) )$-upper semicontinuous, the Fenchel-Moreau theorem and definition of $L$ yield
	\begin{equation}
	\label{eq:represent_F_L}
		F(X^\pi_T,H ) = \inf_{M \in \mathcal{M} }\Set{ E_{ Q }\left[M H \right] + L(X^\pi_T,M ) }.
	\end{equation}
	In particular,
	\begin{equation*}
		v(0) = \inf_{ M \in {\cal M} }\sup_{X^\pi_T \in {\cal C} } \Set{ L(X^\pi_T, M) }\quad \text{and}\quad u(0) = \sup_{X^\pi_T \in {\cal C} }\inf_{ M \in \mathcal{M} }\Set{L(X^\pi_T, M)}.
	\end{equation*}
	Let $\pi \in \Pi$ and $M \in \mathcal{M}$.
	By definition of the Laplacian, we have
	\begin{align*}
		L(X^\pi_T,M ) &= \sup_{H \in {\cal H}^1(Q) }\Set{ F(X^\pi_T ,  H) - E_{Q}\left[ MH \right] }\\
	                       &= \sup_{H^\prime \in {\cal H}^1(Q)}\Set{ F(X^\pi_T,H^\prime - X^\pi_T - \xi) - E_{Q}\left[ M(H^\prime - X^\pi_T - \xi) \right] }\\
	                       &=  \sup_{H^\prime \in  {\cal H}^1(Q) }\Set{ \mathcal{E}^g_0(H^\prime) - E_{Q}\left[ MH^\prime \right] } + E_{Q}\left[M(X^\pi_T + \xi) \right] \\
	                       &= E_{Q}\left[M(X^\pi_T + \xi) \right] - \mathcal{E}^\ast_0(M).
	\end{align*}
	But by the proof of \citep[Theorem 3.10]{tarpodual}, the function 
	$$\alpha_{\text{min}}: M \mapsto \inf_{\set{\beta \in {\cal D}: E[M]=D^\beta_{0,T}}}E_{Q^q}\left[\int_0^TD^\beta_{0,u}g^\ast_u(\beta_u, q_u)\,du \right] $$ 
	is convex and $\sigma({\cal H}^1(Q), \text{BMO(Q)})$-lower semicontinuous; that is, it is the minimal penalty function. 
	Hence,	$ -\mathcal{E}^\ast_0(M) = \alpha_{\text{min}}(M)$ and therefore,
	\begin{equation*}
		L(X^\pi_T, M) = \inf_{\set{\beta \in {\cal D}: E[M]=D^\beta_{0,T}}}E_{Q^q}\left[D^\beta_{0,T}\frac{dQ}{dP}(X^\pi_T + \xi)+ \int_0^TD^\beta_{0,u}g^\ast_u(\beta_u, q_u)\,du \right].
	\end{equation*}
	In particular, this implies
	\begin{equation*}
		u(0) = \sup_{X^\pi_T \in {\cal C}} \inf_{(\beta, q) \in {\cal D}\times \bar{\cal Q}} E_{Q^q}\left[D^\beta_{0,T}\frac{dQ}{dP}(X^\pi_T + \xi)+ \int_0^TD^\beta_{0,u}g^\ast_u(\beta_u, q_u)\,du \right].
	\end{equation*}
\end{proof}
Next, we show that the control problem \eqref{eq:optprob} satisfies the minimax property even if we do not assume any growth condition on the generator $g$. 
Notice that it does not ensure existence of a saddle point.
We refer to \citep{bac-Fon} for some similar results in robust utility maximization.
\begin{theorem}
\label{thm:minimax}
	Assume that the driver $g$ satisfies \ref{jconv}, \ref{lsc}, \ref{nor} and \ref{pos}. Then, the value functions of the primal problem and dual problem coincide. More precisely, it holds
	\begin{equation*}
		\inf_{ M\in \mathcal{M} }\sup_{X^\pi_T \in {\cal C} } \Set{ L(X^\pi_T, (\beta ,q )) } = \sup_{X^\pi_T \in {\cal C} }\inf_{M  \in \mathcal{M} } \Set{L(X^\pi_T, (\beta ,q )) }.
	\end{equation*}	
\end{theorem}
\begin{proof}
	The main argument of the proof is the Fenchel-Rockafellar theorem applied on the Banach space ${\cal H}^1(Q) $. 
	By definition $\mathcal{M} = \mathcal{C}^*$, the polar cone of $\mathcal{C}$ with respect to the dual pair $({\cal H}^1(Q), \text{BMO}(Q))$. 
	Moreover, since $\mathcal{C}$ is a cone, $\mathcal{M}$ is the polar of $\mathcal{C}$, i.e. $\mathcal{M} = \mathcal{C}^\circ$. 
	Consider the convex-indicator function 
	\begin{equation*}
		\delta_{\mathcal{C}}(H) = 
		\begin{cases}
			0 & \text{ if} \quad H \in \mathcal{C}\\
			\infty & \text{ if} \quad H \in {\cal H}^1(Q)\setminus \mathcal{C}.
		\end{cases}
	\end{equation*}
	We can rewrite $u$ as
	\begin{equation*}
		u(0) = \sup_{H \in {\cal H}^1(Q)}\Set{ F(H, 0) - \delta_{\mathcal{C}}(H) }.
	\end{equation*}
	Since ${\cal C}$ is $\sigma({\cal H}^1(Q), \text{BMO}(Q))$-closed (see proof of Lemma \ref{thm:L0_usc}), the function $F(\cdot, 0) - \delta_{\mathcal{C}}(\cdot)$ is concave and $\sigma({\cal H}^1(Q), \text{BMO}(Q))$-upper semicontinuous.
	Hence, by \citep[Corollary 1]{Roc66} we have
	\begin{equation*}
		u(0) = \inf_{M \in \text{BMO}(Q)}\Set{\delta^\ast_{\mathcal{C}}(\beta,q) - F^\ast(0,M)}.
	\end{equation*}
	The function $\delta_{\mathcal{C}}$ obeys the conjugacy relation $\delta^\ast_{\mathcal{C}} = \delta_{\mathcal{C}^\circ} = \delta_{\mathcal{M}}$, see \citep[Section 11.E]{Roc-Wets}. 
	Thus,
	\begin{align*}
		u(0) & = \inf_{ M \in \text{BMO}(Q)}\Set{ \delta_{\mathcal{M}} (M) - F^\ast(0, M)}\\
             & = \inf_{ M \in \mathcal{M}}\set{ -F^\ast(0,M) } = v(0).
	\end{align*}
	This concludes the proof.
 	\end{proof}

 \begin{appendix}
\section{Proofs of Intermediate Results}
\label{appendix}
\begin{proof}[of Lemma \ref{convex_adm}]
	Let $(X^1, Y^1, Z^1)$ and $(X^2, Y^2, Z^2)$ be two elements of $\mathcal{A}(x)$; and $\lambda_1, \lambda_2\in (0,1)$ such that $\lambda_1 + \lambda_2 = 1$. 
	Then, by joint convexity of $g$, $(\lambda_1Y^1 + \lambda_2Y^2, \lambda_1Z^1 + \lambda_2Z^2)$ satisfies Equation \eqref{eq:subsol} and the terminal condition $\lambda_1Y^1_T + \lambda_2Y^2_T \leq \lambda_1X^1_T + \lambda_2X^2_T + \xi$ is also satisfied. 
	In addition, since $u^{-1}(E[u(\cdot)])$ is concave, for all $0\leq s\leq t\leq T$, we have
	\begin{align*}
		u^{-1}(E[u(\lambda_1Y^1_t + \lambda_2Y^2_t )\mid \mathcal{F}_s]) &\geq \lambda_1u^{-1}(E[u(Y^1_t)\mid \mathcal{F}_s]) + \lambda_2u^{-1}(E[u(Y^2_t)\mid \mathcal{F}_s])\\
                                     & \geq \lambda_1u^{-1}(u(Y^1_s)) + \lambda_2u^{-1}(u(Y^2_s))\\
                                     & = \lambda_1Y^1_s + \lambda_2Y^2_s,
	\end{align*}
	where the second inequality comes from the facts that $Y^1$ and $Y^2$ are admissible and $u^{-1}$ increasing. 
	Hence because $u$ is increasing, we have 
	\begin{equation*}
		E[u(\lambda_1Y^1_t + \lambda_2Y^2_t )\mid \mathcal{F}_s]\geq u(\lambda_1Y^1_s + \lambda_2Y^2_s),
	\end{equation*}
	which implies that $\lambda_1Y^1 + \lambda_2Y^2$ is admissible. 
	Put $X^1 = X^{\pi^1}$ and $X^2 = X^{\pi^2}$.
	The process $\lambda_1X^1 + \lambda_2X^2$ is a wealth process, since 
	$$\lambda_1X^1_t + \lambda_2X^2_t = x + \int_0^t(\lambda_1\pi^1_u + \lambda_2\pi^2_u)\sigma_u\,dW^Q_u.$$ 
\end{proof}

\begin{proof}[of Proposition \ref{thm:stability-xi}]
	First notice that the operator $\mathcal{E}^g_0(\cdot)$ is increasing. 
	Indeed, if $\xi' \leq \xi$ then $\mathcal{A}^r(\xi^\prime, g) \subseteq \mathcal{A}^r(\xi, g)$, which implies $\mathcal{E}^g_0(\xi^\prime) \le \mathcal{E}^g_0(\xi)$.
	Since the sequence $(\xi^n) \subseteq L^{\infty}_+$ is decreasing, the limit $\xi$ belongs to $L^{\infty}_+$.
	By monotonicity, $(\mathcal{E}^g_0(\xi^n))$ is a decreasing sequence, bounded from below by $\mathcal{E}^g_0(\xi)$. 
	Thus, we can define $Y_0 := \lim_{n\rightarrow \infty}\mathcal{E}^g_0(\xi^n) \geq \mathcal{E}^g_0(\xi)$. 
	By monotonicity and the condition $\ref{nor}$, $\mathcal{E}^g_0(\xi) \ge \mathcal{E}^g_0(0) > - \infty$.  
	Theorem \ref{existence} yields a maximal subsolution $( \bar{Y}^n, \bar{Z}^n ) \in \mathcal{A}^r(\xi^n,g) $ with $\bar{Y}^n_0 = \mathcal{E}^g_0(\xi^n) $ for all $n \in \mathbb{N}	$.
	We can use the method introduced in the proof of Theorem \ref{existence} to obtain a pair $( \bar{Y}, \bar{Z}) \in \mathcal{A}^r( \xi,g)$ with
	\begin{equation*}
		Y_0= \lim_{n\rightarrow \infty}\mathcal{E}^g_0(\xi^n) = \bar{Y}_0 = \mathcal{E}^g_0(\xi).
	\end{equation*}
	The sequence $(\bar{Y}^n_0)$ is not increasing as in the proof of Theorem \ref{existence} but decreasing. 
	Nevertheless we can obtain an estimate such as that of \eqref{z1} using $\bar{Y}^n_0 \leq Y^1_0$. 
	Finally, $\mathcal{E}^g_0(\xi)$ is optimal. 
	In fact, let $( Y, Z) \in \mathcal{A}^r(\xi, g )$ be any subsolution. Since $\xi \leq \xi^n$ for all $n \in \mathbb{N}$, we have $(Y, Z) \in \mathcal{A}^r(\xi^n, g )$. 
	Thus, $Y_0 \leq \mathcal{E}^g_0(\xi^n)$ for all $n$. 
	Taking the limit as $n$ tends to infinity, we conclude $Y_0\leq \mathcal{E}^g_0(\xi)$.
\end{proof}

\begin{proof}[of Proposition \ref{thm:stability-g}]
		Since $(g^n)$ is increasing, $(\mathcal{E}^{g^n}_0(\xi))$ is decreasing and bounded from below by $\mathcal{E}^{g}_0(\xi)$.
	Define $Y_0 := \lim_{n\rightarrow \infty}\mathcal{E}^{g^n}_0(\xi) \geq \mathcal{E}^g_0(\xi)$. 
	$Y_0$ is finite since $\mathcal{E}_0^{g}(\xi) \leq Y_0\le \mathcal{E}_0^{g^1}(\xi)$.
	For all $n$, there exists $(\bar{Y}^n, \bar{Z}^n) \in {\cal A}^r(\xi, g^n)$ such that ${\cal E}^{g^n}_0(\xi) = \bar{Y}^n_0$.	 
	Then by the method introduced in the proof of Theorem \ref{existence} we can obtain a candidate $( \bar{Y}, \bar{Z})$, maximal subsolution of the system with parameters $g$ and $\xi$.
	The verification that $( \bar{Y}, \bar{Z})$ is indeed an element of $\mathcal{A}^r(\xi,g)$ relies on Fatou's lemma and monotone convergence theorem, since $g^n\uparrow g$. 
	See the proof of \citep[Theorem 4.14]{DHK1101} for similar arguments. 
	The subsolution $(\bar{Y}, \bar{Z})$ is maximal, since $\bar{Y}_0 = Y_0 $.
\end{proof}

\begin{proof}[of Lemma \ref{thm:L0_usc}]
	Let us first show that ${\cal C}$ is closed in ${\cal H}^1(Q)$.
	For any sequence $(X^n_T) \subseteq {\cal C}$ converging to $X_T$ in ${\cal H}^1(Q)$, the process $X_t := E_Q[X_T\mid {\cal F}_t]$ defines a positive $Q$-martingale starting at $x$.
	By martingale representation theorem, there exists $\nu \in {\cal L}^1(Q)$ such that $X_t = x + \int_0^t\nu_u\,dW_u^Q$, but since $\sigma\sigma^\prime$ is of full rank, we can find a predictable process $\pi$ such that $\pi\sigma = \nu$.
	Therefore, $dX_t = \pi_t\sigma_t(\theta_tdt + dW_t)$.
	That is, $X \in {\cal C}$.
	
	Now it suffices to show that the function $F$ is $\sigma({\cal H}^1(Q)\times {\cal H}^1(Q), \text{BMO}(Q)\times \text{BMO}(Q))$-upper semicontinuous on ${\cal C}\times {\cal C}$ because the extension to ${\cal H}^1(Q)\times {\cal H}^1(Q)$ would also be weakly upper semicontinuous. 
	Hence, we need to show that for every $c \ge 0$ the concave level set $\set{(\alpha,\gamma)\in {\cal C}\times {\cal C} : F(\alpha, \gamma) \ge c} $ is closed in ${\cal C}\times {\cal C}$.
	Let $c \ge 0$ be fixed and let us show that $\set{\zeta\in  {\cal H}^1(Q) : {\cal E}^g_0(\zeta) \ge c}$ is ${\cal H}^1(Q)$-closed. 
	Let $(\zeta^n)$ be a sequence converging in ${\cal H}^1(Q)$ to $\zeta$ and such that $\mathcal{E}^g_0(\zeta^n) \ge c $ for every $n \in \mathbb{N}$.
	Put $\eta^n := \sup_{m\ge n}\zeta^m $, $n \in \mathbb{N}$. 
	The sequence $(\eta^n)$ decreases to $\zeta$ and by Proposition \ref{thm:stability-xi}, $(\mathcal{E}^g_0(\eta^n))$ converges to $\mathcal{E}^g_0(\zeta)$ and is decreasing.
	Hence, since $\mathcal{E}^g_0(\zeta) = \lim_{n \rightarrow \infty}\mathcal{E}^g_0(\eta^n) = \inf_{n}\mathcal{E}^g_0(\eta^n) $, it holds
	\begin{align*}
		\mathcal{E}^g_0(\zeta) & = \inf_{n \in \mathbb{N} }\mathcal{E}^g_0\left( \sup_{m \ge n}\zeta^m \right)\\
						       & \ge \inf_{n \in \mathbb{N} }\sup_{m \ge n}\mathcal{E}^g_0(\zeta^m) = \limsup_{n \rightarrow \infty}\mathcal{E}^g_0(\zeta^n).
	\end{align*}
	Now for every sequence $(\alpha^n, \gamma^n) \subseteq {\cal C}\times {\cal C}$ converging to $(\alpha,\gamma)\in {\cal C}\times {\cal C}$ in ${\cal H}^1(Q)\times {\cal H}^1(Q)$ such that $F(\alpha^n, \gamma^n)\ge c$ for every $n\in \mathbb{N}$ one has
	\begin{align*}
		c&\le \limsup_{n \rightarrow \infty}F(\alpha^n, \gamma^n) = \limsup_{n \rightarrow \infty}{\cal E}^g_0(\alpha^n + \gamma^n + \xi) \\
		&\le {\cal E}^g_0(\alpha + \gamma + \xi) = F(\alpha, \gamma).
	\end{align*}
	This concludes the proof.
\end{proof}
 \end{appendix}

\bibliographystyle{abbrvnat}
\bibliography{my-references}



@Book{scipy,
  Title                    = {{SciPy Reference Guide}},
  Author                   = {The SciPy community},
  Publisher                = {0.8.dev},
  Year                     = {2010},

  File                     = {:/home/ludovic/References/Books/Programming/scipy-ref.pdf:PDF},
  Owner                    = {ludovic},
  Timestamp                = {2011.08.25}
}

@Book{oks,
  Title                    = {{Stochastic Differential Equations: an Introduction with Applications}},
  Author                   = {Bernd {\O}ksendal},
  Publisher                = {Springer},
  Year                     = {2007},
  Edition                  = {6},

  File                     = {:/home/ludovic/References/Books/Stochastic Analysis/Stochastic Differential Equations An Introduction with Applications 5th ed - Oksendal B.pdf:PDF},
  Owner                    = {ludovic},
  Timestamp                = {2011.07.09}
}

@Unpublished{Oks-Sul13,
  Title                    = {{A Stochastic Control Approach to Robust Duality in Utility Maximization }},
  Author                   = {Bernt {\O}ksendal and Agn\`es Sulem},
  Note                     = {Preprint},
  Year                     = {2013},

  File                     = {:/home/ludovic/References/Papers/Oks-Sul13.pdf:PDF},
  Owner                    = {ludovic},
  Timestamp                = {2013.05.14}
}

@Article{Oks-Sul12,
  Title                    = {{Forward-Backward Stochastic Differential Games and Stochastic Control under Model Uncertainty }},
  Author                   = {Bernt {\O}ksendal and Agn\`es Sulem},
  Journal                  = {J. Optim. Theory Appl.},
  Year                     = {2012},
  Pages                    = {1-31},
  Volume                   = {1},

  File                     = {:/home/ludovic/References/Papers/Oks-Sul12.pdf:PDF},
  Owner                    = {ludovic},
  Timestamp                = {2013.05.14}
}

@Article{Acc-etal,
  Title                    = {{A Trajectorial Interpretation of Doob Martingale Inequalities}},
  Author                   = {Acciaio, B, and Beiglb\"ock, M and Penkner, F. and Schachermayer, W. and Temme, J},
  Journal                  = {Annals of Applied Probability},
  Year                     = {2013},
  Number                   = {4},
  Pages                    = {1494 --- 1505},
  Volume                   = {23},

  File                     = {:Papers/Acc-etal.pdf:PDF},
  Owner                    = {ludovic},
  Quality                  = {1},
  Timestamp                = {2013.08.06}
}

@Unpublished{Acc-etal2013,
  Title                    = {{A Model-free Version of the Fundamental Theorem of Asset Pricing and the Super-Replication Theorem}},
  Author                   = {Beatrice Acciaio and Mathias Beiglb\"ock and F. Penkner and Walter Schachermayer},
  Note                     = {Forthcoming in Mathematical Finance},

  Month                    = {Mar.},
  Year                     = {2013},

  File                     = {:Papers/Acc-etal1013.pdf:PDF},
  Owner                    = {ludovic},
  Quality                  = {1},
  Timestamp                = {2013.12.02}
}

@Article{Acc-Pen,
  Title                    = {{Dynamic Risk Measures}},
  Author                   = {Beatrice Acciaio and Irina Penner},
  Year                     = {2010},

  Owner                    = {ludovic},
  Timestamp                = {2014.04.07}
}

@Book{Aliprantis,
  Title                    = {{Infinite Dimensional Analysis}},
  Author                   = {Charalambos Aliprantis and Kim Border},
  Publisher                = {Springer},
  Year                     = {2006},

  Owner                    = {ludovic},
  Quality                  = {1},
  Timestamp                = {2013.09.19}
}

@MastersThesis{Adam,
  Title                    = {On {W}eak {D}ifferentiability of {B}ackward {S}{D}{E}s and {C}ross {H}edging of {I}nsurance {D}erivatives},
  Author                   = {Adam Andersson},
  School                   = {Chalmers, Goteborg University},
  Year                     = {2008},
  Type                     = {Msc},

  Abstract                 = {This thesis deals with distributional differentiability of the solution (X,Y,Z) to a quadratic non-degenerate forward-backward SDE. The differentiability is considered with respect to the initial value of the solution X to the coupled forward SDE. It is proved that the solution process Y is weakly differentiable, and that the solution process Z can be represented using the distributional gradient of Y. This result is new in the way that it relaxes technical conditions imposed by previous authors in a significant way and in a way that is important e.g., in the applications described below. The proof makes use of Dirichlet space techniques to conclude that Y is a member of a local Sobolev space. Our results are applied to derive new results in mathematical finance and insurance theory. When derivatives are written on non-tradeable underlying assets, such as weather, a strongly correlated tradeable asset price process is used instead of the non-tradeable one to partially hedge the risk of the derivative. This concept is known as cross hedging. Applications for non-differentiable European type pay off functions are given and explicit hedging strategies are derived using a distributional gradient.},
  File                     = {:/home/ludovic/References/Papers/Adam.pdf:PDF;:home/ludovic/US-Msc-thesis/Papers/Adam.pdf:PDF}
}

@Article{Ank-Sca-Loi,
  Title                    = {Credit {R}isk {P}remia and {Q}uadratic {B}{S}{D}{E}s with a {S}ingle {J}ump},
  Author                   = {Stefan Ankirchner and Christophette Blanchet-Scalliet and Anne Eyraud-Loisel},
  Journal                  = {Int. J. Theor. Appl. Finance},
  Year                     = {2010},

  Month                    = {Nov},
  Number                   = {07},
  Pages                    = {1103 - 1129},
  Volume                   = {13},

  Abstract                 = {This paper is concerned with the determination of credit risk premia of defaultable contingent claims by means of indifference valuation princi- ples. Assuming exponential utility preferences we derive representations of indifference premia of credit risk in terms of solutions of Backward Stochastic Differential Equations (BSDE). The class of BSDEs needed for that representation allows for quadratic growth generators and jumps at random times. Since the existence and uniqueness theory for this class of BSDEs has not yet been developed to the required generality, the first part of the paper is devoted to fill that gap. By using a simple constructive algorithm, and known results on continuous quadratic BSDEs, we provide sufficient conditions for the existence and uniqueness of quadratic BSDEs with discontinuities at random times.},
  File                     = {:/home/ludovic/US-Msc-thesis/Papers/Ank-Sca-Loi.pdf:PDF},
  Owner                    = {ludovic}
}

@InProceedings{Ank-Imk,
  Title                    = {{Hedging with Residual Risk: a BSDE Approach }},
  Author                   = {Stefan Ankirchner and Peter Imkeller},
  Booktitle                = {Seminar on Stochastic Analysis, Random Fields and Applications VI},
  Year                     = {2011},
  Editor                   = {Robert Dalang},
  Pages                    = {311-325},
  Publisher                = {Springer},
  Volume                   = {63},

  File                     = {:/home/ludovic/References/Papers/Ank-Imk.pdf:PDF},
  Owner                    = {ludovic},
  Timestamp                = {2011.08.27}
}

@Unpublished{Ank-Imk08,
  Title                    = {{Quadratic Hedging of Weather and Catastrophe Risk by using Short Term Climate Predictions}},
  Author                   = {Stefan Ankirchner and Peter Imkeller},
  Year                     = {2008},

  File                     = {:/home/ludovic/References/Papers/Ank-Imk08.pdf:PDF},
  Owner                    = {ludovic},
  Timestamp                = {2011.08.26}
}

@Article{Ank-Imk-Pop2,
  Title                    = {On {M}easure {S}olutions of {B}ackward {S}tochastic {D}ifferential {E}quations},
  Author                   = {Stefan Ankirchner and Peter Imkeller and Alexandre Popier},
  Journal                  = {Stochastic Process. Appl.},
  Year                     = {2009},
  Pages                    = {2744-2772},
  Volume                   = {9},

  Abstract                 = {We consider backward stochastic differential equations (BSDEs) with nonlinear generators typically of quadratic growth in the control variable. A measure solution of such a BSDE will be understood as a probability measure under which the generator is seen as vanishing, so that the classical solution can be reconstructed by a combination of the operations of conditioning and using martingale representations. For the case where the terminal condition is bounded and the generator fulfills the usual continuity and boundedness conditions, we show that measure solutions with equivalent measures just reinterpret classical ones. For the case of terminal conditions that have only exponentially bounded moments, we discuss a series of examples which show that in the case of non-uniqueness, classical solutions that fail to be measure solutions can coexist with different measure solutions.},
  File                     = {:/home/ludovic/References/Papers/Ank-Imk-Pop2.pdf:PDF;:home/ludovic/US-Msc-thesis/Papers/Ank-Imk-Pop2.pdf:PDF}
}

@Article{Ank-Imk-Pop,
  Title                    = {Optimal {C}ross {H}edging of {I}nsurance {D}erivatives},
  Author                   = {Stefan Ankirchner and Peter Imkeller and Alexandre Popier},
  Journal                  = {Stoch. Anal. Appl.},
  Year                     = {2008},
  Pages                    = {679-709},
  Volume                   = {26},

  Abstract                 = {We consider insurance derivatives depending on an external physical risk process, for example, a temperature in a low dimensional climate model. We assume that this process is correlated with a tradable financial asset. We derive optimal strategies for exponential utility from terminal wealth, determine the indifference prices of the derivatives, and interpret them in terms of diversification pressure. Moreover, we check the optimal investment strategies for standard admissibility criteria. Finally, we compare the static risk connected with an insurance derivative to the reduced risk due to a dynamic investment into the correlated asset. We show that dynamic hedging reduces the risk aversion in terms of entropic risk measures by a factor related to the correlation.},
  File                     = {:/home/ludovic/References/Papers/Ank-Imk-Pop.pdf:PDF;:home/ludovic/US-Msc-thesis/Papers/Ank-Imk-Pop.pdf:PDF}
}

@Article{Ank-Imk-Reis,
  Title                    = {Pricing and {H}edging of {D}erivatives {B}ased on {N}on-tradable {U}nderlyings},
  Author                   = {Stefan Ankirchner and Peter Imkeller and Goncalo Reis},
  Journal                  = {Mathematical Finance},
  Year                     = {2010},

  Month                    = {March},
  Pages                    = {289-312},
  Volume                   = {20},

  Abstract                 = {This paper is concerned with the study of insurance related derivatives on financial mar- kets that are based on non-tradable underlyings, but are correlated with tradable assets. We calculate exponential utility-based indifference prices, and corresponding derivative hedges. We use the fact that they can be represented in terms of solutions of forward-backward stochastic differential equations (FBSDE) with quadratic growth generators. We derive the Markov property of such FBSDE and generalize results on the differentiability relative to the initial value of their forward components. In this case the optimal hedge can be repre- sented by the price gradient multiplied with the correlation coefficient. This way we obtain a generalization of the classical ?delta hedge? in complete markets.},
  File                     = {:/home/ludovic/References/Papers/Ank-Imk-Reis.pdf:PDF;:home/ludovic/US-Msc-thesis/Papers/Ank-Imk-Reis.pdf:PDF}
}

@Book{Bac-Bie-Hipp-,
  Title                    = {{Stochastic Methods in Finance}},
  Author                   = {K. Back and T. R. Bielecki and C. Hipp and S. Peng and Walter Schachermayer},
  Editor                   = {Marco Frittelli and Wolfgan Runggaldier},
  Publisher                = {Springer},
  Year                     = {2004},
  Note                     = {Lecture Notes in Mathematics},

  File                     = {:/home/ludovic/References/Books/Finance stochastics/Stochastic Methods in Finance.pdf:PDF},
  Owner                    = {ludovic},
  Timestamp                = {2011.06.24}
}

@Unpublished{bac-Fon,
  Title                    = {{Robust Utility Maximization without Model Compactness}},
  Author                   = {Julio Backhoff and Joaqu\`{i}n Fontbona},
  Note                     = {Preprint},
  Year                     = {2014},

  Owner                    = {ludovic},
  Timestamp                = {2014.08.28}
}

@Article{Bar-Pro,
  Title                    = {{On convergence of Semimartingales}},
  Author                   = {M. Barlow and P. E. Protter},
  Journal                  = {S\'eminaire de Probabilit\'e XIII, Lect. Notes Math.},
  Year                     = {1990},
  Pages                    = {188-193},
  Volume                   = {1426},

  File                     = {:/home/ludovic/References/Papers/Bar-Pro.pdf:PDF},
  Owner                    = {ludovic},
  Timestamp                = {2012.02.28}
}

@Article{Bar-Karoui,
  Title                    = {Closedness {R}esults for {B}{M}{O} semi-martiingales and {A}pplication to {Q}uadratic {B}{S}{D}{E}s},
  Author                   = {Pauline Barrieu and Nicolas Cazanave and Nicole {El Karoui}},
  Journal                  = {Comptes Rendus Acad\'{e}mie des Sciences de Paris},
  Year                     = {2008},
  Pages                    = {881-886},
  Volume                   = {1},

  Abstract                 = {We give a closedness result for a convex set of BMO semi-martingales, that contains solutions to quadratic BSDEs. We deduce convergence and monotone stability results for quadratic BSDEs},
  File                     = {:/home/ludovic/References/Papers/Bar-Caz-Kar.pdf:PDF;:home/ludovic/US-Msc-thesis/Papers/Bar-Caz-Kar.pdf:PDF},
  Issue                    = {346}
}

@InProceedings{Bar-Kar,
  Title                    = {{Pricing, Hedging and Optimally Designing Derivatives via Minimization of Risk Measures}},
  Author                   = {Pauline Barrieu and Nicole {El Karoui}},
  Booktitle                = {INDIFFERENCE PRICING},
  Year                     = {2007},
  Editor                   = {Rene Carmona, ed.},
  Pages                    = {144--172},
  Publisher                = {Princeton University Press},

  File                     = {:/home/ludovic/References/Papers/barrieu and El Karoui see thm 7.4.pdf:PDF}
}

@Unpublished{Bay-Zha,
  Title                    = {{ Fundamental Theorem of Asset Pricing under Transaction costs and Model Uncertainty}},
  Author                   = {Erhan Bayraktar and Yuchong Zhang},
  Note                     = {Preprint},
  Year                     = {2013},

  Owner                    = {ludovic},
  Quality                  = {1},
  Timestamp                = {2013.12.02}
}

@Article{Bec,
  Title                    = {Bounded {S}olutions to {B}ackward {S}{D}{E}s with {J}umps for {U}tility {O}ptimization and {I}ndifference {H}edging},
  Author                   = {D. Becherer},
  Journal                  = {Ann. Appl. Probab.},
  Year                     = {2006},
  Pages                    = {2027-2054},
  Volume                   = {16}
}

@Unpublished{Bei-Sch-Vel,
  Title                    = {{A Short Proof of the Doob-Meyer Theorem}},
  Author                   = {M. Beiglb\"{o}ck and W. Schachermayer and B. Veliyev},
  Note                     = {Submitted},
  Year                     = {2010},

  File                     = {:/home/ludovic/References/Papers/Bei-Sch-Vel.pdf:PDF},
  Owner                    = {ludovic},
  Timestamp                = {2011.03.21}
}

@Article{Bei-HL-Pen13,
  Title                    = {{Model-independent Bounds for Option Prices: A Mass Transport Approach}},
  Author                   = {Matthias Beiglb\"ock and P. Henry-Labordaire and F. Penkner},
  Journal                  = {Finance Stoch.},
  Year                     = {2013},
  Number                   = {3},
  Pages                    = {477-501},
  Volume                   = {17},

  Owner                    = {ludovic},
  Timestamp                = {2014.11.07}
}

@Unpublished{Ben-Ste,
  Title                    = {{Error Criteria for Numerical Solutions of Backward SDEs}},
  Author                   = {Christian Bender and Jessica Steiner},
  Note                     = {Preprint},
  Year                     = {2010},

  File                     = {:/home/ludovic/US-Msc-thesis/Papers/Ben-Ste.pdf:PDF},
  Owner                    = {ludovic},
  Timestamp                = {2011.02.19}
}

@Article{Ben,
  Title                    = {Existence of {O}ptimal {S}trategies {B}ased on {S}pecific {I}nformation, for a {C}lass of {S}tochastic {D}ecision {P}roblems},
  Author                   = {V. E. Benes},
  Journal                  = {SIAM Journal on Control and Optimization},
  Year                     = {1970},
  Pages                    = {179-188},
  Volume                   = {8}
}

@Article{Bensou,
  Title                    = {{Stochastic Maximum Principle for Distributed Parameter Systems}},
  Author                   = {A. Bensoussan},
  Journal                  = {Journal of The Franklin Institute},
  Year                     = {1983},
  Number                   = {5-6},
  Pages                    = {387 - 406},
  Volume                   = {315},

  Abstract                 = {The paper extends the finite dimensional stochastic maximum principle to a large class of distributed parameter systems. Although this class of systems does not include all actual distributed parameter systems, the approach of this paper can be easily adapted to any particular situation.},
  Doi                      = {DOI: 10.1016/0016-0032(83)90059-5},
  ISSN                     = {0016-0032},
  Owner                    = {ludovic},
  Timestamp                = {2010.11.24}
}

@Article{Ben-Kar,
  Title                    = {A {P}{D}{E} {R}epresentation of the {D}ensity of the {M}inimal {E}ntropy {M}artingale {M}easure in {S}tochastic {V}olatility {M}arkets},
  Author                   = {F. E. Benth and K. H. Karlsen},
  Journal                  = {Stochastics},
  Year                     = {2005},
  Pages                    = {109-137},
  Volume                   = {77}
}

@Article{Bernouilli,
  Title                    = {Specimen theoriae novae de mensura sortis, {{Commentarii Academiae Scientiarum Imperialis Petropolitanae}} ({5}, 175-192, 1738)},
  Author                   = {Daniel Bernoulli},
  Journal                  = {Econometrica},
  Year                     = {1954},
  Note                     = {Translated by L. Sommer},
  Pages                    = {23--36},
  Volume                   = {22}
}

@Article{Biagini-Frittelli,
  Title                    = {Utility {M}aximization in {I}ncomplete {M}arkets for {U}nbounded {P}rocesses},
  Author                   = {S. Biagini and Marco Frittelli},
  Journal                  = {Finance Stochastic},
  Year                     = {2005},
  Pages                    = {493-515},
  Volume                   = {9}
}

@InCollection{Bie-Jea,
  Title                    = {Indifference {P}ricing of {D}efaultable {C}laims},
  Author                   = {Tomasz R. Bielecki and Monique Jeanblanc},
  Booktitle                = {Indifference Pricing},
  Publisher                = {Princeton University Press},
  Year                     = {2009},
  Editor                   = {Ren\'e Carmona},
  Pages                    = {211--230},

  Abstract                 = {The goal of this chapter is to give an application of the theory of indifference prices in the context of defaultable claims within the reduced-form approach. In this approach the defaultable market is incomplete and there does not exist a (perfect) hedging strategy for claims which depend on the occurrence of the default. An important issue is the issue of choice of relevant information.},
  File                     = {:/home/ludovic/References/Papers/Bie-Jea.pdf:PDF;:home/ludovic/US-Msc-thesis/Papers/Bie-Jea.pdf:PDF}
}

@Book{Bil,
  Title                    = {{Probability and Measure}},
  Author                   = {Patrick Billingsley},
  Publisher                = {John Wiley and Sons},
  Year                     = {1986},

  File                     = {:/home/ludovic/References/Books/Stochastic Analysis/Billingsley P. - Probability and Measure (1986)(2nd_ed.)(en)(622s).djvu:Djvu},
  Owner                    = {ludovic},
  Timestamp                = {2012.12.19}
}

@Book{Billingsley1968,
  Title                    = {Convergence of probability measures},
  Author                   = {Billingsley, Patrick},
  Publisher                = {John Wiley \& Sons Inc.},
  Year                     = {1968},

  Address                  = {New York},

  Owner                    = {ludovic},
  Timestamp                = {2014.08.01}
}

@Article{Bis,
  Title                    = {{Conjugate Convex Functions in Optimal Stochastic Control}},
  Author                   = {Jean-Michel Bismut},
  Journal                  = {Journal of Mathematical Analysis and Applications},
  Year                     = {1973},
  Pages                    = {384-404},
  Volume                   = {44},

  Abstract                 = {This paper is concerned with the applications of general methods of convex analysis to problems of optimal stochastic control. In particular we will define what dual problems are in optimal stochastic control, and what the coextremality conditions for dual optimums are. The problem that we solve here being more general than a purely deterministic one, the results which are given include the results of deterministic control. The methods and the exposition of the results are very similar to the corresponding methods used by Rockafellar in [13], to which we will refer constantly. One of the apparent shortcomings of the method is that, using strictly variational methods, it must suppose that the information u-fields are fixed. In some cases, where the information u-fields are generated by the state variable, it is possible to apply the duality methods to a modified problem. But they will not give us the strong results it is possible to obtain by studying more specialized problems, as existence of optimal Markov controls. We develop other methods in [2] for this type of problem. The obvious reason is that, by developing a formalism applicable to purely deterministic cases, as to stochastic cases, it does not use the stochastic features of the problem in some purely stochastic cases. Because of their very technical nature, existence results will not be given here, but are developed extensively in [l] and [2]. The results of probability theory which we use can be found in [7] and [8], which we will take as references.},
  File                     = {:/home/ludovic/References/Papers/Bis1.pdf:PDF},
  Owner                    = {ludovic},
  Timestamp                = {2011.02.20}
}

@Book{Bjo,
  Title                    = {{Arbitrage Theory in Continuous Time}},
  Author                   = {Tomas Bj\"ork},
  Publisher                = {Oxford University Press},
  Year                     = {2009},
  Edition                  = {Third edition},

  File                     = {:/home/ludovic/References/Books/Stochastic finance/Arbitrage Theory in Continuous Time.pdf:PDF},
  Owner                    = {ludovic},
  Pages                    = {546},
  Timestamp                = {2011.02.22}
}

@Article{Bon-Tan13,
  Title                    = {A model-free no-arbitrage price bound for variance options},
  Author                   = {J.F. Bonnans and X. Tan},
  Journal                  = {Applied Mathematics and Optimization},
  Year                     = {2013.},
  Number                   = {1},
  Pages                    = {43-73},
  Volume                   = {68},

  Owner                    = {ludovic},
  Timestamp                = {2014.11.07}
}

@InCollection{Bor-Mat-Sch,
  Title                    = {{A Stochastic Control Approach to a Robust Utility Maximization Problem}},
  Author                   = {Bordigoni, Giuliana and Matoussi, Anis and Schweizer, Martin},
  Booktitle                = {Stochastic Analysis and Applications},
  Publisher                = {Springer Berlin Heidelberg},
  Year                     = {2007},
  Editor                   = {Benth, FredEspen and Nunno, Giulia and Lindstr{\o}m, Tom and {\O}ksendal, Bernt and Zhang, Tusheng},
  Pages                    = {125-151},
  Series                   = {Abel Symposia},
  Volume                   = {2},

  ISBN                     = {978-3-540-70846-9}
}

@Unpublished{Bru-Nutz,
  Title                    = {{Arbitrage and Duality in Nondominated Discrete-Time Models}},
  Author                   = {Bruno Bouchard and Marcel Nutz},
  Note                     = {Forthcoming in Annals of Applied Probability},

  Month                    = {Feb.},
  Year                     = {2014},

  Owner                    = {ludovic},
  Quality                  = {1},
  Timestamp                = {2013.12.02}
}

@Article{Bou-Tou,
  Title                    = {{Discrete-time Approximation and Monte-Carlo Simulation of Backward Stochastic Differential Equations}},
  Author                   = {Bruno Bouchard and Nizar Touzi},
  Journal                  = {Stoch. Anal. Appl.},
  Year                     = {2004},
  Pages                    = {175-206},
  Volume                   = {111},

  Abstract                 = {We suggest a discrete-time approximation for decoupled forward–backward stochastic di er- ential equations. The Lp norm of the error is shown to be of the order of the time step. Given a simulation-based estimator of the conditional expectation operator, we then suggest a backward simulation scheme, and we study the induced Lp error. This estimate is more investigated in the context of the Malliavin approach for the approximation of conditional expectations. Extensions to the re ected case are also considered.},
  File                     = {:/home/ludovic/References/Papers/Bou-Tou.pdf:PDF},
  Owner                    = {ludovic},
  Timestamp                = {2011.02.20}
}

@Book{Bou,
  Title                    = {{Topologie G\'{e}n\'{e}rale chapitres 1 \`{a} 4 }},
  Author                   = {Nicolas Bourbaki},
  Publisher                = {Diffusion C.C.L.S Paris},
  Year                     = {1971},

  File                     = {:Books/Math_gen/Bourbaki_N._Topologie_generale_Elements_de_math.__.djvu:Djvu},
  Owner                    = {ludovic},
  Quality                  = {1},
  Timestamp                = {2013.10.25}
}

@Book{Bra,
  Title                    = {{Numerical Methods in Finance and Economics: A MATLAB-Based Introduction}},
  Author                   = {Paolo Brandimarte},
  Publisher                = {John Wiley and Sons},
  Year                     = {2006},
  Edition                  = {2},

  File                     = {:/home/ludovic/References/Books/Finance/Numerical Methods in Finance-matlab.pdf:PDF},
  Owner                    = {ludovic},
  Pages                    = {694},
  Timestamp                = {2011.06.12}
}

@Article{Bri-Hu,
  Title                    = {Quadratic {B}{S}{D}{E}s with {C}onvex {G}enerators and {U}nbounded {T}erminal {C}onditions},
  Author                   = {Philippe Briand and Ying Hu},
  Journal                  = {Probab. Theory Relat. Fields},
  Year                     = {2008},
  Pages                    = {543-567},
  Volume                   = {141},

  Abstract                 = {In Briand and Hu (Probab Theory Relat Fields 136(4):604?618, 2006), the authors proved an existence result for BSDEs with quadratic generators with respect to the variable z and with unbounded terminal conditions. However, no uniqueness result was stated in that work. The main goal of this paper is to fill this gap. In order to obtain a comparison theorem for this kind of BSDEs, we assume that the generator is convex with respect to the variable z. Under this assumption of convexity, we are also able to prove a stability result in the spirit of the a priori estimates stated in Karoui et al. (Math Finance 7(1):1?71, 1997). With these tools in hands, we can derive the nonlinear Feynman?Kac formula in this context.},
  File                     = {:/home/ludovic/References/Papers/Bri-Hu.pdf:PDF;:home/ludovic/US-Msc-thesis/Papers/Bri-Hu.pdf:PDF}
}

@Article{Bri-Hu1,
  Title                    = {B{S}{D}{E} with {Q}uadratic {G}rowth and {U}nbounded {T}erminal {V}alue},
  Author                   = {Philippe Briand and Ying Hu},
  Journal                  = {Probab. Theory Relat. Fields},
  Year                     = {2006},
  Pages                    = {604-618},
  Volume                   = {136},

  Abstract                 = {In this paper, we study the existence of solution to BSDE with quadratic growth and unbounded terminal value. The main idea consists in using a localization procedure together with a priori bounds.},
  File                     = {:/home/ludovic/References/Papers/Bri-Hu1.pdf:PDF;:home/ludovic/US-Msc-thesis/Papers/Bri-Hu1.pdf:PDF}
}

@Book{Brz-Zas,
  Title                    = {{Basic Stochastic Processes}},
  Author                   = {Zdzislaw Brzezniak and Tomasz Zastawniak},
  Publisher                = {Springer Undergraduate Series},
  Year                     = {2001},

  File                     = {:/home/ludovic/References/Books/Stochastic Analysis/Basic Stochastic Processes.pdf:PDF;:/home/ludovic/References/Books/Brz-Zas.pdf:PDF},
  Owner                    = {ludovic},
  Timestamp                = {2011.05.30}
}

@Article{Buc-Hu,
  Title                    = {{Hedging Contingent Claims for a Large Investor in an Incomplete Market}},
  Author                   = {Buckdahn, Rainer and Hu, Ying},
  Journal                  = {Advances in Applied Probability},
  Year                     = {1998},
  Number                   = {1},
  Pages                    = {pp. 239-255},
  Volume                   = {30},

  Abstract                 = {In this paper we study the problem of pricing contingent claims for a large investor (i.e. the coefficients of the price equation can also depend on the wealth process of the hedger) in an incomplete market where the portfolios are constrained. We formulate this problem so as to find the minimal solution of forward-backward stochastic differential equations (FBSDEs) with constraints. We use the penalization method to construct a sequence of FBSDEs without constraints, and we show that the solutions of these equations converge to the minimal solution we are interested in.},
  Copyright                = {Copyright © 1998 Applied Probability Trust},
  ISSN                     = {00018678},
  Jstor_articletype        = {research-article},
  Jstor_formatteddate      = {Mar., 1998},
  Language                 = {English},
  Publisher                = {Applied Probability Trust},
  Url                      = {http://www.jstor.org/stable/1427886}
}

@Book{Cap-Kopp,
  Title                    = {{Measure, Integral and Probability}},
  Author                   = {Marek Capinski and Ekkehard Kopp},
  Publisher                = {Springer-Verlag},
  Year                     = {2003},

  File                     = {:/home/ludovic/References/Books/Probability &Stochastic/Measure__Integral_and_Probability.pdf:PDF},
  Owner                    = {ludovic},
  Timestamp                = {2011.06.24}
}

@Article{Carr-Mad05,
  Title                    = {A note on sufficient conditions for no arbitrage},
  Author                   = {Peter Carr and Dilip Madan},
  Journal                  = {Finance Research Letters},
  Year                     = {2005},
  Number                   = {3},
  Pages                    = {125-130},
  Volume                   = {2},

  Owner                    = {ludovic},
  Timestamp                = {2014.11.07}
}

@Unpublished{Che-Hu,
  Title                    = { Optimal {C}onsumption and {I}nvestment in {I}ncomplete {M}arkets with {G}eneral {C}onstraints},
  Author                   = {Patrick Cheridito and Ying Hu},
  Note                     = {Available at: arXiv:1010.0080v2},

  Month                    = {Oct},
  Year                     = {2010},

  Abstract                 = {We study an optimal consumption and investment problem in a possibly in- complete market with general, not necessarily convex, stochastic constraints. We give explicit solutions for investors with exponential, logarithmic and power utility. Our approach is based on martingale methods which rely on recent results on the existence and uniqueness of solutions to BSDEs with drivers of quadratic growth.},
  File                     = {:/home/ludovic/References/Papers/Che-Hu.pdf:PDF;:home/ludovic/US-Msc-thesis/Papers/Che-Hu.pdf:PDF}
}

@Unpublished{transport,
  Title                    = {Model-free Pricing and Hedging},
  Author                   = {Patrick Cheridito and Michael Kupper and Ludovic Tangpi},
  Note                     = {In preparation},

  Owner                    = {ludovic},
  Timestamp                = {2015.01.22}
}

@Unpublished{Che-Kup-Vol,
  Title                    = {Conditional {A}nalysis on $\mathbb{R}^d$},
  Author                   = {Patrick Cheridito and M. Kupper and N. Vogelpoth},
  Note                     = {Preprint},
  Year                     = {2012},

  File                     = {:/home/ludovic/References/Papers/Che-Kup-Vol.pdf:PDF},
  Owner                    = {ludovic},
  Timestamp                = {2011.02.01}
}

@Unpublished{Che-Sta,
  Title                    = {{Existence, Minimality and Approximation of Solutions to BSDEs with Convex Drivers}},
  Author                   = {Patrick Cheridito and Mitja Stadje},
  Note                     = {Preprint},
  Year                     = {2011},

  Owner                    = {ludovic},
  Timestamp                = {2011.10.31}
}

@MastersThesis{Trust,
  Title                    = {Pricing and {Hedging Asian Options Using Monte Carlo and Integral Transform}},
  Author                   = {Trust Chibawara},
  School                   = {Stellenbosch University},
  Year                     = {2009},
  Month                    = {December},

  Abstract                 = {In this thesis, we discuss and apply the Monte Carlo and integral transform methods in pricing options. These methods have proved to be very effective in the valuation of options especially when acceleration techniques are introduced. By first pricing European call options we have motivated the use of these methods in pricing arithmetic Asian options which have proved to be difficult to price and hedge under the Black−Scholes framework. The arithmetic average of the prices in this framework, is a sum of correlated lognormal distributions whose distribution does not admit a simple analytic expression. However, many approaches have been reported in the academic literature for pricing these options. We provide a hedging strategy by manipulating the results by Geman and Yor [42] for continuous fixed strike arithmetic Asian call options. We then derive a double Laplace transform formula for pricing continuous Asian call options following the approach by Fu et al. [39]. By applying the multi-Laguerre and iterated Talbot inversion techniques for Laplace transforms to the resulting pricing formula we obtain the option prices. Finally, we discuss the shortcomings of using the Laplace transform in pricing options.},
  File                     = {:/home/ludovic/US-Msc-thesis/Papers/Trust.pdf:PDF},
  Owner                    = {ludovic},
  Timestamp                = {2010.11.27}
}

@InCollection{Chi,
  Title                    = {{Martingale Ideology in the Theory of Controlled Stochastic Processes}},
  Author                   = {Chitashvili, R.},
  Booktitle                = {Probability Theory and Mathematical Statistics},
  Publisher                = {Springer Berlin / Heidelberg},
  Year                     = {1983},
  Editor                   = {Prokhorov, Jurii and It\^{o}, Kiyosi},
  Note                     = {10.1007/BFb0072905},
  Pages                    = {73-92},
  Series                   = {Lecture Notes in Mathematics},
  Volume                   = {1021},

  Affiliation              = {Institute of Economics &amp; Law of the Georgian SSP. Acad. of Sci. Tbilisi}
}

@Unpublished{Chr-Viens11,
  Title                    = {{Stochastic Volatility and Option Pricing with Long-Memory in Discrete and Continuous Time}},
  Author                   = {Alexandra Chronopoulou and Frederi Viens},
  Note                     = {Preprint},
  Year                     = {2011},

  File                     = {:/home/ludovic/References/Papers/Chr-Viens11.pdf:PDF},
  Owner                    = {ludovic},
  Timestamp                = {2012.04.02}
}

@Article{Chr-Viens,
  Title                    = {{Estimation and Pricing under Long-Memory Stochastic Volatility}},
  Author                   = {Chronopoulou, Alexandra and Viens, Frederi},
  Journal                  = {Ann. Finance},
  Year                     = {2010},
  Note                     = {Springer Berlin / Heidelberg},
  Pages                    = {1-25},

  Affiliation              = {Purdue University Department of Statistics 150 N. University St. West Lafayette IN 47907-2067 USA},
  File                     = {:/home/ludovic/References/Papers/Chr-Viens.pdf:PDF},
  ISSN                     = {1614-2446},
  Keyword                  = {Business and Economics},
  Owner                    = {ludovic},
  Publisher                = {Springer Berlin / Heidelberg},
  Timestamp                = {2011.08.24}
}

@Book{Chu-Wil,
  Title                    = {{Introduction to Stochastic Integration}},
  Author                   = {K. L. Chung and R. J. Williams},
  Publisher                = {Birk\"auser Boston, Base, Berlin},
  Year                     = {1990},

  File                     = {:Books/Stochastic Analysis/Chung_K.L.,_Williams_R.J._Introduction_to_stochastic_integration__1990.djvu:Djvu},
  Owner                    = {ludovic},
  Quality                  = {1},
  Timestamp                = {2013.08.21}
}

@Book{Cla,
  Title                    = {{Functional Analysis, Calculus of Variations and Optimal Control}},
  Author                   = {Francis Clarke},
  Publisher                = {Springer},
  Year                     = {2013},

  Owner                    = {ludovic},
  Quality                  = {1},
  Timestamp                = {2013.09.18}
}

@Unpublished{Coc-Jea-Nik,
  Title                    = { Default {T}imes, non {A}rbitrage {C}onditions and {C}hange of {P}robability {M}easures},
  Author                   = {Delia Coculescu and Monique Jeanblanc and Ashkan Nikeghbali},
  Note                     = {Available at: arXiv:0812.4064v1},
  Year                     = {2008},

  Annote                   = {Stochastic Processes and their Applications doi: 10.1016/j.spa.2010.03.015 }
}

@TechReport{cell11,
  Title                    = {{Mobile device detection based on user agent strings}},
  Author                   = {Colin Please, Ludovic Tangpi, Asha Tailor, Dario Fanucchi Byron Jacobs, Shaun Kimmelman, Graeme Hocking},
  Institution              = {University of the Witwatersrand},
  Year                     = {2011},
  Month                    = {Jan},
  Note                     = {Report of the Mathematics in Insdustry Study Group in South Africa},

  Owner                    = {ludovic},
  Timestamp                = {2011.05.30}
}

@Book{Cont-Tan,
  Title                    = {Financial {M}odelling with {J}ump {P}rocesses},
  Author                   = {Rama Cont and Peter Tankov},
  Publisher                = {Chapman and Hall/CRC},
  Year                     = {2004},

  File                     = {:/home/ludovic/References/Books/Stochastic finance/Financial Modelling with Jump Processes.djvu:Djvu;:/home/ludovic/References/Books/Finance stochastics/Financial Modelling with Jump Processes.djvu:Djvu}
}

@Article{Coq-Hu-Mem-Peng,
  Title                    = {{Filtration-Consistent Nonlinear Expectations and Related g-Expectations}},
  Author                   = {Coquet, François and Hu, Ying and Mémin, Jean and Peng, Shige},
  Journal                  = {Probability Theory and Related Fields},
  Year                     = {2002},
  Pages                    = {1-27},
  Volume                   = {123},

  Doi                      = {10.1007/s004400100172},
  File                     = {:/home/ludovic/References/Papers/Coq-Hu-Mem-Peng.pdf:PDF},
  ISSN                     = {0178-8051},
  Issue                    = {1},
  Language                 = {English},
  Owner                    = {ludovic},
  Publisher                = {Springer-Verlag},
  Url                      = {http://dx.doi.org/10.1007/s004400100172}
}

@Book{Areski,
  Title                    = {Paris-Princeton Lectures on Mathematical Finance},
  Author                   = {Areski Cousin and St\'ephane Cr\`epey and Olivier Gu\'eant and David Hobson and Monique Jeanblanc and Jean-Michel Lasry and Jean-Paul Laurent and Pierre-Louis Lions and Peter Tankov},
  Editor                   = {J.-M. Morel and F. Takens and B. Teissier},
  Publisher                = {Springer},
  Year                     = {2010},

  File                     = {:/home/ludovic/References/Books/Stochastic finance/Paris-Princeton Lecture notes on Mathematical Finance 2010.pdf:PDF},
  Owner                    = {ludovic},
  Timestamp                = {2011.02.22}
}

@Article{Cri-Man,
  Title                    = {{Solving Backward Stochastic Differential Equations Using the Cubature Method: Application to Nonlinear Pricing}},
  Author                   = {Dan Crisan and K. Manolarakis},
  Journal                  = {Progress in Analysis and its Applications},
  Year                     = {2010},
  Pages                    = {389-397},

  File                     = {:/home/ludovic/References/Papers/Cri-Man.pdf:PDF},
  Owner                    = {ludovic},
  Timestamp                = {2011.03.23}
}

@Article{Cri-Man-Tou,
  Title                    = {{On the Monte Carlo Simulation of BSDEs: an Improvement on the Malliavin Weights}},
  Author                   = {Dan Crisan and K. Manolarakis and Nizar Touzi},
  Journal                  = {Stochastic Process. Appl.},
  Year                     = {2010},

  Abstract                 = {We propose a generic framework for the analysis of Monte Carlo simulation schemes of backward SDEs. The general results are used to re-visit the convergence of the algorithm suggested by Bouchard and Touzi (2004) [6]. By keeping the higher order terms in the expansion of the Skorohod integrals resulting from the Malliavin integration by parts in [6], we introduce a variant of the latter algorithm which allows for a significant reduction of the numerical complexity. We prove the convergence of this improved Malliavin- based algorithm, and derive a bound on the induced error. In particular, we show that the price to pay for our simplification is to use a more accurate localizing function.},
  File                     = {:/home/ludovic/References/Papers/Cri-Man-Tou.pdf:PDF},
  Owner                    = {ludovic},
  Timestamp                = {2011.08.29}
}

@Article{Cvi-Sch-Wang,
  Title                    = {{Utility Maximization in Incomplete Market with Random Endowment}},
  Author                   = {Jaksa Cvitani\'c and Walter Schachermayer and Hui Wang},
  Journal                  = {Finance Stoch.},
  Year                     = {2001},
  Pages                    = {259-272},
  Volume                   = {5},

  File                     = {:/home/ludovic/References/Papers/Cvi-Sch-Wang.pdf:PDF},
  Owner                    = {ludovic},
  Timestamp                = {2012.07.11}
}

@Article{Dal-Mor-Wil,
  Title                    = {{Equivalent Martingale Measures and no-Arbitrage in Stochastic Securities Market Models}},
  Author                   = {Robert Dalang and Andrew Morton and Walter Willinger},
  Journal                  = {Stochastics and Stochastic Reports},
  Year                     = {1990},
  Number                   = {2},
  Pages                    = {185-201},
  Volume                   = {29},

  Owner                    = {ludovic},
  Quality                  = {1},
  Timestamp                = {2013.12.02}
}

@Article{Dar-Par,
  Title                    = {{Backwards SDE with Random Terminal Time and Applications to Semilinear Elliptic PDE}},
  Author                   = {Darling, R. W. R. and Pardoux, Etienne},
  Journal                  = {Annals of Probability},
  Year                     = {1997},
  Number                   = {3},
  Pages                    = {1135--1159},
  Volume                   = {25},

  Ajournal                 = {Ann. Probab.},
  Publisher                = {The Institute of Mathematical Statistics}
}

@Article{Dav-Hob07,
  Title                    = {{The range of traded option prices}},
  Author                   = {M.H.A Davis and David Hobson},
  Journal                  = {Math. Finance},
  Year                     = {2007},
  Number                   = {1},
  Pages                    = {1-14},
  Volume                   = {17},

  Owner                    = {ludovic},
  Timestamp                = {2014.11.07}
}

@Article{Dav-Obl-Rav12,
  Title                    = {{Arbitrage Bounds for Prices of Options on Realized Variace}},
  Author                   = {M.H.A Davis and Jan Ob{\l}{\'o}j and V. Raval},
  Journal                  = {Math. Finance},
  Year                     = {2012},
  Note                     = {to appear},

  Owner                    = {ludovic},
  Timestamp                = {2014.11.07}
}

@Article{Dav-Var,
  Title                    = {{Dynamic Programming Conditions for Partially Observed Stochastic Systems}},
  Author                   = {M. Davis and P. Varaiya},
  Journal                  = {SIAM J. Control},
  Year                     = {1973},
  Pages                    = {226--261},
  Volume                   = {11},

  Owner                    = {ludovic},
  Timestamp                = {2011.08.14}
}

@Book{Del12,
  Title                    = {{Monetary Utility Functions}},
  Author                   = {Freddy Delbaen},
  Publisher                = {Osaka University Press},
  Year                     = {2012},

  Owner                    = {ludovic},
  Timestamp                = {2013.03.15}
}

@InCollection{Del06,
  Title                    = {{The Structure of m-Stable Sets in Particular the Sets of Risk Neutral Measures}},
  Author                   = {Freddy Delbaen},
  Booktitle                = {In Memoriam Paul-Andre M\'eyer},
  Publisher                = {Springer Berlin / Heidelberg},
  Year                     = {2006},
  Pages                    = {215-258},
  Series                   = {Lecture Notes in Mathematics},
  Volume                   = {1874},

  File                     = {:/home/ludovic/References/Papers/Del06.pdf:PDF},
  Owner                    = {ludovic},
  Timestamp                = {2013.02.22}
}

@Article{Delbaen2,
  Title                    = {Exponential {H}edging and {E}ntropic {P}enalities},
  Author                   = {F Delbaen and P Grandits and R Rheinlander and D Samperi and M Schweizer},
  Journal                  = {Mathematical Finance},
  Year                     = {2002},
  Pages                    = {99-123},
  Volume                   = {12},

  File                     = {:/home/ludovic/References/Papers/Dealbean2.pdf:PDF}
}

@Article{Delbaen11,
  Title                    = {{Backward SDEs with Superquadratic Growth}},
  Author                   = {Freddy Delbaen and Ying Hu and Xiaobo Bao},
  Journal                  = {Probability Theory and Related Fields},
  Year                     = {2011},
  Pages                    = {145-192},
  Volume                   = {150},

  File                     = {:/home/ludovic/References/Papers/Del-Hu-Bao.pdf:PDF},
  ISSN                     = {0178-8051},
  Issue                    = {1-2},
  Publisher                = {Springer-Verlag}
}

@Article{Del-Peng-Gia,
  Title                    = {{Representation of the Penalty Term of Dynamic Concave Utilities}},
  Author                   = {Delbaen, Freddy and Peng, Shige and Rosazza Gianin, Emanuela},
  Journal                  = {Finance and Stochastics},
  Year                     = {2010},
  Pages                    = {449-472},
  Volume                   = {14},

  Abstract                 = {In the context of a Brownian filtration and with a fixed finite time horizon, we provide a representation of the penalty term of general dynamic concave utilities (hence of dynamic convex risk measures) by applying the theory of g -expectations.},
  Affiliation              = {ETH Department of Mathematics Zürich Switzerland},
  File                     = {:/home/ludovic/References/Papers/Del-Peng-Ros.pdf:PDF},
  Issue                    = {3},
  Keyword                  = {Mathematics and Statistics},
  Publisher                = {Springer Berlin / Heidelberg}
}

@Article{Del-Sch04,
  Title                    = {{What is a Free Lunch?}},
  Author                   = {F. Delbaen and W. Schachermayer},
  Journal                  = {Notice of the AMS},
  Year                     = {2004},
  Pages                    = {153-170},
  Volume                   = {22},

  File                     = {:/home/ludovic/References/Papers/Del-Sch.pdf:PDF},
  Owner                    = {ludovic},
  Timestamp                = {2011.03.21}
}

@Article{Del-Sch98,
  Title                    = {{The Fundamental Theorem of Asset Pricing for Unbounded Stochastic Processes}},
  Author                   = {Freddy Delbaen and Walter Schachermayer},
  Journal                  = {Math. Ann.},
  Year                     = {1998},
  Pages                    = {215-250},
  Volume                   = {312},

  File                     = {:/home/ludovic/References/Papers/Del-Sch98.pdf:PDF},
  Owner                    = {ludovic},
  Timestamp                = {2012.06.23}
}

@InProceedings{Del-Sch96,
  Title                    = {{A Compactness Principle for Bounded Sequences of Martingales with Applications}},
  Author                   = {Freddy Delbaen and Walter Schachermayer},
  Booktitle                = {Proceedings of the Seminar of Stochastic Analysis, Random Fields and Application, Progress in Probability},
  Year                     = {1996},

  File                     = {:/home/ludovic/References/Papers/Del-Scha.pdf:PDF},
  Owner                    = {ludovic},
  Timestamp                = {2011.10.14}
}

@Article{Del-Sch94,
  Title                    = {{A General Version of the Fundamental Theorem of Asset Pricing}},
  Author                   = {Freddy Delbaen and Walter Schachermayer},
  Journal                  = {Mathematische Annalen},
  Year                     = {1994},
  Pages                    = {463-520},
  Volume                   = {300},

  File                     = {:/home/ludovic/References/Papers/Del-Sch94.pdf:PDF},
  Owner                    = {ludovic},
  Timestamp                = {2012.04.05}
}

@Book{Del-Mey,
  Title                    = {{Probabilities and Potential B Theory of Martingales}},
  Author                   = {Claude Dellacherie and Paul-Andr\'{e} Meyer},
  Publisher                = {North-Holland Publishing Company},
  Year                     = {1982},

  File                     = {:/home/ludovic/References/Books/Stochastic Analysis/Del-Mey.pdf:PDF},
  Owner                    = {ludovic},
  Timestamp                = {2012.03.14}
}

@Unpublished{Delong,
  Title                    = {{Applications of Time-Delayed Backward Stochastic Differential Equations to Pricing, Hedging and Portfolio Management}},
  Author                   = {Lukasz Delong},
  Note                     = {Preprint},

  Abstract                 = {In this paper we investigate novel applications of a new class of equations which we call time-delayed backward stochastic differential equations. Time-delayed BSDEs may arise in finance when we want to find an investment strategy and an investment portfolio which should replicate a liability or meet a target depending on the applied strategy or the past values of the portfolio. In this setting, a managed investment portfolio serves simultaneously as the underlying security on which the liability/target is contingent and as a replicating portfolio for that liability/target. This is usu- ally the case for capital-protected investments and performance-linked pay-offs. We give examples of pricing, hedging and portfolio management problems (asset-liability management problems) which could be investigated in the framework of time-delayed BSDEs. Our motivation comes from life insurance and we focus on participating con- tracts and variable annuities. We derive the corresponding time-delayed BSDEs and solve them explicitly or at least provide hints how to solve them numerically. We give a financial interpretation of the theoretical fact that a time-delayed BSDE may not have a solution or may have multiple solutions.},
  File                     = {:/home/ludovic/References/Papers/Del.pdf:PDF},
  Keywords                 = {backward stochastic differential equations, asset-liability management, participating contracts, variable annuities, profit-sharing schemes, bonus schemes, capital-protected investments, performance-linked pay-offs.},
  Owner                    = {ludovic},
  Timestamp                = {2011.08.30}
}

@Article{Del-Imk,
  Title                    = {{Backward Stochastic Differential Equations with Time Delayed Generators - Results and Counterexamples}},
  Author                   = {Lukasz Delong and Peter Imkeller},
  Journal                  = {Ann. Appl. Probab.},
  Year                     = {2010},
  Pages                    = {1512-1536},
  Volume                   = {20},

  File                     = {:/home/ludovic/References/Papers/Del-Imk.ps:PostScript},
  Owner                    = {ludovic},
  Timestamp                = {2011.08.29}
}

@Article{Det-Sca,
  Title                    = {{Conditional and Dynamic Convex Risk Measures}},
  Author                   = {Kai Detlefsen and Giacomo Scandolo},
  Journal                  = {Finance Stoch.},
  Year                     = {2005},
  Pages                    = {539 – 561},
  Volume                   = {9},

  File                     = {:/home/ludovic/References/Papers/Det-Sca.pdf:PDF},
  Owner                    = {ludovic},
  Timestamp                = {2013.04.19}
}

@Article{Dol-Mey,
  Title                    = {{In\'{e}galites de Normes avec Poids}},
  Author                   = {Catherine Dol\'{e}ans-Dade and Paul-Andr\'{e} Meyer},
  Journal                  = {S\'{e}minaire de Probabilit\'{e}s (Strasbourg)},
  Year                     = {1979},
  Pages                    = {313-331},
  Volume                   = {13},

  File                     = {:/home/ludovic/References/Papers/Dol-Mey.pdf:PDF},
  Owner                    = {ludovic},
  Timestamp                = {2012.05.04}
}

@Article{DHK1101,
  Title                    = {{Minimal Supersolutions of Convex BSDEs}},
  Author                   = {Samuel Drapeau and Gregor Heyne and Michael Kupper},
  Journal                  = {Annals of Probability},
  Year                     = {2013},
  Number                   = {6},
  Pages                    = {3697-4427},
  Volume                   = {41},

  Owner                    = {ludovic},
  Timestamp                = {2014.05.05}
}

@Unpublished{DHK2013,
  Title                    = {{Minimal Supersolutions of 2BSDE}},
  Author                   = {Samuel Drapeau and Gregor Heyne and Michael Kupper},
  Note                     = {Preprint},
  Year                     = {2013},

  Owner                    = {ludovic},
  Timestamp                = {2014.03.18}
}

@Article{Dra-Kup,
  Title                    = {{Risk Preferences and their Robust Representation}},
  Author                   = {Samuel Drapeau and Michael Kupper},
  Journal                  = {Mathematics of Operations Research},
  Year                     = {2013},
  Number                   = {1},
  Pages                    = {28-62},
  Volume                   = {38},

  File                     = {:/home/ludovic/References/Papers/Dra-Kup.pdf:PDF},
  Owner                    = {ludovic},
  Timestamp                = {2012.12.13}
}

@Unpublished{tarpodual,
  Title                    = {{Dual Representation of Minimal Supersolutions of Convex BSDEs}},
  Author                   = {Samuel Drapeau and Michael Kupper and Emanuela Rosazza Gianin and Ludovic Tangpi},
  Note                     = {Forthcoming in Annales de l'Institut Henry Poincar\'e (B)},
  Year                     = {2014},

  Owner                    = {ludovic},
  Quality                  = {1},
  Timestamp                = {2013.08.07}
}

@Book{Duffie,
  Title                    = {{Security Markets Stochastic Models}},
  Author                   = {Darrell Duffie},
  Publisher                = {Academic Press},
  Year                     = {1988},

  File                     = {:/home/ludovic/References/Books/Stochastic finance/Security markets stochastic models -  Darrell Duffie.pdf:PDF},
  Owner                    = {ludovic},
  Timestamp                = {2011.06.12}
}

@Article{Duf-Eps,
  Title                    = {{Stochastic Differential Utility}},
  Author                   = {Darrell Duffie and L. Epstein},
  Journal                  = {Econometrica},
  Year                     = {1992},
  Pages                    = {353 - 394},
  Volume                   = {60},

  Owner                    = {ludovic},
  Quality                  = {1},
  Timestamp                = {2013.08.20}
}

@Book{Duf,
  Title                    = {{Introduction to C++ for Financial Engineers}},
  Author                   = {Daniel J. Duffy},
  Publisher                = {John Wiley and Sons},
  Year                     = {2006},

  File                     = {:/home/ludovic/References/Books/Finance/Introduction To C++ For Financial Engineers.pdf:PDF},
  Owner                    = {ludovic},
  Pages                    = {441},
  Timestamp                = {2011.06.12}
}

@Book{Dun-Sch,
  Title                    = {{Linear Operators Part I: General Theory}},
  Author                   = {Nelson Dunford and Jacob T. Schwartz},
  Publisher                = {Interscience Publishers, Inc., New York},
  Year                     = {1967},

  File                     = {:Books/Math_gen/Nelson_Dunford,_Jacob_T._Schwartz_Linear_Operators,_Part_I_General_Theory_Wiley_Classics_Library__1988.djvu:Djvu},
  Owner                    = {ludovic},
  Quality                  = {1},
  Timestamp                = {2013.10.25}
}

@Book{Dup-Ste,
  Title                    = {{Stochastic Modeling in Economics and Finance}},
  Author                   = {Jitka Dupacov\'{a} and Josef Step\'{a}n},
  Editor                   = {Panos M. Pardalos and Donald Hearn},
  Publisher                = {Kluwer Academic Publishers},
  Year                     = {2003},

  File                     = {:/home/ludovic/References/Books/Finance stochastics/Stochastic Modeling in Economics and Finance - Jan Hurt.pdf:PDF},
  Owner                    = {ludovic},
  Timestamp                = {2011.06.24}
}

@Book{Eke-Tem,
  Title                    = {{Convex Analysis and Variational Problems}},
  Author                   = {Ivar Ekeland and R. T\'{e}mam},
  Publisher                = {North-Holland Publishing Company},
  Year                     = {1976},

  File                     = {:/home/ludovic/References/Books/convex analysis/Ivar_Ekeland,_Roger_Témam_Convex_analysis_and_variational_problems__1987.djvu:Djvu},
  Owner                    = {ludovic},
  Timestamp                = {2012.07.11}
}

@Article{Kar-etal,
  Title                    = {{Reflected Solutions of Backward SDE's, and Related Obstacle Problems for PDE's}},
  Author                   = {Nicole {El Karoui} and C Kapoudjian and Emmanuel Pardoux and Shige Peng and Marie Claire Quenez},
  Journal                  = {Ann. Probab.},
  Year                     = {1997},
  Pages                    = {702-737},
  Volume                   = {25},

  File                     = {:/home/ludovic/References/Papers/Ka-Kap-Par-Peng-Que.pdf:PDF},
  Owner                    = {ludovic},
  Timestamp                = {2012.01.26}
}

@Article{Karoui-Peng-Que,
  Title                    = {Backward {S}tochastic {D}ifferential {E}quations in {F}inance},
  Author                   = {Nicole {El Karoui} and S Peng and M.C Quenez},
  Journal                  = {Math. Finance},
  Year                     = {1997},
  Pages                    = {1-71},
  Volume                   = {7},

  Abstract                 = {We are concerned with different properties of backward stochastic differential equations and their applications to finance. These equations, first introduced by Pardoux and Peng (1990), are useful for the theory of contingent claim valuation, especially cases with constraints and for the theory of recursive utilities, introduced by Duffie and Epstein (1992a, 1992b).},
  File                     = {:/home/ludovic/References/Papers/Karoui-Peng-Que.pdf:PDF;:home/ludovic/US-Msc-thesis/Papers/Karoui-Peng-Que.pdf:PDF}
}

@Article{Kar-Peng-Que01,
  Title                    = {{A Dynamic Maximum Principle for the Optimization of Recursive Utilities under Constraints}},
  Author                   = {Nicole {El Karoui} and Shige Peng and Marie-Claire Quenez},
  Journal                  = {Annals of Applied Probability},
  Year                     = {2001},
  Pages                    = {663 - 693},
  Volume                   = {11},

  Owner                    = {ludovic},
  Quality                  = {1},
  Timestamp                = {2013.08.20}
}

@Article{Karoui-Que,
  Title                    = {Dynamic {P}rogramming and {P}ricing of {C}ontigent {C}laims in {I}ncomplete {M}arkets},
  Author                   = {Nicole {El Karoui} and M. C Quenez},
  Journal                  = {SIAM Journal on Control and Optimization},
  Year                     = {1991},
  Pages                    = {29-66},
  Volume                   = {31},

  Issue                    = {1}
}

@Article{Kar-Rav,
  Title                    = {CASH SUBADDITIVE RISK MEASURES AND INTEREST RATE AMBIGUITY},
  Author                   = {El Karoui, Nicole and Ravanelli, Claudia},
  Journal                  = {Mathematical Finance},
  Year                     = {2009},
  Number                   = {4},
  Pages                    = {561--590},
  Volume                   = {19},

  File                     = {:/home/ludovic/References/Papers/Kar-Rav.pdf:PDF},
  ISSN                     = {1467-9965},
  Keywords                 = {risk measures, Fenchel-Legendre transform, model uncertainty, inf-convolution, backward stochastic differential equations},
  Publisher                = {Blackwell Publishing Inc}
}

@Article{Kar-Rou,
  Title                    = {Pricing via {U}tility {M}aximization and {E}ntropy},
  Author                   = {Nicole {El Karoui} and Richard Rouge},
  Journal                  = {Mathematical Finance},
  Year                     = {2000},
  Pages                    = {259-276},
  Volume                   = {10},

  File                     = {:/home/ludovic/References/Papers/Kar-Rou.pdf:PDF;:home/ludovic/US-Msc-thesis/Papers/Kar-Rou.pdf:PDF}
}

@PhdThesis{Eyr,
  Title                    = {{EDSR et EDSPR avec Grossissement de Filtration, Probl\`ermes d'Asym\'etrie d'Information et de Couverture sur les March\'es Financiers}},
  Author                   = {Anne Eyraud-Loisel},
  School                   = {Universit\'e Paul Sabatier Toulouse III},
  Year                     = {2005},

  File                     = {:/home/ludovic/References/Thesis/Eyr.pdf:PDF;:/home/ludovic/References/Papers/Eyr.pdf:PDF},
  Owner                    = {ludovic},
  Timestamp                = {2011.02.20}
}

@Book{Fol-Sch,
  Title                    = {{Stochastic Finance: An Introduction in Discrete Time}},
  Author                   = {Hans F\"ollmer and Alexander Schied},
  Publisher                = {Walter de Gruyter},
  Year                     = {2004},

  Address                  = {Berlin, New York},
  Edition                  = {2},

  File                     = {:Books/Stochastic finance/Hans_Follmer,_Alexander_Schied_Stochastic_Finance_An_Introduction_in_Discrete_Time__2004.pdf:PDF},
  Owner                    = {ludovic},
  Timestamp                = {2011.12.08}
}

@Unpublished{Fai-Mat-Mnif,
  Title                    = {{Robust Utility Maximization Problem with a General Penalty Term}},
  Author                   = {Wahid Faidi and Anis Matoussi and Mohamed Mni},
  Note                     = {Preprint},
  Year                     = {2013},

  Owner                    = {ludovic},
  Quality                  = {1},
  Timestamp                = {2013.08.22}
}

@Book{Fer,
  Title                    = {{Stochastic Portfolio Theory}},
  Author                   = {Robert Fernholz},
  Editor                   = {B. Rozovski\^i and Marc Yor},
  Publisher                = {Springer},
  Year                     = {2002},

  File                     = {:Books/Stochastic finance/Stochastic Portfolio Theory.pdf:PDF},
  Owner                    = {ludovic},
  Timestamp                = {2011.06.24}
}

@Article{Fou-Pap-Sir,
  Title                    = {{Mean-Reverting Stochastic Volatility}},
  Author                   = {Jean-Pierre Fouque and George Papanicolaous and Ronnie Sircar},
  Journal                  = {Int. J. Theor. Appl. Finance},
  Year                     = {2000},
  Pages                    = {101-142},
  Volume                   = {3},

  Owner                    = {ludovic},
  Timestamp                = {2012.04.28}
}

@Book{Fra-Har-Sta,
  Title                    = {{Measuring Risk in Complex Stochastic Systems}},
  Author                   = {J. Franke and Wolfgand H\"{a}rdle and Gerhard Stahl},
  Publisher                = {Seminar Berlin-Paris},
  Year                     = {2000},

  File                     = {:/home/ludovic/References/Books/Finance stochastics/Measuring Risk in Complex Stochastic Systems - J. Franke, W. Hardle, G. Stahl.pdf:PDF},
  Owner                    = {ludovic},
  Timestamp                = {2011.06.12},
  Url                      = {http://www.xplore-stat.de/ebooks/ebooks.html}
}

@Article{Frei-Reis,
  Title                    = {{A financial market with interacting investors: does an equilibrium exist?}},
  Author                   = {Frei, Christoph and Dos Reis, Gonçalo},
  Journal                  = {Mathematics and Financial Economics},
  Year                     = {2011},

  Month                    = {Feb},
  Note                     = {10.1007/s11579-011-0039-0},
  Pages                    = {1-22},

  Abstract                 = {While trading on a financial market, the agents we consider take the performance of their peers into account. By maximizing individual utility subject to investment constraints, the agents may ruin each other even unintentionally so that no equilibrium can exist. However, when the agents are willing to waive little expected utility, an approximated equilibrium can be established. The study of the associated backward stochastic differential equation (BSDE) reveals the mathematical reason for the absence of an equilibrium. Presenting an illustrative counterexample, we explain why such multidimensional quadratic BSDEs may not have solutions despite bounded terminal conditions and in contrast to the one-dimensional case.},
  Affiliation              = {Department of Mathematical and Statistical Sciences, University of Alberta, Edmonton, AB T6G 2G1, Canada},
  File                     = {:/home/ludovic/References/Papers/Frei-Reis.pdf:PDF},
  ISSN                     = {1862-9679},
  Keyword                  = {Business and Economics},
  Owner                    = {ludovic},
  Publisher                = {Springer Berlin / Heidelberg},
  Timestamp                = {2011.04.26},
  Url                      = {http://dx.doi.org/10.1007/s11579-011-0039-0}
}

@Unpublished{Frei-Mal-Sch,
  Title                    = {{Convexity Bounds for BSDE Solutions, with Applications to Indifference Valuation}},
  Author                   = {Christoph Frei and Semyon Malamud and Martin Schweizer},
  Note                     = {Preprint},
  Year                     = {2010},

  Abstract                 = {We consider backward stochastic differential equations (BSDEs) with a particular quadratic generator and study the behaviour of their solu- tions when the probability measure is changed, the filtration is shrunk, or the underlying probability space is transformed. Our main results are upper bounds for the solutions of the original BSDEs in terms of solutions to other BSDEs which are easier to solve. We illustrate our results by applying them to exponential utility indifference valuation in a multidimensional Itˆ process setting.},
  File                     = {:/home/ludovic/References/Papers/Frei-Mal-Sch.pdf:PDF},
  Owner                    = {ludovic},
  Timestamp                = {2011.02.20}
}

@InCollection{Frei-Sch,
  Title                    = {{Exponential utility indifference valuation in a general semimartingale model }},
  Author                   = {Christoph Frei and Martin Schweizer},
  Booktitle                = {Optimality and Risk — Modern Trends in Mathematical Finance. The Kabanov Festschrif},
  Publisher                = {Delbaen, F., R\`asonyi, M. and Stricker,},
  Year                     = {2009},

  File                     = {:/home/ludovic/References/Papers/Frei-Sch.pdf:PDF},
  Owner                    = {ludovic},
  Timestamp                = {2011.02.01}
}

@PhdThesis{Frei,
  Title                    = {Exponential {U}tility {I}ndifference {V}aluation: {C}orrelation, {S}emimartingales, {B}{S}{D}{E}s, {C}onvergence},
  Author                   = {Christoph Marin Frei},
  School                   = {ETH Zurich},
  Year                     = {2009},
  Type                     = {Ph.D},

  Abstract                 = {Exponential utility indifference valuation assigns to contingent claims H due at time T a value for an investor with exponential utility preferences. The indifference value ht for H at time t ? [0, T ] makes the investor indifferent, in terms of maximal expected utility, between not selling H and selling H for the amount ht . This thesis studies the form of the implicitly defined ht and properties of the process (ht )0?t?T in four chapters, whose keywords are ?correlation?, ?semimartingales?, ?BSDEs? and ?convergence?. Correlation. We consider a two-dimensional Brownian model where H depends on a nontradable asset stochastically correlated with the traded asset available for hedging. The use of martingale arguments yields a structurally explicit formula for ht , even with a fairly general stochastic correlation ? between the two Brownian motions. After a change of measure, ht enjoys a monotonicity property in |?|. This is the reason why we can generalise the explicit formula for ht known from the literature for constant ?. Semimartingales. Also in a general semimartingale model, we can derive a formula for ht , although it is much less explicit than in the Brownian model. A second result in this general setting is a description of (ht )0?t?T as the unique solution (in a suitable class of processes) of a backward stochastic differential equation (BSDE). The key to both results is what we call the fundamental entropy representation of H, a decomposition of H into a hedged and an unhedged part depending on the investor?s risk aversion. BSDEs. In a multidimensional Brownian model, we study in more detail the type of BSDE related to (ht )0?t?T , with the goal of deriving bounds for (ht ). We transform such BSDEs by changing the probability measure, shrinking the filtration, or symmetrising the underlying probability space. These transformations yield bounds for the solutions of the original BSDEs in terms of solutions to other BSDEs which are easier to solve. Convergence. Revisiting the two-dimensional Brownian model with sto- chastic correlation, we derive an explicitly computable sequence that con- verges to ht . This result complements the structurally explicit formula for ht . It is based on a convergence theorem for quadratic BSDEs, which we prove in a general continuous filtration.},
  File                     = {:/home/ludovic/References/Thesis/Frei.pdf:PDF}
}

@Article{Gal-HL-Tou14,
  Title                    = {A stochastic control approach to no-arbitrage bounds given marginals, with an application to Lookback options},
  Author                   = {A. Galichon and P. Henry-Labord\`{e}re and N. Touzi},
  Journal                  = {Annals of Applied Probability},
  Year                     = {2014},
  Number                   = {1},
  Pages                    = {312-336},
  Volume                   = {24},

  Owner                    = {ludovic},
  Timestamp                = {2014.11.07}
}

@Book{Gla,
  Title                    = {{Monte Carlo Methods in Financial Ingenieering}},
  Author                   = {Paul Glasserman},
  Editor                   = {B. Ruzovskii and Marc Yor},
  Publisher                = {Springer},
  Year                     = {2004},

  File                     = {:/home/ludovic/References/Books/Finance/Monte Carlo Methods in Financal ModelingGlasserman.pdf:PDF},
  Owner                    = {ludovic},
  Pages                    = {614},
  Timestamp                = {2011.06.12}
}

@InProceedings{Gob-Lab2,
  Title                    = {{A Numerical Algorithm for Solving BSDEs}},
  Author                   = {Emmanuel Gobet and Celine Labart},
  Booktitle                = {Applied Mathematics and Mechanics},
  Year                     = {2007},

  Abstract                 = {We present and analyze a numerical algorithm to solve BSDEs based on Picard’s iterations and on a sequential control variate method. Its convergence is geometric. Moreover, our algorithm provides a regular solution w.r.t. time and space.},
  File                     = {:/home/ludovic/References/Papers/Gob-Lab2.pdf:PDF},
  Owner                    = {ludovic},
  Timestamp                = {2011.02.20}
}

@Article{Har-Kre,
  Title                    = {{Martingales and Arbitrage in Multiperiod Securities Markets}},
  Author                   = {Harrison, J. Michael and Kreps, David M.},
  Journal                  = {Journal of Economic Theory},
  Year                     = {1979},

  Month                    = {June},
  Number                   = {3},
  Pages                    = {381-408},
  Volume                   = {20},

  Abstract                 = {No abstract is available for this item.},
  Url                      = {http://ideas.repec.org/a/eee/jetheo/v20y1979i3p381-408.html}
}

@Article{Hen,
  Title                    = {The {I}mpact of the {M}arket {P}ortfolio on the {V}aluation, {I}ncentives and {O}ptimalty of {E}xecutive {S}tock {O}ptions},
  Author                   = {V. Henderson},
  Journal                  = {Quantitative Finance},
  Year                     = {2005},
  Pages                    = {35-47},
  Volume                   = {5}
}

@Article{Hen-Hob,
  Title                    = {Substitute {H}edging},
  Author                   = {V. Henderson and D Hobson},
  Journal                  = {Risk},
  Year                     = {2002},
  Pages                    = {71-75},
  Volume                   = {15}
}

@Article{Hey-Kup-Mai,
  Title                    = {{Minimal Supersolutions of BSDEs with Lower Semicontinuous Generators}},
  Author                   = {Gregor Heyne and Michael Kupper and Christoph Mainberger},
  Journal                  = {Annales de l'Institut Henri Poincare (B) Probability and Statistics},
  Year                     = {2014},
  Number                   = {2},
  Volume                   = {50},

  Owner                    = {ludovic},
  Timestamp                = {2014.08.20}
}

@Unpublished{gammaconstraints,
  Title                    = {{Minimal Supersolutions of Convex BSDEs under Constraints}},
  Author                   = {Gregor Heyne and Michael Kupper and Christoph Mainberger and Ludovic Tangpi},
  Note                     = {Preprint},
  Year                     = {2013},

  Owner                    = {ludovic},
  Timestamp                = {2014.04.29}
}

@Article{Hob98,
  Title                    = {{Robust Hedging of the lookback option}},
  Author                   = {David Hobson},
  Journal                  = {Finance Stoch.},
  Year                     = {1998},
  Pages                    = {329-347},
  Volume                   = {2},

  Owner                    = {ludovic},
  Timestamp                = {2014.11.07}
}

@Article{Hob-Kli12,
  Title                    = {Model independent Hedging strategies for Variace swaps},
  Author                   = {David Hobson and M. Klimmek},
  Journal                  = {Finance Stoch.},
  Year                     = {2012},
  Number                   = {6},
  Pages                    = {611-649},
  Volume                   = {16},

  Owner                    = {ludovic},
  Timestamp                = {2014.11.07}
}

@Article{Hod-Neu,
  Title                    = {Optimal {R}eplication of {C}ontingent {C}laim under {T}ransaction {C}osts},
  Author                   = {S. D Hodges and A Neuberger},
  Journal                  = {Revue Futures Markets},
  Year                     = {1989},
  Pages                    = {222-239},
  Volume                   = {8},

  Issue                    = {1}
}

@Article{Hor-etal,
  Title                    = {{Forward Backward Systems for Expected Utility Maximization}},
  Author                   = {Ulrich Horst and Ying Hu and Peter Imkeller and Anthony R\'eveillac and Jianing Zhang},
  Journal                  = {Stochastic Processes and their Applications},
  Year                     = {2014},
  Number                   = {5},
  Pages                    = {1813-1848},
  Volume                   = {124},

  Owner                    = {ludovic},
  Timestamp                = {2014.10.23}
}

@Article{Hu-Imk-Mul,
  Title                    = {Utility {M}aximization in {I}ncomplete {M}arkets},
  Author                   = {Ying Hu and Peter Imkeller and Matthias M\"uller},
  Journal                  = {Ann. Appl. Probab.},
  Year                     = {2005},
  Pages                    = {1691-1712},
  Volume                   = {15},

  File                     = {:/home/ludovic/References/Papers/Hu-Imk-Mul.pdf:PDF;:home/ludovic/US-Msc-thesis/Papers/Hu-Imk-Mul.pdf:PDF}
}

@Book{Hull,
  Title                    = {Options, {F}utures, and other {D}erivative {S}ecurities},
  Author                   = {John Hull},
  Publisher                = {Prentice-Hall},
  Year                     = {1989},
  Edition                  = {1}
}

@Unpublished{Imk,
  Title                    = {Malliavin's {C}alculus and {A}pplications in {S}tochastic {C}ontrol and {F}inance},
  Author                   = {Peter Imkeller},
  Note                     = {Institute of Mathematics, Polish Academy of Sciences Humbolt-Universit{\"a}t zu Berlin},
  Year                     = {2008},

  File                     = {:/home/ludovic/References/Books/Stochastic Analysis/warschau_malliavin2.pdf:PDF}
}

@Unpublished{Imk-Kup-Rev-Zha,
  Title                    = {{FBSDE Representation of Robust Utility Maximization Problems with Endowment}},
  Author                   = {Peter Imkeller and Michael Kupper and Anthony R\'eveillac and Jianing Zhang},
  Note                     = {In preparation},

  Owner                    = {ludovic},
  Timestamp                = {2012.10.30}
}

@Unpublished{Imk-Rev-Ric,
  Title                    = {{Differentiability of Quadratic BSDE Generated by Continuous Martingales and Hedging in Incomplete Markets}},
  Author                   = {Peter Imkeller and Anthony R\'eveillac and Anja Richter},
  Note                     = {Preprint},
  Year                     = {2010},

  Abstract                 = {In this paper we consider a class of BSDE with drivers of quadratic growth, on a stochastic basis generated by continuous local martingales. We first derive the Markov property of a forward-backward system (FBSDE) if the generating martingale is a strong Markov process. Then we establish the differentiability of a FBSDE with respect to the initial value of its forward component. This enables us to obtain the main result of this article which from the perspective of a utility optimization interpretation of the underlying control problem on a financial market takes the following form. The control process of the BSDE steers the system into a random liability depending on a market external uncertainty and this way describes the optimal derivative hedge of the liability by investment in a capital market the dynamics of which is described by the forward component. This delta hedge is described in a key formula in terms of a derivative functional of the solution process and the correlation structure of the internal uncertainty captured by the forward process and the external uncertainty responsible for the market incompleteness. The formula largely extends the scope of validity of the results obtained by several authors in the Brownian setting, designed to give a genuinely stochastic representation of the optimal delta hedge in the context of cross hedging insurance derivatives generalizing the derivative hedge in the Black-Scholes model. Of course, Malliavin’s calculus needed in the Brownian setting is not available in the general local martingale framework. We replace it by new tools based on stochastic calculus techniques.},
  File                     = {:/home/ludovic/References/Papers/Imk-Rev-Ric.pdf:PDF},
  Owner                    = {ludovic},
  Timestamp                = {2011.02.20}
}

@Unpublished{Imk-Rev-Zha,
  Title                    = {Solvability and {N}umerical {S}imulation of {B}{S}{D}{E}s {R}elated to {B}{S}{P}{D}{E}s with {A}pplications to {U}tility {M}aximization},
  Author                   = {Peter Imkeller and Anthony R\'eveillac and Jianing Zhang},
  Note                     = {To appear in International Journal of Theoretical and Applied Finance. DOI No: 10.1142/S0219024911006437},
  Year                     = {2011},

  Abstract                 = {In this paper we study BSDEs arising from a special class of backward stochastic partial differential equations (BSPDEs) that is intimately related to utility maximization problems with respect to arbitrary utility functions. After providing existence and uniqueness we discuss the numerical realizability. Then we study utility maximization problems on incomplete financial markets whose dynamics are governed by continuous semimartingales. Adapting standard methods that solve the utility maximization problem using BSDEs, we give solutions for the portfolio optimization problem which involve the delivery of a liability at maturity. We illustrate our study by numerical simulations for selected examples. As a byproduct we prove existence of a solution to a very particular quadratic growth BSDE with unbounded terminal condition. This complements results on this topic obtained in [6,7,8].},
  File                     = {:/home/ludovic/References/Papers/Imk-Rev-Zha.pdf:PDF},
  Owner                    = {ludovic},
  Timestamp                = {2010.11.15}
}

@Article{Imk-Reis,
  Title                    = {Path {R}egularity and {E}xplicit {C}onvergence {R}ate for {B}{S}{D}{E} with {T}runcated {Q}uadratic {G}rowth},
  Author                   = {Peter Imkeller and Goncalo Dos Reis},
  Journal                  = {Stochastic Process. Appl.},
  Year                     = {2010},
  Pages                    = {348-379},
  Volume                   = {120},

  Abstract                 = {We consider backward stochastic differential equations with drivers of quadratic growth (qgBSDE). We prove several statements concerning path regularity and stochastic smoothness of the solution processes of the qgBSDE, in particular we prove an extension of Zhang?s path regularity theorem to the quadratic growth setting. We give explicit convergence rates for the difference between the solution of a qgBSDE and its truncation, filling an important gap in numerics for qgBSDE. We give an alternative proof of second order Malliavin differentiability for BSDE with drivers that are Lipschitz continuous (and differentiable), and then derive an analogous result for qgBSDE.},
  File                     = {:/home/ludovic/References/Papers/Imk-Reis.pdf:PDF;:home/ludovic/US-Msc-thesis/Papers/Imk-Reis.pdf:PDF}
}

@InCollection{Imk-Reis-Zha,
  Title                    = {{Results on Numerics for FBSDE with Drivers of Quadratic Growth}},
  Author                   = {Peter Imkeller and Goncalo Dos Reis and Jianing Zhang},
  Booktitle                = {Contemporary Quantitative Finance},
  Publisher                = {Springer},
  Year                     = {2010},
  Editor                   = {Alexander Chiarella, Carl; Novikov},
  Month                    = {April},
  Note                     = {Essays in Honour of Eckhard Platen},
  Pages                    = {440},

  File                     = {:/home/ludovic/References/Papers/Imk-Reis-Zha.pdf:PDF;:/home/ludovic/References/Papers/Imk-Rev-Zha.pdf:PDF},
  Owner                    = {ludovic},
  Timestamp                = {2011.03.23}
}

@Book{Jac-Pro,
  Title                    = {{Probability Essentials}},
  Author                   = {Jean Jacod and Philip Protter},
  Editor                   = {2},
  Publisher                = {Springer},
  Year                     = {2003},

  File                     = {:/home/ludovic/References/Books/Probability &Stochastic/ProbabilityEssentials.djvu:Djvu},
  Owner                    = {ludovic},
  Pages                    = {266},
  Timestamp                = {2011.07.19}
}

@Article{Jea-Cam,
  Title                    = {Progressive {E}nlargement of {F}iltrations with {I}nitial {T}imes},
  Author                   = {Monique Jeanblanc and Yann Le Cam},
  Journal                  = {Stochastic Process. Appl.},
  Year                     = {2009},
  Pages                    = {2523-2543},
  Volume                   = {119}
}

@Article{Jeu,
  Title                    = {Grossissement d'une {F}iltration et {A}pplications},
  Author                   = {T Jeulin},
  Journal                  = {S\'eminaire de Probabilit\'e XIII, Lecture Notes in Mathematics},
  Year                     = {1979},
  Pages                    = {574-609},
  Volume                   = {721}
}

@Article{Jeu-Yor,
  Title                    = {Grossissement d'une {F}iltration et {S}emimartingales: {F}ormules {E}xiplicites},
  Author                   = {T Jeulin and M Yor},
  Journal                  = {S\'eminaire de Probabilit\'e XIII, Lecture Notes in Mathematics},
  Year                     = {1978},
  Pages                    = {78-97},
  Volume                   = {649}
}

@InCollection{Jou-Sch-Tou,
  Title                    = {Law invariant risk measures have the Fatou property},
  Author                   = {Jouini, Elyès and Schachermayer, Walter and Touzi, Nizar},
  Booktitle                = {Advances in Mathematical Economics},
  Publisher                = {Springer Japan},
  Year                     = {2006},
  Editor                   = {Kusuoka, Shigeo and Yamazaki, Akira},
  Pages                    = {49-71},
  Series                   = {Advances in Mathematical Economics},
  Volume                   = {9},

  Doi                      = {10.1007/4-431-34342-3_4},
  ISBN                     = {978-4-431-34341-7},
  Keywords                 = {law-invariance; cash-invariance; Fatou and Lebesgue properties},
  Language                 = {English},
  Url                      = {http://dx.doi.org/10.1007/4-431-34342-3_4}
}

@Book{Kar-Shr,
  Title                    = {Brownian {M}otion and {S}tochastic {C}alculus},
  Author                   = {Ioannis Karatzas and Steven E. Shreve},
  Publisher                = {Springer},
  Year                     = {1988},
  Edition                  = {2},

  File                     = {:/home/ludovic/References/Books/Stochastic Analysis/karatzas_-_brownian_motion_and_stochastic_calculus__springer_1988__.djvu:Djvu}
}

@Book{Kaz,
  Title                    = {Continuous {E}xponential {M}artingales and {B}{M}{O}},
  Author                   = {Norihiko Kazamaki},
  Publisher                = {Volume 1579 of Lecture Notes in Mathematics. Springer-Verlag},
  Year                     = {1994},
  Edition                  = {1},

  File                     = {:/home/ludovic/References/Books/Stochastic Analysis/BMO.djvu:Djvu}
}

@Article{Koby,
  Title                    = {Backward {S}tochastic {D}ifferential {E}quations and {P}artial {D}ifferential {E}quations with {Q}uadratic {G}rowth},
  Author                   = {Magdalena Kobylanski},
  Journal                  = {Ann. Probab.},
  Year                     = {2000},
  Pages                    = {558-602},
  Volume                   = {28},

  Abstract                 = {We provide existence, comparison and stability results for one- dimensional backward stochastic differential equations (BSDEs) when the coefficient(or generator) F(t, Y, Z) is continuous and has a quadratic growthin Z and the terminal condition is bounded. We also give, in this framework,the links between the solutions of BSDEs set on a diffusion and viscosity or Sobolev solutions of the correspondingsemilinear partial differential equations.},
  File                     = {:/home/ludovic/References/Papers/Koby.pdf:PDF;:home/ludovic/US-Msc-thesis/Papers/Koby.pdf:PDF},
  Issue                    = {2}
}

@Article{Koby1,
  Title                    = {R\'{e}sultats d'existence et d'{U}nicit\'{e} pour les {E}quations {D}iff\'{e}rentielles {S}tochastiques {R}\'{e}trogrades avec {G}\'{e}n\'{e}rateurs \`{a} {C}roissance {Q}uadratique.},
  Author                   = {Magdalena Kobylanski},
  Journal                  = {Comptes Rendus Academiques des Sciences de Paris},
  Year                     = {1997},
  Pages                    = {81-86},
  Volume                   = {324},

  Issue                    = {2}
}

@Article{Korn,
  Title                    = {The {M}artingale {O}ptimality {P}rinciple: {T}he {B}est you can do is {E}nough},
  Author                   = {Ralf Korn},
  Journal                  = {Wilmott},
  Year                     = {2003},
  Pages                    = {61-67},
  Volume                   = {1},

  Abstract                 = {The area of application where this article is centered around is not pricing but optimal behaviour of an indi- vidual at a financial market or at any area where decisions about control actions have to be taken such as looking for optimal investment strategies, steering an airplane in an efficient way, or searching for the opti- mal velocity of a production line.},
  File                     = {:/home/ludovic/References/Papers/Korn.pdf:PDF;:home/ludovic/US-Msc-thesis/Papers/Korn.pdf:PDF}
}

@InCollection{Korn-Olaf,
  Title                    = {On Worst-case {I}nvestment with {A}pplications in {F}inance and {I}nsurance {M}athematics},
  Author                   = {Ralf Korn and Olaf Menkens},
  Booktitle                = {Interacting Stochastic Systems},
  Publisher                = {Springer Link},
  Year                     = {2005},

  Address                  = {New York},
  Editor                   = {Jean-Dominique Deuschel and Andreas Greven},
  Pages                    = {397--407},

  File                     = {:/home/ludovic/References/Papers/Korn-Olaf.pdf:PDF;:home/ludovic/US-Msc-thesis/Papers/Korn-Olaf.pdf:PDF}
}

@Article{Kra-Sch,
  Title                    = {{The Asymptotic Elasticity of Utility Functions and Optimal Investment in Incomplete Market}},
  Author                   = {D. Kramkov and Walter Schachermayer},
  Journal                  = {Ann. Appl. Probab.},
  Year                     = {1999},
  Number                   = {3},
  Pages                    = {904-950},
  Volume                   = {9},

  File                     = {:Papers/Kra-Sch.pdf:PDF},
  Owner                    = {ludovic},
  Timestamp                = {2011.11.01}
}

@Unpublished{Krrrrr,
  Title                    = {How to Write your First Paper},
  Author                   = {Steven G. Krantz},

  File                     = {:/home/ludovic/References/Papers/Krantz.pdf:PDF},
  Owner                    = {ludovic},
  Timestamp                = {2012.06.09}
}

@Article{Kry,
  Title                    = {{On the Itô–Wentzell Formula for Distribution-Valued Processes and Related Topics}},
  Author                   = {N. V. Krylov},
  Journal                  = {Probab. Theory Relat. Fields},
  Year                     = {2011},
  Pages                    = {295-319},
  Volume                   = {150},

  Owner                    = {ludovic},
  Timestamp                = {2011.05.23}
}

@InProceedings{Kun,
  Title                    = {{Stochastic Differential Equations based on L\'evy Processes and Stochastic Flows of Diffeomorphisms}},
  Author                   = {Hiroshi Kunita},
  Booktitle                = {Real and Stochastic Analysis},
  Year                     = {2004},

  Address                  = {Boston, MA},
  Note                     = {Trends Math.},
  Pages                    = {305-373},
  Publisher                = {Birk\"auser Boston},

  File                     = {:/home/ludovic/References/Papers/Kun.pdf:PDF},
  Owner                    = {ludovic},
  Timestamp                = {2012.05.09}
}

@PhdThesis{Kup,
  Title                    = {{Monetary Risk Measures for Stochastic Processes}},
  Author                   = {Michael Kupper},
  School                   = {Swiss Federal Institute of Technology},
  Year                     = {2005},

  File                     = {:/home/ludovic/References/Thesis/Kup.pdf:PDF},
  Owner                    = {ludovic},
  Timestamp                = {2012.03.14}
}

@Unpublished{Kup-Sch14,
  Title                    = {{Stability and Duality Results for Supersolutions of BSDEs Without Reference Measure}},
  Author                   = {Michael Kupper and Reinhard Schimdt},
  Note                     = {In preparation},

  Month                    = {March},
  Year                     = {2014},

  Owner                    = {ludovic},
  Timestamp                = {2014.03.19}
}

@Article{Kup-Svi,
  Title                    = {{Dual Representation of Monotone Convex Functions on $L^0$}},
  Author                   = {Michael Kupper and Gregor Svindland},
  Journal                  = {Proceedings of the AMS},
  Year                     = {2011},
  Number                   = {11},
  Pages                    = {4073--4086},
  Volume                   = {139},

  Owner                    = {ludovic},
  Quality                  = {1},
  Timestamp                = {2013.08.07}
}

@Book{Lam-Lap,
  Title                    = {{Introduction to Stochastic Calculus Applied to Finance}},
  Author                   = {Damien Lamberton and Bernard Lapeyre},
  Publisher                = {Chapman and Hall/CRC},
  Year                     = {1996},

  File                     = {:/home/ludovic/References/Books/Stochastic finance/Lamberton-Lapeyre__Introduction_to_Stochastic_Calculus_Applied_to_Finance__Stochastic_Modeling_.djvu:Djvu},
  Owner                    = {ludovic},
  Timestamp                = {2011.10.06}
}

@Article{Laz,
  Title                    = {Generalized {S}tochastic {D}ifferential {U}tility and {P}reference for {I}nformation},
  Author                   = {Ali Lazrak},
  Journal                  = {Ann. Appl. Probab.},
  Year                     = {2004},
  Pages                    = {2149-2175},
  Volume                   = {14},

  Abstract                 = {This paper develops, in a Brownian information setting, an ap- proach for analyzing the preference for information, a question that motivates the stochastic differential utility (SDU) due to Duffie and Epstein [Econometrica 60 (1992) 353–394]. For a class of backward stochastic differential equations (BSDEs) including the generalized SDU [Lazrak and Quenez Math. Oper. Res. 28 (2003) 154–180], we formulate the information neutrality property as an invariance prin- ciple when the filtration is coarser (or finer) and characterize it. We also provide concrete examples of heterogeneity in information that illustrate explicitly the nonneutrality property for some GSDUs. Our results suggest that, within the GSDUs class of intertemporal utili- ties, risk aversion or ambiguity aversion are inflexibly linked to the preference for information.},
  File                     = {:/home/ludovic/References/Papers/Laz.pdf:PDF},
  Owner                    = {ludovic},
  Timestamp                = {2010.11.22}
}

@Article{Laz-Que,
  Title                    = {{A Generalized Stochastic Differential Utility}},
  Author                   = {Lazrak, Ali and Quenez, Marie Claire},
  Journal                  = {MATHEMATICS OF OPERATIONS RESEARCH},
  Year                     = {2003},
  Number                   = {1},
  Pages                    = {154-180},
  Volume                   = {28},

  Abstract                 = {This paper generalizes, in the setting of Brownian information, the Duffie-Epstein (1992) stochastic differential formulation of intertemporal recursive utility (SDU). We provide a utility functional of state-contingent consumption plans that exhibits a local dependency with respect to the utility intensity process (the integrand of the quadratic variation) and call it the generalized SDU. This mathematical generalization of the SDU permits, in fact, more flexibility in the separation between risk aversion and intertemporal substitution and allows to model asymmetry in risk aversion. We extensively use the backward stochastic differential equation theory to give sufficient conditions for comparative and absolute risk aversion behavior as well as aversion to specific directional risk. Additionally, we discuss whether our functional exhibits monotonicity to its information filtration argument. For purposes of illustration, we provide some applications to the consumption/portfolio strategy selection problem in a complete securities market.},
  Doi                      = {10.1287/moor.28.1.154.14259},
  File                     = {:/home/ludovic/US-Msc-thesis/Papers/Laz-Que.pdf:PDF}
}

@PhdThesis{Lem,
  Title                    = {{Approximation par Projections et Simulations de Monte-Carlo des \'Equations Diff\'erentielles Stochastiques R\'etrogrades}},
  Author                   = {Jean-Philippe Lemor},
  School                   = {\'Ecole Polytechnique},
  Year                     = {2005},

  File                     = {:/home/ludovic/References/Thesis/Lem.pdf:PDF;:/home/ludovic/References/Papers/Lem.pdf:PDF},
  Owner                    = {ludovic},
  Timestamp                = {2011.02.20}
}

@Book{Bro-Sho,
  Title                    = {{Essentials of Game Theory}},
  Author                   = {Kevin Leyton-Brown and Yoav Shoham},
  Publisher                = {Morgan and Claypool},
  Year                     = {2008},

  File                     = {:/home/ludovic/References/Books/Finance/Essentials of Game Theory.pdf:PDF},
  Owner                    = {ludovic},
  Timestamp                = {2011.06.03}
}

@Book{Lip,
  Title                    = {{Schaum's Outline of Theory and Problems of General Topology}},
  Author                   = {Seymour Lipschutz},
  Publisher                = {McGraw-Hill Book Company},
  Year                     = {1965},
  Note                     = {Shaum's Outline series},

  File                     = {:/home/ludovic/References/Books/Math_gen/Shaum's Outline of General Topology.pdf:PDF},
  Owner                    = {ludovic},
  Timestamp                = {2011.06.24}
}

@Article{Ma-Pro-Yong,
  Title                    = {{Solving Forward-Backward Stochastic Differential Equations Explicitly - A Four Step Scheme}},
  Author                   = {Jin Ma and Philip Protter and Jiongmin Yong},
  Journal                  = {Probab. Theory Relat. Fields},
  Year                     = {1994},
  Pages                    = {339-359},
  Volume                   = {98},

  File                     = {:/home/ludovic/References/Papers/Ma-Pro-Yong.pdf:PDF},
  Owner                    = {ludovic},
  Timestamp                = {2011.10.26}
}

@Book{Ma-Yong,
  Title                    = {{Forward-Backward Stochastic Differential Equations and their Applications}},
  Author                   = {Jin Ma and Jiongmin Yong},
  Editor                   = {J.-M. Morel and F. Takens and B. Teissier},
  Publisher                = {Springer},
  Year                     = {2007},
  Series                   = {Lecture Notes in Mathematics},

  File                     = {:/home/ludovic/References/Books/Stochastic Analysis/Forward Backward Stochastic Differential Equations Ma-Yong.pdf:PDF;:/home/ludovic/References/Books/Ma-Yong.pdf:PDF},
  Owner                    = {ludovic},
  Pages                    = {284},
  Timestamp                = {2011.02.22}
}

@Book{Mal,
  Title                    = {{Stochastic Methods in Economics: and Finance}},
  Author                   = {A. G. Malliaris},
  Editor                   = {W. A. Brock},
  Publisher                = {Elsevier Science},
  Year                     = {1982},

  File                     = {:/home/ludovic/References/Books/Finance stochastics/Stochastic methods in economics and finance - Malliaris A.G..pdf:PDF},
  Owner                    = {ludovic},
  Timestamp                = {2011.06.24}
}

@Article{Man-Tev,
  Title                    = {{Backward Stochastic Partial Differential Equations Related to Utility Maximization and Hedging}},
  Author                   = {Mania, M. and Tevzadze, R.},
  Journal                  = {Journal of Mathematical Sciences},
  Year                     = {2008},
  Note                     = {10.1007/s10958-008-9129-9},
  Pages                    = {291-380},
  Volume                   = {153},

  Abstract                 = {We study the utility maximization problem, the problem of minimization of the hedging error and the corresponding dual problems using dynamic programming approach. We consider an incomplete financial market model, where the dynamics of asset prices are described by an ℝd-valued continuous semimartingale. Under some regularity assumptions, we derive the backward stochastic PDEs for the value functions related to these problems, and for the primal problem, we show that the strategy is optimal if and only if the corresponding wealth process satisfies a certain forward SDE. As examples we consider the mean-variance hedging problem and the cases of power, exponential, logarithmic utilities, and corresponding dual problems.},
  Affiliation              = {Georgian-American University, Business School Tbilisi Georgia},
  File                     = {:/home/ludovic/References/Papers/Man-Tev.pdf:PDF},
  ISSN                     = {1072-3374},
  Issue                    = {3},
  Keyword                  = {Mathematics and Statistics},
  Publisher                = {Springer New York}
}

@Book{Man-Yor,
  Title                    = {Random {T}imes and {E}nlargement of {F}iltrations in a {B}rownian {S}etting},
  Author                   = {R. Mansur and M. Yor},
  Publisher                = {Volume 1873 of Lecture Notes in Mathematics. Springer},
  Year                     = {2006},
  Edition                  = {1}
}

@MastersThesis{Lydienne,
  Title                    = {Cubature Methods and Applications to Option Pricing},
  Author                   = {Lydienne Matchie},
  School                   = {University of Stellenbosch},
  Year                     = {2010},
  Month                    = {Dec},

  File                     = {:/home/ludovic/References/Thesis/Lydienne.pdf:PDF;:/home/ludovic/US-Msc-thesis/Papers/Lydienne.pdf:PDF},
  Owner                    = {ludovic},
  Timestamp                = {2010.12.21}
}

@Book{Math,
  Title                    = {{Optimization Toolbox for Use with MATLAB}},
  Author                   = {The Mathwork},
  Publisher                = {The Mathwork, Inc},
  Year                     = {2002},

  File                     = {:/home/ludovic/References/Books/Optimization/Optimization Toolbox_Matlab.pdf:PDF},
  Owner                    = {ludovic},
  Timestamp                = {2011.06.24}
}

@Book{Mat-Gar,
  Title                    = {{Understanding and Using Linear Programming}},
  Author                   = {Jir\'{i} Matousek and Bernd G\"{a}rtner},
  Publisher                = {Springer},
  Year                     = {2007},
  Note                     = {Universitex},

  File                     = {:/home/ludovic/References/Books/Optimization/Understanding and using linear programming.pdf:PDF},
  Owner                    = {ludovic},
  Timestamp                = {2011.06.24}
}

@Article{Mer,
  Title                    = {{Optimum Consumption and Portfolio Rules in a Continuous-Time Model}},
  Author                   = {Robert C. Merton},
  Journal                  = {Journal of Economic Theory},
  Year                     = {1971},
  Number                   = {4},
  Pages                    = {373 - 413},
  Volume                   = {3},

  Doi                      = {DOI: 10.1016/0022-0531(71)90038-X},
  ISSN                     = {0022-0531}
}

@Article{Meyer,
  Title                    = {{Limites M\'ediales d'apr\`es Mokobodzki}},
  Author                   = {Paul-Andr\'{e} Meyer},
  Journal                  = {S\'eminaire de Probabilit\'e (Strasbourg)},
  Year                     = {1973},
  Pages                    = {198-204},
  Volume                   = {7},

  Owner                    = {ludovic},
  Timestamp                = {2014.03.18}
}

@Book{Mil,
  Title                    = {{An Introduction to Functinal Analysis}},
  Author                   = {Vitali Milman},
  Publisher                = {World},
  Year                     = {1999},

  File                     = {:/home/ludovic/References/Books/Math_gen/An Introduction to Functional Analysis Mil.pdf:PDF},
  Owner                    = {ludovic},
  Timestamp                = {2011.05.30}
}

@Book{Mol,
  Title                    = {{Theory of Random Sets}},
  Author                   = {Ilya Molchanov},
  Publisher                = {Springer},
  Year                     = {2005},

  File                     = {:/home/ludovic/References/Books/Stochastic Analysis/Mol.pdf:PDF},
  Owner                    = {ludovic},
  Timestamp                = {2012.02.17}
}

@Article{Mor,
  Title                    = {{A New Existence Result for Quadratic BSDEs with Jumps with Application to the Utility Maximization Problem}},
  Author                   = {Marie-Amelie Morlais},
  Journal                  = {Stochastic Process. Appl.},
  Year                     = {2010},
  Number                   = {10},
  Pages                    = {1966 - 1995},
  Volume                   = {120},

  Abstract                 = {In this study, we consider the exponential utility maximization problem in the context of a jump-diffusion model. To solve this problem, we rely on the dynamic programming principle to express the value process of this problem in terms of the solution of a quadratic BSDE with jumps. Since the quadratic BSDE1 under study is driven by both a Wiener process and a Poisson random measure having a Lévy measure with infinite mass, our main task is therefore to establish a new existence result for the specific BSDE introduced.},
  Doi                      = {DOI: 10.1016/j.spa.2010.05.011},
  File                     = {:/home/ludovic/References/Papers/Mor.pdf:PDF;:home/ludovic/US-Msc-thesis/Papers/Mor.pdf:PDF},
  ISSN                     = {0304-4149},
  Keywords                 = {Backward stochastic differential equations}
}

@InCollection{Mui-Zar2,
  Title                    = {Stochastic {P}artial {D}ifferential {E}quation and {P}ortfolio {C}hoice},
  Author                   = {M. Musiela and Thaleia Zariphopoulou},
  Booktitle                = {C. Chiarella and A. Novikov},
  Publisher                = {Springer},
  Year                     = {2010},

  Abstract                 = {We introduce a stochastic partial di¤erential equation which describes the evolution of the investment performance process in portfolio choice models. The equation is derived for two formulations of the investment problem, namely, the traditional one (based on maximal expected utility of terminal wealth) and the recently developed forward formulation. The novel element in the forward case is the volatility process which is up to the investor to choose. We provide various examples for both cases and discuss the di¤erences and similarities between the di¤erent forms of the equation as well as the associated solutions and optimal processes.},
  File                     = {:/home/ludovic/References/Papers/Mui-Zar2.pdf:PDF},
  Owner                    = {ludovic},
  Timestamp                = {2010.11.13}
}

@Article{Mus-Zar,
  Title                    = {A {V}aluation {A}lgorithm for {I}ndifference {P}rices in {I}ncomplete {M}arkets},
  Author                   = {Marek Musiela and Thaleia Zariphopoulou},
  Journal                  = {Finance and Stochastics},
  Year                     = {2004},
  Pages                    = {399-413},
  Volume                   = {8}
}

@Book{vonN-Mor,
  Title                    = {Theory of Games and Economics Behavior},
  Author                   = {von Neumann, John and Morgenstern, Oskar},
  Publisher                = {Princeton University Press},
  Year                     = {1947},
  Edition                  = {2nd},

  Owner                    = {tapir},
  Timestamp                = {2006.05.19}
}

@Book{Nua,
  Title                    = {{The Malliavin Calculus and Related Topics }},
  Author                   = {David Nualar},
  Editor                   = {J. Gani and C.C Heyde and T.E Kurtz},
  Publisher                = {Springer-Verlag},
  Year                     = {1995},

  Address                  = {New York Berlin},

  Owner                    = {ludovic},
  Timestamp                = {2010.12.06}
}

@Article{Par-Peng,
  Title                    = {{Adapted Solution of a Backward Stochastic Differential Equation}},
  Author                   = {Pardoux, {\'E}. and Peng, S. G.},
  Journal                  = {Systems Control Lett.},
  Year                     = {1990},
  Number                   = {1},
  Pages                    = {55--61},
  Volume                   = {14},

  Coden                    = {SCLEDC},
  Doi                      = {10.1016/0167-6911(90)90082-6},
  File                     = {:/home/ludovic/References/Papers/Par-Peng.pdf:PDF},
  Fjournal                 = {Systems \& Control Letters},
  ISSN                     = {0167-6911},
  Mrclass                  = {60H10 (60H20 93E03)},
  Mrnumber                 = {1037747 (91e:60171)},
  Mrreviewer               = {Kiyomasa Narita}
}

@Article{Par,
  Title                    = {{Backward Stochastic Differential Equations and Viscosity Solutions of Systems of Semilinear Parabolic and Elliptic PDEs of Second Order}},
  Author                   = {Emmanuel Pardoux},
  Journal                  = {Stochastic Analysis and Related Topics},
  Year                     = {1996},
  Pages                    = {55-61},
  Volume                   = {1},

  Owner                    = {ludovic},
  Timestamp                = {2010.11.02}
}

@Book{Pat-You,
  Title                    = {{Linear Functinal Analysis}},
  Author                   = {Rynne Bryan Patrick and Martin A. Youngson},
  Publisher                = {Springer},
  Year                     = {2000},

  File                     = {:/home/ludovic/References/Books/Math_gen/Rynne B.P., Youngson M.A. Linear Functional Analysis.djvu:Djvu},
  Owner                    = {ludovic},
  Timestamp                = {2011.06.24}
}

@Book{Peng10,
  Title                    = {{Nonlinear Expectations and Stochastic Calculus under Uncertainty -- with Robust Central Limit Theorem and G-Brownian Motion}},
  Author                   = {Shige Peng},
  Year                     = {2010},
  Edition                  = {1},

  File                     = {:/home/ludovic/References/Books/Stochastic Analysis/Peng.pdf:PDF},
  Owner                    = {ludovic},
  Timestamp                = {2012.04.16}
}

@Article{Peng99,
  Title                    = {{Monotonic Limit Theorem of BSDE and Nonlinear Decomposition Theorem of Doob - Meyer's Type}},
  Author                   = {Shige Peng},
  Journal                  = {Probab. Theory Relat. Fields},
  Year                     = {1999},
  Pages                    = {473-499},
  Volume                   = {113},

  Abstract                 = {We have obtained the following limit theorem: if a sequence of RCLL supersolutions of a backward stochastic differential equations (BSDE) converges monotonically up to ( y t ) with E [sup t | y t | 2 ] &lt; ∞, then ( y t ) itself is a RCLL supersolution of the same BSDE (Theorem 2.4 and 3.6). We apply this result to the following two problems: 1) nonlinear Doob–Meyer Decomposition Theorem. 2) the smallest supersolution of a BSDE with constraints on the solution ( y, z ). The constraints may be non convex with respect to ( y, z ) and may be only measurable with respect to the time variable t . this result may be applied to the pricing of hedging contingent claims with constrained portfolios and/or wealth processes.},
  Affiliation              = {Department of Mathematics, Shandong University, Jinan, 250100, P. R. China. e-mail: peng@public.jn.sd.cn and peng@sdu.edu.cn CN},
  File                     = {:/home/ludovic/References/Papers/Peng99.pdf:PDF},
  ISSN                     = {0178-8051},
  Issue                    = {4},
  Keyword                  = {Mathematics and Statistics},
  Publisher                = {Springer Berlin / Heidelberg}
}

@Article{PengFBSDE,
  Title                    = {{Backward Stochastic Differential Equations and Applications to Optimal Control}},
  Author                   = {Shige Peng},
  Journal                  = {Applied Mathematics and Optimization},
  Year                     = {1993},
  Pages                    = {125-144},
  Volume                   = {27},

  File                     = {:/home/ludovic/References/Papers/PengFBSDE.pdf:PDF},
  Owner                    = {ludovic},
  Timestamp                = {2013.05.07}
}

@Article{Peng-Xu,
  Title                    = {{The Smallest g-supermartingale and Reflected BSDE with Single and Double Obstacles}},
  Author                   = {Shige Peng and Mingyu Xu},
  Journal                  = {Annales de l'Institut Henri Poincare (B) Probability and Statistics},
  Year                     = {2005},
  Note                     = {<ce:title>En hommage a Paul André Meyer</ce:title>},
  Number                   = {3},
  Pages                    = {605 - 630},
  Volume                   = {41},

  Abstract                 = {In this paper we show how a solution of BSDE can be reflected by a very irregular L 2 -obstacle. We prove that this problem is equivalent to find the smallest g-supermartingale of BSDE that dominates this obstacle. We then obtain the existence and uniqueness and continuous dependence theorem for this reflected BSDE. We also consider the problem of existence and uniqueness of reflected BSDE with double L 2 obstacles, by using a penalization method. A new monotonic limit theorem is developed to prove the convergence of the penalization sequence, and to prove the existence theorem. We also prove that this reflected BSDE with double obstacles is equivalent to a problem of the smallest g-supermartingale and the largest g-submartingale.},
  Doi                      = {10.1016/j.anihpb.2004.12.002},
  File                     = {:/home/ludovic/References/Papers/Peng-Xu.pdf:PDF},
  ISSN                     = {0246-0203},
  Keywords                 = {Reflected backward stochastic differential equation},
  Url                      = {http://www.sciencedirect.com/science/article/pii/S0246020305000245}
}

@Unpublished{Pen-Rev,
  Title                    = {{Convex Risk Measures for Processes and Related BSDEs}},
  Author                   = {Irina Penner and Anthony R\'eveillac},
  Note                     = {In preparation},

  Month                    = {Novermber},
  Year                     = {2012},

  Owner                    = {ludovic},
  Timestamp                = {2012.11.23}
}

@Book{Pha,
  Title                    = {{Continuous-time Stochastic Control and Optimization with Financial Applications}},
  Author                   = {Huy\^en Pham},
  Editor                   = {B. Rozovski\^i and G. Grimmett},
  Publisher                = {Springer},
  Year                     = {2009},
  Edition                  = {1},

  File                     = {:/home/ludovic/References/Books/Stochastic finance/Continuous time Stochastic Control_ and Introduction with Financial Applications Pha.pdf:PDF},
  Owner                    = {ludovic},
  Timestamp                = {2011.02.22}
}

@Article{Pham,
  Title                    = {Stochastic {C}ontrol {under} {P}rogressive {E}nlargement of {F}iltrations and {A}pplications to {M}ultiple {D}efaults {R}isk {M}anagement},
  Author                   = {Huy\^en Pham},
  Journal                  = {Stochastic Process. Appl.},
  Year                     = {2010},

  Month                    = {May},
  Pages                    = {1795-1820},
  Volume                   = {120},

  Abstract                 = {We formulate and investigate a general stochastic control problem under a progressive enlargement of filtration. The global information is enlarged from a reference filtration and the knowledge of multiple random times together with associated marks when they occur. By working under a density hypothesis on the conditional joint distribution of the random times and marks, we prove a decomposition of the original stochastic control problem under the global filtration into classical stochastic control problems under the reference filtration, which is determined in a finite backward induction. Our method revisits and extends in particular stochastic control of diffusion processes with a finite number of jumps. This study is motivated by optimization problems arising in default risk management, and we provide applications of our decomposition result for the indifference pricing of defaultable claims, and the optimal investment under bilateral counterparty risk. The solutions are expressed in terms of BSDEs involving only Brownian filtration, and remarkably without jump terms coming from the default times and marks in the global filtration.},
  File                     = {:/home/ludovic/References/Papers/Pham.pdf:PDF;:home/ludovic/US-Msc-thesis/Papers/Pham.pdf:PDF}
}

@Book{Pol,
  Title                    = {{Convergence of Stochastic Processes}},
  Author                   = {David Pollard},
  Publisher                = {Springer-Verlag},
  Year                     = {1984},

  File                     = {:/home/ludovic/References/Books/Stochastic Analysis/pollard1984.pdf:PDF},
  Owner                    = {ludovic},
  Timestamp                = {2011.12.11}
}

@Article{Pra,
  Title                    = {{A Minimax Theorem Without Compactness Hypothesis}},
  Author                   = {Maurizio Pratelli},
  Journal                  = {Mediterr. J. Math.},
  Year                     = {2005},
  Pages                    = {103-112},
  Volume                   = {2},

  Owner                    = {ludovic},
  Timestamp                = {2012.07.12}
}

@Book{Pri,
  Title                    = {{Portfolio Optimization and Performance Analysis}},
  Author                   = {Jean-Luc Prigent},
  Publisher                = {Chapman and Hall/CRC},
  Year                     = {2007},

  File                     = {:/home/ludovic/References/Books/Finance/Portfolio Optimization and Performance Analysis.pdf:PDF},
  Owner                    = {ludovic},
  Timestamp                = {2011.06.12}
}

@Book{Pro,
  Title                    = {{Stochastic Integration and Differential Equations}},
  Author                   = {Philip E. Protter},
  Editor                   = {B. Rozovski\^i and Marc Yor},
  Publisher                = {Springer-Verlag},
  Year                     = {2004},
  Edition                  = {2},

  File                     = {:/home/ludovic/References/Books/Stochastic Analysis/Stochastic Integration and Differential Equations 2ed - Protter.pdf:PDF},
  Owner                    = {ludovic},
  Pages                    = {430p},
  Timestamp                = {2011.08.29}
}

@PhdThesis{Que,
  Title                    = {M\'ethodes de {C}ontr\^ole {S}tochastique en {F}inance},
  Author                   = {M. C. Quenez},
  School                   = {Universit\'e Pierre et Marie Curie},
  Year                     = {1993},
  Type                     = {Ph.D}
}

@Unpublished{Rev,
  Title                    = {{On the Orthogonal Component of BSDEs in a Markovian Setting}},
  Author                   = {Anthony R\'eveillac},
  Year                     = {2011},

  File                     = {:/home/ludovic/References/Papers/Rev.ps:PostScript},
  Owner                    = {ludovic},
  Timestamp                = {2011.07.31}
}

@Article{karandikar01,
  Title                    = {{On Pathwise Stochastic Integration}},
  Author                   = {L Rajeeva and Karandikar},
  Journal                  = {Stochastic Processes and their Applications},
  Year                     = {1995},
  Pages                    = {11-18},

  Owner                    = {ludovic},
  Timestamp                = {2014.03.18}
}

@MastersThesis{Dimby,
  Title                    = {Growth {Optimal Portfolios and Real World Pricing}},
  Author                   = {Dimbinirina Ramarimbahoaka},
  School                   = {Stellenbosch University},
  Year                     = {2008},
  Month                    = {December},

  Abstract                 = {In the Benchmark Approach to Finance, it has been shown that by taking the Growth Optimal Portfolio as numéraire, a candidate for a pricing derivatives formula under the real world probability can be given. This result allows us to price in an incomplete financial market model. The result comes from two different approaches. In the first approach we use the supermartingale property of portfolios in units of the benchmark portfolio which leads to the fact that an equivalent measure is not needed. In the second approach the numéraire property of the Growth Optimal Portfolio is used. The numéraire portfolio defines an equivalent martingale measure and by change of measure using the Radon-Nikodým derivative, a real world pricing formula is derived which is the same as the one given by the first approach stated above.},
  File                     = {:/home/ludovic/References/Thesis/Dimby.pdf:PDF;:/home/ludovic/US-Msc-thesis/Papers/Dimby.pdf:PDF},
  Owner                    = {ludovic},
  Timestamp                = {2010.11.27},
  Url                      = {http://scholar.sun.ac.za/bitstream/handle/10019.1/2209/Ramarimbahoaka
}

@Book{Rei-Bro,
  Title                    = {{Investment Analysis and Portfolio Management}},
  Author                   = {Frank Reilly and Keith Brown},
  Edition                  = {7},

  File                     = {:/home/ludovic/References/Books/Finance/Investment Analysis and Portfolio Management.pdf:PDF},
  Owner                    = {ludovic},
  Timestamp                = {2011.06.12}
}

@Article{Reis-Rev-Zha,
  Title                    = {{FBSDE with Tim Delayed Generators: Lp-Solutions, Defferentiability, Representation formulas and Path Regularity}},
  Author                   = {Goncalo dos Reis and Anthony R\'eveillac and Jianing Zhang},
  Journal                  = {Stoch. Anal. Appl.},
  Year                     = {2011},

  File                     = {:/home/ludovic/References/Papers/Reis-Rev-Zha.pdf:PDF},
  Owner                    = {ludovic},
  Timestamp                = {2011.07.31}
}

@PhdThesis{Reis,
  Title                    = {On {S}ome {P}roperties of {S}olutions of {Q}uadratic {G}rowth {B}{S}{D}{E} and {A}pplications in {F}inance and {I}nsurance},
  Author                   = {Goncalo Dos Reis},
  School                   = {Humboldt-Universit{\"a}t zu Berlin},
  Year                     = {2010},
  Month                    = {May},
  Type                     = {Ph.D},

  Abstract                 = {We consider Backward Stochastic Differential Equations (BSDE) with genera- tors that grow quadratically in the control variable. In a more abstract setting, we first allow both the terminal condition and the generator to depend on a vec- tor parameter x. We give sufficient conditions for the solution pair of the BSDE to be differentiable in x. These results can be applied to systems of forward- backward SDE. If the terminal condition of the BSDE is given by a sufficiently smooth function of the terminal value of a forward SDE, then its solution pair is differentiable with respect to the initial vector of the forward equation. Finally we prove sufficient conditions for solutions of quadratic BSDE to be once and twice differentiable in the variational sense (Malliavin differentiable). We apply the obtained results to extract several properties of these stochastic equations, namely the Markov property, a comonotonicity result, path regularity and an explicit convergence rate for truncated quadratic BSDE. We study a problem of insurance related derivatives on financial markets that are based on non-tradable underlyings, but are correlated with tradable assets. We calculate exponential utility-based indifference prices, and corresponding derivative hedges. We use the fact that they can be represented in terms of solutions of forward-backward stochastic differential equations (FBSDE) with quadratic growth generators. In this case the optimal hedge can be represented by the price gradient multiplied with the correlation coefficient. This way we obtain a generalization of the classical ?delta hedge? in complete markets.},
  File                     = {:/home/ludovic/References/Thesis/Reis.pdf:PDF}
}

@Book{Rev-Yor,
  Title                    = {Continuous {M}artingales and {B}rownian {M}otion},
  Author                   = {D Revuz and M Yor},
  Publisher                = {Springer},
  Year                     = {1999},

  File                     = {:/home/ludovic/References/Books/Stochastic Analysis/Continuous Martingales and Brownian motion.djvu:Djvu}
}

@PhdThesis{Ric,
  Title                    = {{\'Etude Th\'eorique et Num\'erique des \'Equations Diff\'erentielles Stochastiques R\'etrogrades}},
  Author                   = {Adrien Richou},
  School                   = {Universit\'{e} de Rennes 1},
  Year                     = {2010},
  Type                     = {Ph.D},

  Abstract                 = {This thesis is made of three independent parts. Firstly, we study a new class of ergodic backward stochastic differential equations - EBSDEs for short - which is linked with semi-linear Neumann type boundary value problems related to ergodic phenomena. The particularity of these problems is that the ergodic constant appears in Neumann boundary conditions. We study the existence anduniqueness of solutions to EBSDEs and the link with partial differential equations. We also apply these results to optimal ergodic control problems. In a second part, we generalise a work of P. Briand and Y. Hu published in 2008. these authors have proved the uniqueness among the solutions of quadratic BSDEs with convex generators and unbounded terminal conditions which admit every exponential moments. We prove that uniqueness holds among solutions whichadmit some given exponential moments. These exponential moments are natural as they are given by the existence theorem. Thanks to this uniqueness result we can strengthen the nonlinear Feynman-Kac formula proved by P. Briand and Y. Hu. Finally, we deal with the numerical resolution of Markovian quadratic BSDEs with bounded terminal conditions. We first show some bound estimates on the process Z and we specify the Zhang's path regularity theorem. Then we give a new time discretization scheme with a non uniform time net for such BSDEs and we obtain an explicit convergence rate for this scheme. We also compute some numerical simulations to study the efficiency of our scheme in a practical situation.},
  File                     = {:/home/ludovic/References/Thesis/Ric.pdf:PDF}
}

@Unpublished{Ric2,
  Title                    = {Numerical {S}imulation of {B}{S}{D}{E}s with {Q}uadratic {G}rowth},
  Author                   = {Adrien Richou},
  Note                     = {To appear in the Annals of Applied Probability},
  Year                     = {2010},

  Abstract                 = {This article deals with the numerical resolution of Markovian backward stochastic differential equations (BSDEs) with drivers of quadratic growth with respect to z and bounded terminal conditions. We first show some bound estimates on the process Z and we specify the Zhang?s path regularity theorem. Then we give a new time discretization scheme with a non uniform time net for such BSDEs and we obtain an explicit convergence rate for this scheme.},
  Annote                   = {Preprint, Available at: arXiv:1001.0401v3},
  File                     = {:/home/ludovic/References/Papers/Ric2.pdf:PDF;:home/ludovic/US-Msc-thesis/Papers/Ric2.pdf:PDF}
}

@Unpublished{Rie2011,
  Title                    = {{The Fundamental Theorem of Asset Pricing without Probabilistic Prior Assumptions}},
  Author                   = {Frank Riedel},
  Note                     = {Preprint},
  Year                     = {2011},

  Owner                    = {ludovic},
  Quality                  = {1},
  Timestamp                = {2013.12.02}
}

@InProceedings{Roc76,
  Title                    = {{Integral Functionals, Normal Integrands and Measurable Selections}},
  Author                   = {R. T. Rockafellar},
  Booktitle                = {Nonlinear Operators and the Calculus of Variations},
  Year                     = {1976},
  Editor                   = {J. Gossez and E. Lami Dozo and J. Mawhin and L. Waelbroeck},
  Note                     = {Lecture Notes in Mathematics},
  Pages                    = {157-207},
  Publisher                = {Springer Berlin / Heidelberg},
  Volume                   = {543},

  File                     = {:Papers/Roc76.pdf:PDF},
  Owner                    = {ludovic},
  Timestamp                = {2012.06.08}
}

@Book{roc74,
  Title                    = {{Conjugate Duality and Optimization}},
  Author                   = {Rockafellar, R. T.},
  Editor                   = {Society for Industrial and Applied Mathematics},
  Year                     = {1974},

  File                     = {:/home/ludovic/References/Books/convex analysis/rtr-ConjugateDuality.pdf:PDF},
  Owner                    = {ludovic},
  Timestamp                = {2013.06.11}
}

@Book{roc70,
  Title                    = {Convex Analysis},
  Author                   = {Rockafellar, R. T.},
  Publisher                = {Princeton University Press},
  Year                     = {1970},

  File                     = {:/home/ludovic/References/Books/convex analysis/rtr-ConvexAnalysis.djvu:Djvu},
  Owner                    = {ludovic},
  Timestamp                = {2013.02.10}
}

@Article{Roc66,
  Title                    = {{Extension of Fenchel's Duality Theorem for Convex Functions}},
  Author                   = {R. T. Rockafellar},
  Journal                  = {Duke Math. J.},
  Year                     = {1966},
  Pages                    = {81-89},
  Volume                   = {33},

  File                     = {:/home/ludovic/References/Papers/Roc66.pdf:PDF},
  Owner                    = {ludovic},
  Timestamp                = {2012.07.11}
}

@Book{Roc-Wets,
  Title                    = {{Variational Analysis}},
  Author                   = {R. Tyrrell Rockafellar and Roger Wets},
  Publisher                = {Springer},
  Year                     = {1998},

  Address                  = {Berlin, New York},

  File                     = {:Books/convex analysis/rtr-VarAnalysis-RockWets.pdf:PDF},
  Owner                    = {ludovic},
  Timestamp                = {2012.03.16}
}

@Book{Rog-Wil,
  Title                    = {Diffusion, {M}arkov {P}rocesses, and {M}artingales},
  Author                   = {L. C. G Rogers and David Williams},
  Publisher                = {John Wiley and Sons},
  Year                     = {1987},
  Edition                  = {2}
}

@InCollection{Rong,
  Title                    = {B{S}{D}{E}s with {J}umps and with {Q}uadratic {G}rowth {C}oefficients and {O}ptimal {C}onsumption},
  Author                   = {Situ Rong},
  Booktitle                = {Harmonic, Wavelet and $p$-Adic Analysis},
  Publisher                = {World Scientific Publishing Co},
  Year                     = {2007},
  Editor                   = {N. M. Chuong},
  Pages                    = {343--361}
}

@Article{Ros,
  Title                    = {{Risk Measures via $g$-Expectations$}},
  Author                   = {Emanuela Rosazza Gianin},
  Journal                  = {Insurance Mathematics and Economics},
  Year                     = {2006},
  Pages                    = {19-34},
  Volume                   = {39},

  Abstract                 = {This paper shows how g-expectations and conditional g-expectations provide some families of static and dynamic risk measures. Conversely, some sufficient conditions for a dynamic risk measure to be induced by a conditional g-expectation are provided. A financial interpretation of the functional g will be given.},
  File                     = {:/home/ludovic/References/Papers/Ros.pdf:PDF},
  Owner                    = {ludovic},
  Timestamp                = {2012.10.10}
}

@InCollection{Run,
  Title                    = {Jump-Diffusion Models},
  Author                   = {Williams Runggaldier},
  Booktitle                = {Handbooks in Finance, Book 1},
  Publisher                = {Elesevier/North-Holland},
  Year                     = {2003},
  Editor                   = {W. Ziemba},
  Pages                    = {169-209},

  File                     = {:/home/ludovic/References/Papers/Run.ps:PostScript},
  Owner                    = {ludovic},
  Timestamp                = {2011.04.26}
}

@Book{Savage01,
  Title                    = {The foundations of Statistics},
  Author                   = {Savage, Leonard J. },
  Publisher                = {Dover Publications},
  Year                     = {1972},
  Edition                  = {2 Revised},

  Citeulike-article-id     = {2526956},
  Keywords                 = {bibtex-import},
  Owner                    = {ludovic},
  Priority                 = {0},
  Timestamp                = {2015.02.26}
}

@Unpublished{Sch,
  Title                    = {{Portfolio Optimization in Incomplete Financial Markets}},
  Author                   = {Walter Schachermayer},
  Note                     = {Notes of the Scuola Normal Superiore Cattedra Galileiana, Pisa, v+65 p. ISBN 88764216},
  Year                     = {2004},

  File                     = {:/home/ludovic/References/Papers/Sch.pdf:PDF},
  Owner                    = {ludovic},
  Timestamp                = {2011.03.21}
}

@Article{Sch01,
  Title                    = {{Optimal Investment in Incomplete Market when the Wealth may Become Negative}},
  Author                   = {Walter Schachermayer},
  Journal                  = {Ann. Appl. Probab.},
  Year                     = {2001},
  Pages                    = {694-734},
  Volume                   = {11},

  File                     = {:/home/ludovic/References/Papers/Sch01.pdf:PDF},
  Owner                    = {ludovic},
  Timestamp                = {2012.06.23}
}

@Article{SchroderandSkiadas,
  Title                    = {Optimal lifetime consumption-portfolio strategies under trading constraints and generalized recursive preferences},
  Author                   = {Mark Schroder and Costis Skiadas},
  Journal                  = {Stochastic Process. Appl.},
  Year                     = {2003},
  Number                   = {2},
  Pages                    = {155 - 202},
  Volume                   = {108},

  Abstract                 = {We consider the lifetime consumption-portfolio problem in a competitive securities market with essentially arbitrary continuous price dynamics, and convex trading constraints (e.g., incomplete markets and short-sale constraints). Abstract first-order conditions of optimality are derived, based on a preference-independent notion of constrained state pricing. For homothetic generalized recursive utility, we derive closed-form solutions for the optimal consumption and trading strategy in terms of the solution to a single constrained BSDE. Incomplete market solutions are related to complete markets solutions with modified risk aversion towards non-marketed risk. Methodologically, we develop the utility gradient approach, but for the homothetic case we also verify the solution using the dynamic programming approach, without having to assume a Markovian structure. Finally, we present a class of parametric examples in which the BSDE characterizing the solution reduces to a system of Riccati equations.},
  Doi                      = {DOI: 10.1016/j.spa.2003.09.001},
  ISSN                     = {0304-4149},
  Url                      = {http://www.sciencedirect.com/science/article/B6V1B-49M09NY-2/2/9a5d556d25b06054176a6f548a687236}
}

@Article{Sch-Ski,
  Title                    = {Optimal Consumption and Portfolio Selection with Stochastic Differential Utility},
  Author                   = {Mark Schroder and Costis Skiadas},
  Journal                  = {Journal of Economic Theory},
  Year                     = {1999},
  Number                   = {1},
  Pages                    = {68 - 126},
  Volume                   = {89},

  Abstract                 = {We develop the utility gradient (or martingale) approach for computing portfolio and consumption plans that maximize stochastic differential utility (SDU), a continuous-time version of recursive utility due to D. Duffie and L. Epstein (1992, Econometrica60, 353-394). We characterize the first-order conditions of optimality as a system of forward-backward SDEs, which, in the Markovian case, reduces to a system of PDEs and forward only SDEs that is amenable to numerical computation. Another contribution is a proof of existence, uniqueness, and basic properties for a parametric class of homothetic SDUs that can be thought of as a continuous-time version of the CES Kreps-Porteus utilities studied by L. Epstein and A. Zin (1989, Econometrica57, 937-969). For this class, we derive closed-form solutions in terms of a single backward SDE (without imposing a Markovian structure). We conclude with several tractable concrete examples involving the type of #affine# state price dynamics that are familiar from the term structure literature. Journal of Economic Literature Classification Numbers: G11, E21, D91, D81, C61.},
  Doi                      = {DOI: 10.1006/jeth.1999.2558},
  File                     = {:/home/ludovic/References/Papers/Sch-Ski.pdf:PDF},
  ISSN                     = {0022-0531},
  Owner                    = {ludovic},
  Timestamp                = {2010.11.25}
}

@Book{Shr2,
  Title                    = {{Stochastic Calculus for Finance}},
  Author                   = {Steven E. Shreve},
  Editor                   = {Prasad Chalasani and Somesh Jha},
  Publisher                = {Springer},
  Year                     = {1997},

  File                     = {:/home/ludovic/References/Books/Stochastic finance/Stochastic Calculus and Finance - Steven Shreve.pdf:PDF;:/home/ludovic/References/Books/Finance stochastics/Shreve S.E. Stochastic calculus for finance II.djvu:Djvu},
  Owner                    = {ludovic},
  Timestamp                = {2011.06.24}
}

@Book{Shr1,
  Title                    = {{Stochastic Calculus for Finance 1 the Binomial Asset Pricing Model}},
  Author                   = {Steven E. Shreve},
  Publisher                = {Springer},
  Year                     = {2004},

  File                     = {:/home/ludovic/References/Books/Stochastic finance/Shreve S.E. Stochastic calculus for finance I.djvu:Djvu;:/home/ludovic/References/Books/Finance stochastics/Shreve S.E. Stochastic calculus for finance I.djvu:Djvu},
  Owner                    = {ludovic},
  Pages                    = {203},
  Timestamp                = {2011.06.24}
}

@Article{Ski,
  Title                    = {Robust control and recursive utility},
  Author                   = {Skiadas, Costis},
  Journal                  = {Finance and Stochastics},
  Year                     = {2003},
  Note                     = {10.1007/s007800300100},
  Pages                    = {475-489},
  Volume                   = {7},

  Abstract                 = {This paper shows that a finite-horizon version of the robust control criterion appearing in recent papers by Hansen, Sargent, and their coauthors can be described as recursive utility, which in continuous time takes the form of the Stochastic Differential Utility (SDU) of Duffie and Epstein (1992). While it has previously been noted that Bellman equations arising in robust control settings are of the same form as Bellman equations arising from SDU maximization, here this connection is shown directly without reference to any underlying dynamics, or Markov structure.},
  Affiliation              = {Department of Finance, Kellogg School of Management, Northwestern University, Evanston, IL 60208, USA (e-mail: c-skiadas@kellogg.northwestern.edu) US},
  File                     = {:/home/ludovic/References/Papers/Ski.pdf:PDF},
  ISSN                     = {0949-2984},
  Issue                    = {4},
  Keyword                  = {Business and Economics},
  Publisher                = {Springer Berlin / Heidelberg}
}

@Book{Son,
  Title                    = {{Introduction to Stochastic Calculus for Finance A new didactic approach}},
  Author                   = {Dierter Sondermann},
  Publisher                = {Springer},
  Year                     = {2006},
  Note                     = {Lecture notes in Economics and Mathematical Systems},

  File                     = {:/home/ludovic/References/Books/Finance stochastics/Introduction to Stochastic Calculus for Finance.pdf:PDF},
  Owner                    = {ludovic},
  Pages                    = {143},
  Timestamp                = {2011.06.12}
}

@Unpublished{Sung-Wan,
  Title                    = {Equilibrium {E}quity {P}remium, {I}nterest {R}ate and the {C}ost of {C}apital in a {M}oral-{H}azard {E}conomy},
  Author                   = {Jaeyoung Sung and Xuhu Wan},
  Note                     = {Available at SSRN: http://ssrn.com/abstract=1570986},
  Year                     = {2010}
}

@MastersThesis{Tan,
  Title                    = {{Optimal Cross Hedging of Insurance Derivatives Using Quadratic BSDEs}},
  Author                   = {Ludovic Tangpi},
  School                   = {Stellenbosch University},
  Year                     = {2011},

  File                     = {:/home/ludovic/References/Thesis/Tan.pdf:PDF},
  Owner                    = {ludovic},
  Timestamp                = {2011.04.19}
}

@Unpublished{ludo,
  Title                    = {Stochastic {C}ontrol: with {A}pplications to {F}inancial {M}athematics},
  Author                   = {Ludovic Tangpi},
  Note                     = {Postgraduate Diploma Essay, African Institute for Mathematical Sciences},
  Year                     = {2010},

  File                     = {:/home/ludovic/References/Thesis/ludo.pdf:PDF},
  Owner                    = {ludovic},
  Timestamp                = {2011.02.01}
}

@MastersThesis{Nicole,
  Title                    = {Portfolio Optimization Problems: A Martingale and a Convex Duality Approach},
  Author                   = {Nicole Flaure Kouemo Tchamga},
  School                   = {University of Stellenbosch},
  Year                     = {2010},

  File                     = {:/home/ludovic/References/Thesis/Nicole.pdf:PDF;:/home/ludovic/US-Msc-thesis/Papers/Nicole.pdf:PDF},
  Owner                    = {ludovic},
  Timestamp                = {2010.12.21}
}

@Article{Tev,
  Title                    = {Solvability of backward stochastic differential equations with quadratic growth},
  Author                   = {Revaz Tevzadze},
  Journal                  = {Stochastic Process. Appl.},
  Year                     = {2008},
  Number                   = {3},
  Pages                    = {503 - 515},
  Volume                   = {118},

  Abstract                 = {We prove the existence of the unique solution of a general backward stochastic differential equation with quadratic growth driven by martingales. A kind of comparison theorem is also proved.},
  Doi                      = {DOI: 10.1016/j.spa.2007.05.009},
  File                     = {:/home/ludovic/US-Msc-thesis/Papers/Tev.pdf:PDF},
  ISSN                     = {0304-4149},
  Keywords                 = {Backward stochastic differential equation},
  Url                      = {http://www.sciencedirect.com/science/article/B6V1B-4NX2NNB-1/2/abf39cf5f6a729e859cdca0b4b9ecb42}
}

@Book{Tij,
  Title                    = {{A First Course in Stochastic Models}},
  Author                   = {Henk C. Tijms},
  Publisher                = {John Wiley and Sons Ltd},
  Year                     = {2003},

  File                     = {:/home/ludovic/References/Books/Stochastic Analysis/A first Course in Stochastic Models.pdf:PDF;:/home/ludovic/References/Books/Tij.pdf:PDF},
  Owner                    = {ludovic},
  Timestamp                = {2011.05.30}
}

@Book{Var,
  Title                    = {{Lectures Notes on Optimization}},
  Author                   = {Pravin Varaiya},
  Editor                   = {George L. Turin},
  Year                     = {1971},
  Note                     = {Van Nostrand Reinhold Notes on Sys- tem Sciences,},

  File                     = {:/home/ludovic/References/Books/Optimization/NOO LectureNotes on Optimization.pdf:PDF},
  Owner                    = {ludovic},
  Timestamp                = {2011.06.24}
}

@Book{Vie-Mar-Flo,
  Title                    = {{Handbook of Modeling High-Frequency Data in Finance}},
  Author                   = {Frederi Viens and Maria Mariani and Ionut Florescu},
  Editor                   = {Frederi Viens and Maria Mariani and Ionut Florescu},
  Publisher                = {John Wiley and Sons},
  Year                     = {2012},

  Owner                    = {ludovic},
  Timestamp                = {2012.04.28}
}

@Article{Watanabe,
  Title                    = {{The Japanese Contributions to Martingales}},
  Author                   = {Shinzo Watanabe},
  Journal                  = {Electronic Journal for History of Probability and Statistics},
  Year                     = {2009},
  Pages                    = {1-13},
  Volume                   = {5},

  File                     = {:/home/ludovic/References/Papers/Watanabe.pdf:PDF},
  Owner                    = {ludovic},
  Timestamp                = {2011.04.11}
}

@Unpublished{Wes-Zhe,
  Title                    = {{Constrained Nonsmooth Utility Maximization on the Positive Real Line}},
  Author                   = {Nicholas Westray and Harry Zheng},
  Note                     = {Preprint},
  Year                     = {2010},

  Owner                    = {ludovic},
  Timestamp                = {2012.07.19}
}

@Book{Wil,
  Title                    = {{Probability with Martingales}},
  Author                   = {David Williams},
  Publisher                = {Cambridge University Press},
  Year                     = {1991},

  File                     = {:/home/ludovic/References/Books/Probability &Stochastic/Williams D.W. Probability with martingales.djvu:Djvu;:/home/ludovic/References/Books/Probability &Stochastic /Williams D.W. Probability with martingales.djvu:Djvu},
  Owner                    = {ludovic},
  Timestamp                = {2011.06.11}
}

@PhdThesis{Xu,
  Title                    = {{Nontradable Market Index and its Derivates}},
  Author                   = {Peng Xu},
  School                   = {University of Toronto},
  Year                     = {2009},

  File                     = {:/home/ludovic/References/Thesis/Peng.pdf:PDF},
  Owner                    = {ludovic},
  Timestamp                = {2011.08.27}
}

@Article{Yang-Zha1,
  Title                    = {Optimal investment for insurer with jump-diffusion risk process},
  Author                   = {Hailiang Yang and Lihong Zhang},
  Journal                  = {Insurance: Mathematics and Economics},
  Year                     = {2005},
  Number                   = {3},
  Pages                    = {615 - 634},
  Volume                   = {37},

  Abstract                 = {In this paper, we study optimal investment policies of an insurer with jump-diffusion risk process. Under the assumptions that the risk process is compound Poisson process perturbed by a standard Brownian motion and the insurer can invest in the money market and in a risky asset, we obtain the close form expression of the optimal policy when the utility function is exponential. We also study the insurer's optimal policy for general objective function, a verification theorem is proved by using martingale optimality principle and Ito's formula for jump-diffusion process. In the case of minimizing ruin probability, numerical methods and numerical results are presented for various claim-size distributions.},
  Doi                      = {DOI: 10.1016/j.insmatheco.2005.06.009},
  File                     = {:/home/ludovic/References/Papers/Yang-Zha1.pdf:PDF;:/home/ludovic/US-Msc-thesis/Papers/Yang-Zha1.pdf:PDF},
  ISSN                     = {0167-6687},
  Keywords                 = {Hamilton-Jacobi-Bellman equations},
  Owner                    = {ludovic},
  Timestamp                = {2011.02.01},
  Url                      = {http://www.sciencedirect.com/science/article/B6V8N-4GV8SY8-1/2/3b0003e4d7231ec8d0065b65294af9e6}
}

@Book{Jio,
  Title                    = {Stochastic {C}ontrol: {H}amiltonian {S}ystems and {H}{J}{B} {E}quations},
  Author                   = {Jiongmin Yong and Xun Yu Zhou},
  Publisher                = {Springer New York},
  Year                     = {1999}
}

@Article{Zar,
  Title                    = {A {S}olution {A}pproach to {V}aluation with {U}nhedgable {R}isk},
  Author                   = {Thaleia Zariphopoulou},
  Journal                  = {Finance Stochastic},
  Year                     = {2001},
  Pages                    = {61-82},
  Volume                   = {5}
}

@MastersThesis{Zha,
  Title                    = {{Measure Solutions of BSDEs and Feynman-Kac Formula}},
  Author                   = {Jianing Zhang},
  School                   = {Humboldt-Universit\"{a}t zu Berlin},
  Year                     = {2008},

  File                     = {:/home/ludovic/References/Thesis/Zha.pdf:PDF;:/home/ludovic/References/Papers/Zha.pdf:PDF},
  Owner                    = {ludovic},
  Timestamp                = {2011.07.30}
}



\begin{thebibliography}{32}
\providecommand{\natexlab}[1]{#1}
\providecommand{\url}[1]{\texttt{#1}}
\expandafter\ifx\csname urlstyle\endcsname\relax
  \providecommand{\doi}[1]{doi: #1}\else
  \providecommand{\doi}{doi: \begingroup \urlstyle{rm}\Url}\fi

\bibitem[Backhoff and Fontbona(2014)]{bac-Fon}
J.~Backhoff and J.~Fontbona.
\newblock {Robust Utility Maximization without Model Compactness}.
\newblock Preprint, 2014.

\bibitem[Bernoulli(1954)]{Bernouilli}
D.~Bernoulli.
\newblock Specimen theoriae novae de mensura sortis, {{Commentarii Academiae
  Scientiarum Imperialis Petropolitanae}} ({5}, 175-192, 1738).
\newblock \emph{Econometrica}, 22:\penalty0 23--36, 1954.
\newblock Translated by L. Sommer.

\bibitem[Bordigoni et~al.(2007)Bordigoni, Matoussi, and Schweizer]{Bor-Mat-Sch}
G.~Bordigoni, A.~Matoussi, and M.~Schweizer.
\newblock {A Stochastic Control Approach to a Robust Utility Maximization
  Problem}.
\newblock In F.~Benth, G.~Nunno, T.~Lindstr{\o}m, B.~{\O}ksendal, and T.~Zhang,
  editors, \emph{Stochastic Analysis and Applications}, volume~2 of \emph{Abel
  Symposia}, pages 125--151. Springer Berlin Heidelberg, 2007.
\newblock ISBN 978-3-540-70846-9.

\bibitem[Cvitani\'c et~al.(2001)Cvitani\'c, Schachermayer, and
  Wang]{Cvi-Sch-Wang}
J.~Cvitani\'c, W.~Schachermayer, and H.~Wang.
\newblock {Utility Maximization in Incomplete Market with Random Endowment}.
\newblock \emph{Finance Stoch.}, 5:\penalty0 259--272, 2001.

\bibitem[Darling and Pardoux(1997)]{Dar-Par}
R.~W.~R. Darling and E.~Pardoux.
\newblock {Backwards SDE with Random Terminal Time and Applications to
  Semilinear Elliptic PDE}.
\newblock \emph{Annals of Probability}, 25\penalty0 (3):\penalty0 1135--1159,
  1997.

\bibitem[Delbaen and Schachermayer(1994)]{Del-Sch94}
F.~Delbaen and W.~Schachermayer.
\newblock {A General Version of the Fundamental Theorem of Asset Pricing}.
\newblock \emph{Mathematische Annalen}, 300:\penalty0 463--520, 1994.

\bibitem[Delbaen and Schachermayer(1996)]{Del-Sch96}
F.~Delbaen and W.~Schachermayer.
\newblock {A Compactness Principle for Bounded Sequences of Martingales with
  Applications}.
\newblock In \emph{Proceedings of the Seminar of Stochastic Analysis, Random
  Fields and Application, Progress in Probability}, 1996.

\bibitem[Delbaen et~al.(2002)Delbaen, Grandits, Rheinlander, Samperi, and
  Schweizer]{Delbaen2}
F.~Delbaen, P.~Grandits, R.~Rheinlander, D.~Samperi, and M.~Schweizer.
\newblock Exponential {H}edging and {E}ntropic {P}enalities.
\newblock \emph{Mathematical Finance}, 12:\penalty0 99--123, 2002.

\bibitem[Delbaen et~al.(2011)Delbaen, Hu, and Bao]{Delbaen11}
F.~Delbaen, Y.~Hu, and X.~Bao.
\newblock {Backward SDEs with Superquadratic Growth}.
\newblock \emph{Probability Theory and Related Fields}, 150:\penalty0 145--192,
  2011.
\newblock ISSN 0178-8051.

\bibitem[Dol\'{e}ans-Dade and Meyer(1979)]{Dol-Mey}
C.~Dol\'{e}ans-Dade and P.-A. Meyer.
\newblock {In\'{e}galites de Normes avec Poids}.
\newblock \emph{S\'{e}minaire de Probabilit\'{e}s (Strasbourg)}, 13:\penalty0
  313--331, 1979.

\bibitem[Drapeau and Kupper(2013)]{Dra-Kup}
S.~Drapeau and M.~Kupper.
\newblock {Risk Preferences and their Robust Representation}.
\newblock \emph{Mathematics of Operations Research}, 38\penalty0 (1):\penalty0
  28--62, 2013.

\bibitem[Drapeau et~al.(2013)Drapeau, Heyne, and Kupper]{DHK1101}
S.~Drapeau, G.~Heyne, and M.~Kupper.
\newblock {Minimal Supersolutions of Convex BSDEs}.
\newblock \emph{Annals of Probability}, 41\penalty0 (6):\penalty0 3697--4427,
  2013.

\bibitem[Drapeau et~al.(2014)Drapeau, Kupper, Gianin, and Tangpi]{tarpodual}
S.~Drapeau, M.~Kupper, E.~R. Gianin, and L.~Tangpi.
\newblock {Dual Representation of Minimal Supersolutions of Convex BSDEs}.
\newblock Forthcoming in Annales de l'Institut Henry Poincar\'e (B), 2014.

\bibitem[Duffie and Epstein(1992)]{Duf-Eps}
D.~Duffie and L.~Epstein.
\newblock {Stochastic Differential Utility}.
\newblock \emph{Econometrica}, 60:\penalty0 353 -- 394, 1992.

\bibitem[Ekeland and T\'{e}mam(1976)]{Eke-Tem}
I.~Ekeland and R.~T\'{e}mam.
\newblock \emph{{Convex Analysis and Variational Problems}}.
\newblock North-Holland Publishing Company, 1976.

\bibitem[El~Karoui and Ravanelli(2009)]{Kar-Rav}
N.~El~Karoui and C.~Ravanelli.
\newblock Cash subadditive risk measures and interest rate ambiguity.
\newblock \emph{Mathematical Finance}, 19\penalty0 (4):\penalty0 561--590,
  2009.
\newblock ISSN 1467-9965.

\bibitem[{El Karoui} et~al.(2001){El Karoui}, Peng, and Quenez]{Kar-Peng-Que01}
N.~{El Karoui}, S.~Peng, and M.-C. Quenez.
\newblock {A Dynamic Maximum Principle for the Optimization of Recursive
  Utilities under Constraints}.
\newblock \emph{Annals of Applied Probability}, 11:\penalty0 663 -- 693, 2001.

\bibitem[Faidi et~al.(2013)Faidi, Matoussi, and Mni]{Fai-Mat-Mnif}
W.~Faidi, A.~Matoussi, and M.~Mni.
\newblock {Robust Utility Maximization Problem with a General Penalty Term}.
\newblock Preprint, 2013.

\bibitem[Heyne et~al.(2014)Heyne, Kupper, and Mainberger]{Hey-Kup-Mai}
G.~Heyne, M.~Kupper, and C.~Mainberger.
\newblock {Minimal Supersolutions of BSDEs with Lower Semicontinuous
  Generators}.
\newblock \emph{Annales de l'Institut Henri Poincare (B) Probability and
  Statistics}, 50\penalty0 (2), 2014.

\bibitem[Horst et~al.(2014)Horst, Hu, Imkeller, R\'eveillac, and
  Zhang]{Hor-etal}
U.~Horst, Y.~Hu, P.~Imkeller, A.~R\'eveillac, and J.~Zhang.
\newblock {Forward Backward Systems for Expected Utility Maximization}.
\newblock \emph{Stochastic Processes and their Applications}, 124\penalty0
  (5):\penalty0 1813--1848, 2014.

\bibitem[Hu et~al.(2005)Hu, Imkeller, and M\"uller]{Hu-Imk-Mul}
Y.~Hu, P.~Imkeller, and M.~M\"uller.
\newblock Utility {M}aximization in {I}ncomplete {M}arkets.
\newblock \emph{Ann. Appl. Probab.}, 15:\penalty0 1691--1712, 2005.

\bibitem[Karatzas and Shreve(1988)]{Kar-Shr}
I.~Karatzas and S.~E. Shreve.
\newblock \emph{Brownian {M}otion and {S}tochastic {C}alculus}.
\newblock Springer, 2 edition, 1988.

\bibitem[Kazamaki(1994)]{Kaz}
N.~Kazamaki.
\newblock \emph{Continuous {E}xponential {M}artingales and {B}{M}{O}}.
\newblock Volume 1579 of Lecture Notes in Mathematics. Springer-Verlag, 1
  edition, 1994.

\bibitem[Kramkov and Schachermayer(1999)]{Kra-Sch}
D.~Kramkov and W.~Schachermayer.
\newblock {The Asymptotic Elasticity of Utility Functions and Optimal
  Investment in Incomplete Market}.
\newblock \emph{Ann. Appl. Probab.}, 9\penalty0 (3):\penalty0 904--950, 1999.

\bibitem[{\O}ksendal and Sulem(2013)]{Oks-Sul13}
B.~{\O}ksendal and A.~Sulem.
\newblock {A Stochastic Control Approach to Robust Duality in Utility
  Maximization }.
\newblock Preprint, 2013.

\bibitem[Peng(1993)]{PengFBSDE}
S.~Peng.
\newblock {Backward Stochastic Differential Equations and Applications to
  Optimal Control}.
\newblock \emph{Applied Mathematics and Optimization}, 27:\penalty0 125--144,
  1993.

\bibitem[Protter(2004)]{Pro}
P.~E. Protter.
\newblock \emph{{Stochastic Integration and Differential Equations}}.
\newblock Springer-Verlag, 2 edition, 2004.

\bibitem[Rockafellar(1966)]{Roc66}
R.~T. Rockafellar.
\newblock {Extension of Fenchel's Duality Theorem for Convex Functions}.
\newblock \emph{Duke Math. J.}, 33:\penalty0 81--89, 1966.

\bibitem[Rockafellar(1970)]{roc70}
R.~T. Rockafellar.
\newblock \emph{Convex Analysis}.
\newblock Princeton University Press, 1970.

\bibitem[Rockafellar and Wets(1998)]{Roc-Wets}
R.~T. Rockafellar and R.~Wets.
\newblock \emph{{Variational Analysis}}.
\newblock Springer, Berlin, New York, 1998.

\bibitem[Savage(1972)]{Savage01}
L.~J. Savage.
\newblock \emph{The foundations of Statistics}.
\newblock Dover Publications, 2 revised edition, 1972.

\bibitem[von Neumann and Morgenstern(1947)]{vonN-Mor}
J.~von Neumann and O.~Morgenstern.
\newblock \emph{Theory of Games and Economics Behavior}.
\newblock Princeton University Press, 2nd edition, 1947.

\end{thebibliography}

\end{document}